\documentclass[11pt, leqno]{article}
\renewcommand{\baselinestretch}{1.18}
\usepackage{graphicx}
\usepackage{cite}
\usepackage{lscape}
\usepackage{indentfirst}
\usepackage{latexsym}
\usepackage{multirow}
\usepackage{tabls}
\usepackage{wrapfig}
\usepackage{longtable}
\usepackage{supertabular}
\usepackage{subfigure}
\usepackage[bookmarks, colorlinks=true, plainpages=false, citecolor=blue, urlcolor=blue, filecolor=blue, linkcolor=blue, pagebackref, hyperindex]{hyperref}

\usepackage{amsmath}
\usepackage{amssymb}
\usepackage{amsthm}
\usepackage[all,cmtip]{xy}
\usepackage[table]{xcolor}
\usepackage{longtable}

\newtheorem{thm}{Theorem}[section]
\newtheorem{lemma}[thm]{Lemma}
\newtheorem{defin}[thm]{Definition}
\newtheorem{prop}[thm]{Proposition}

\newtheorem{rmk}[thm]{Remark}
\newtheorem{example}[thm]{Example}
\newtheorem{cor}[thm]{Corollary}

\newtheorem*{unnumcor}{Corollary}
\newtheorem*{unnumlemma}{Lemma}
\newtheorem*{unnumprop}{Proposition}
\newtheorem*{unnumthm}{Theorem}
\newtheorem*{unnumrmk}{Remark}
\newtheorem*{maintheorem}{Main Theorem}

\newcommand{\resfield}{k_F}
\newcommand{\valfield}{\mathbb{C}}

\newcommand{\kapchi}{\kappa_\chi}

\newcommand{\wbar}{\bar{w}}
\newcommand{\wtrans}{t_{\lambda(w)}}

\newcommand{\torus}{T}
\newcommand{\dualtorus}{\widehat{T}}

\newcommand{\critgrpJ}{A_{w,J,\resfield}}

\newcommand{\algintF}{\mathcal{O}_F}
\newcommand{\algintR}{\mathcal{O}_r}

\newcommand{\relgrp}{S_w}
\newcommand{\relgrpJ}{S_{w,J}}
\newcommand{\dzrelgrp}{S_w^{\rm{dz}}}
\newcommand{\dzrelgrpJ}{S_{w,J}^{\rm{dz}}}

\newcommand{\chars}[1]{X^*(#1)}
\newcommand{\cochars}[1]{X_*(#1)}

\newcommand{\berncent}{\mathfrak{Z}(G)}

\newcommand{\phirone}{\phi_{r,1}}
\newcommand{\phironeaug}{\phi_{r,1}^\prime}

\title{A Combinatorial Formula for Test Functions \\ with Pro-$p$ Iwahori Level Structure}
\author{Marc Horn}
\date{\today}

\begin{document}

\maketitle

\begin{abstract}
The Test Function Conjecture due to Haines and Kottwitz predicts that the geometric Bernstein center is a source of test functions required by the Langlands-Kottwitz method for expressing the local semisimple Hasse-Weil zeta function of a Shimura variety in terms of automorphic L-functions. Haines and Rapoport found an explicit formula for such test functions in the Drinfeld case with pro-$p$ Iwahori level structure.

This article generalizes the Haines-Rapoport formula for the Drinfeld case to a broader class of split groups. The main theorem presents a new formula for test functions with pro-$p$ Iwahori level structure, which can be computed through some combinatorics on Coxeter groups. Explicit descriptions of the test function in certain low-rank general linear and symplectic group examples are included.
\end{abstract}


\section{Introduction}


The Hasse-Weil zeta function of an algebraic variety defined over a number field is an important object of study in modern number theory connected to several guiding problems. A goal of the Langlands Program is to express these zeta functions in terms of automorphic L-functions. Much can be said about the zeta functions in the case of Shimura varieties due to the contributions of many mathematicians, and in particular, the \emph{Langlands-Kottwitz method} outlines a rigorous strategy for studying local factors in the Euler product of a zeta function of a Shimura variety.

Although the Langlands-Kottwitz method and the Test Function Conjecture of Haines and Kottwitz serve as the motivation for this work, very little of the technology involved with that theory will be used in what is to come. For a complete explanation, see the survey article \cite{haines2005}, while the Test Function Conjecture is precisely stated in \cite{haines2014}, Conjecture 4.30.

We will focus instead on a single aspect of a certain identity involving the semisimple trace of Frobenius on the $\ell$-adic cohomology of a Shimura variety, which must be established in the course of following the Langlands-Kottwitz approach. Here is the formula as stated in \cite{haines2014}, Section 6.1:
$$
{\rm{tr}}^{\rm{ss}}\left(\Phi_p^r, H_c^\bullet(Sh_{K_p} \otimes_E \bar{\mathbb{Q}}_p, \bar{\mathbb{Q}}_\ell)\right) = \sum_{(\gamma_0; \gamma, \delta)} c(\gamma_0; \gamma, \delta) {\rm{O}}_\gamma(1_{K^p}) {\rm{TO}}_{\delta \theta} (\phi_r).
$$

For our purposes, we can limit our attention to the term ${\rm{TO}}_{\delta \theta} (\phi_r)$ in the trace formula, which is a twisted orbital integral defined by
$$
{\rm{TO}}_{\delta \theta}(\phi_r) = \int_{G_{\delta \theta}^\circ (F) \setminus G(F_r)} \phi_r \Big(g^{-1} \delta \theta(g)\Big) d\bar{g},
$$
where
\begin{itemize}
\item $G$ is a connected reductive group over a $p$-adic field $F$,
\item $F_r / F$ is a degree $r$ unramified extension,
\item $\theta$ generates the Galois group ${\rm{Gal}}(F_r/F)$,
\item $\delta$ is an element of $G(F_r)$ whose norm in $G(F)$ is semisimple,
\item $G_{\delta \theta} (F) = \{ g \in G(F_r) \mid g^{-1} \delta \theta (g) = \delta\}$, with identity component $G_{\delta \theta}^\circ (F)$,
\item $d\bar{g}$ is a quotient Haar measure, and
\item $\phi_r$ is a locally-constant compactly-supported $K_{p^r}$-biinvariant function on $G(F_r)$.
\end{itemize}
See \cite{haines2012}, Section 6.2, for a complete explanation of ${\rm{TO}}_{\delta \theta} (\phi_r)$.

The function $\phi_r$ is called a \emph{test function}, and it is the focal point of this article. The aforementioned Test Function Conjecture predicts that the Bernstein center of $G$ is a source of test functions that satisfy the above trace formula; however, we do not directly address the Conjecture. Instead, we consider functions defined via the Bernstein center in the case of split connected reductive groups with connected center with pro-$p$ Iwahori level structure, in which case the function is denoted $\phirone$, and then we develop a combinatorial formula for a closely related function $\phironeaug$ whose twisted orbital integrals match those of $\phirone$, that is, ${\rm{TO}}_{\delta\theta} (\phirone) = {\rm{TO}}_{\delta\theta}(\phironeaug).$

Let us conclude this introduction with some motivation for proving explicit formulas for test functions arising from the Bernstein center and an overview of previous work.

This area grew out of an effort to understand nearby cycles of Shimura varieties with parahoric level structure. At first, calculations were done using geometric methods, but eventually conjectures were formulated using functions coming from the Bernstein center. Then it became useful to have explicit versions of the test functions to compare with the geometric calculations. We refer the reader to the survey article \cite{haines2005} for additional background and history.

A test function $\phi_r = q^{r\ell(t_\mu)/2}(Z_{V_\mu} \ast 1_{K_r})$ is defined in terms of a distribution $Z_{V_\mu}$ in the Bernstein center and a level structure group $K_r$, which is a compact open subgroup of $G(F_r)$. As the subgroup $K_r$ changes, so do both the coefficients of the test function $\phi_r$ and its support. For example, suppose $K_r = G(\mathcal{O}_{F_r})$ is a hyperspecial maximal compact subgroup. Then $\phi_r$ is just $1_{K_r \varpi^\mu K_r}$. So the coefficients of the test function are 0 or 1 and the support is the double $K_r$-coset corresponding to $\mu$. If $K_r = I_r$ is an Iwahori subgroup, the situation becomes more complex. Now the test function $\phi_r$ is supported on the \emph{$\mu$-admissible set} and the coefficient of $\phi_r$ at an admissible element $w$ involves the polynomial $R_{w, t_{\lambda(w)}}(q)$ coming from Kazhdan-Lusztig theory~\cite{haines2000b}.

This paper addresses the case where $K_r = I_r^+$ is a pro-$p$ Iwahori subgroup following on the earlier work of Haines and Rapoport, which considered the Drinfeld case for this level structure \cite{haines-rapoport2012}. In this case, the coefficients of $\phi_r$ are far more complicated than in the case of Iwahori level structure and the support is stratified by elements which are products of elements in the $\mu$-admissible set with certain elements in the set of $k_r$-points of a split maximal torus.

Finally, Scholze~\cite{scholze2011} discovered explicit test function formulas in the $GL_2$ case for deeper level structure groups and subsequently opened new directions of research into the Langlands-Kottwitz method~\cite{scholze2013a}.

\emph{Acknowledgements.} This article was the author's Ph.D. thesis at the University of Maryland, College Park. I thank my advisor, Thomas Haines, for his guidance during this period.


\subsection{Summary of this paper}
\label{section::summary-of-thesis}

We highlight key definitions and results, while pointing out the various hypotheses assumed along the way. Background material can be found in Section~\ref{section::preliminaries}.

The group $G$ is a split connected reductive algebraic group with connected center defined over a $p$-adic field $F$, which admits the cases of general linear groups and general symplectic groups. Fixing a choice of Borel subgroup $B$, which we do, in turn determines a split maximal torus $T \subset B$. Now define the Iwahori subgroup $I$ to be the subgroup of $G(\mathcal{O}_F)$ whose reduction modulo $\varpi$ is $B(\resfield)$. Let $\mu$ be a dominant minuscule cocharacter of $\torus$. Given a degree $r$ unramified extension $F_r/F$, the $F_r$-points of $G$ shall be denoted $G_r$.

Let $q$ denote order of the residue field $k_F$, hence the residue field $k_r$ of $F_r$ has order $q^r$. We will often use the difference $Q_r = q^{-r/2} - q^{r/2}$ in what follows.

Our group $G$ has a dual group $\widehat{G}$ defined over $\mathbb{C}$ corresponding to the dual root datum of $G$. There exists a highest-weight representation $(r_\mu, V_\mu)$ of $\widehat{G}$ determined by our chosen $\mu$. By the theory of the stable Bernstein center $\mathfrak{Z}^{\rm{st}}(G)$, there is an element $Z_{V_\mu}$ in $\mathfrak{Z}^{\rm{st}}(G)$ that maps an infinitesimal character $(\lambda)_{\widehat{G}}$ on the Weil group $W_F$ to the semisimple trace of Frobenius on $V_\mu$ (Proposition~\ref{prop::reg-func-semisimple-trace}). Assuming the LLC+ conjecture, described in \cite{haines2014} Section 5.2, the distribution $Z_{V_\mu}$ can be viewed as an element of the usual Bernstein center $\mathfrak{Z}(G)$. All of this is tied together in Definition~\ref{defin::test-function}, which is stated for a general test function. The discussion at the start of Section~\ref{section::depth-zero-characters} specializes that definition to the case where the level structure group is the pro-unipotent radical $I_r^+$ of an Iwahori subgroup $I_r$ of $G_r$. So we come to consider the test function
$$
\phirone = q^{r \ell(t_\mu)/2} \left(Z_{V_\mu} \ast 1_{I_r^+} \right).
$$

The function $\phirone$ lies in the center of the Hecke algebra $\mathcal{H}(G_r, I_r^+)$. This algebra is related to Hecke algebras $\mathcal{H}(G_r, I_r, \rho_{\chi_r})$, each of which is determined by a depth-zero character $\chi_r$ on $\torus(\mathcal{O}_r)$ obtained by composing a depth-zero character $\chi : \torus(\mathcal{O}_F) \rightarrow \valfield^\times$ with the norm $N_r : \torus(\mathcal{O}_r) \rightarrow \torus(\mathcal{O}_F)$. Because $\torus(k_r) \cong I_r / I_r^+$, by Proposition~\ref{prop::iwahori-prop-quotient}, the character $\chi_r$ can be extended to a character $\rho_{\chi_r}$ on the Iwahori subgroup $I_r$ that is trivial on $I_r^+$. Section~\ref{section::depth-zero-characters} is devoted to objects and results, such as these, associated to depth-zero characters.

Definition~\ref{endoscopic-element-definition} builds on the LLC for Tori to associate an ``endoscopic element'' $\kappa_\chi$ in $\dualtorus(\valfield)$ to each depth-zero character $\chi$ on $\torus(\mathcal{O}_F)$. Proposition~\ref{prop::dz-endoscopic-elements-equal-kernel} characterizes these endoscopic elements as the kernel $K_{q-1}$ of the endomorphism on $\dualtorus(\valfield)$ given by $\kappa \mapsto \kappa^{q-1}$.

Section~\ref{section::first-formula} takes advantage of the work with depth-zero characters to prove
$$
\phirone = [I_r : I_r^+]^{-1} q^{r\ell(t_\mu)/2} \sum_{\xi \in \torus(k_r)^\vee} Z_{V_\mu} \ast e_\xi,
$$
where $e_\xi$ is an idempotent in the Hecke algebra $\mathcal{H}(G)$ and $\xi$ is a depth-zero character on $\torus(\mathcal{O}_r)$. It turns out that we can ignore certain terms in this sum when viewing $\phirone$ as a test function to be plugged into a twisted orbital integral. This is Lemma~\ref{lemma::orbital-intergrals-zero}:

\begin{unnumlemma}{\rm{(Haines)}}
Suppose $\xi \in \torus(k_r)^\vee$ is not a norm, that is, there is not a $\chi \in \torus(\resfield)^\vee$ such that $\xi = \chi \circ N_r$. Then all twisted orbital integrals at $\theta$-semisimple elements vanish on functions in $\mathcal{H}(G_r, I_r, \rho_\xi)$.
\end{unnumlemma}

In light of this lemma, we define a new function
$$
\phironeaug = [I_r : I_r^+]^{-1} q^{r\ell(t_\mu)/2} \sum_{\chi \in \torus(k_F)^\vee} Z_{V_\mu} \ast e_{\chi_r}
$$
whose twisted orbital integrals satisfy ${\rm{TO}}_{\delta\theta}(\phirone) = {\rm{TO}}_{\delta\theta}(\phironeaug)$. \textbf{As we shall see, the main theorem is a combinatorial formula for $\phironeaug$ rather than $\phirone$.} But because the twisted orbital integrals of these functions agree, this formula may as well be a formula for the test function.

\begin{unnumrmk}
It is possible to give a definition of $\phi_{r,1}^\prime$ that does not invoke LLC+ by using the LLC for Tori. Instead of using the distribution $Z_{V_\mu}$ to define functions $Z_{V_\mu} \ast e_{\chi_r}$ in the center of $\mathcal{H}(G_r, I_r, \rho_{\chi_r})$, we can define $Z_{V_\mu} \ast e_{\chi_r}$ to be the function in the center of that Hecke algebra which acts by semisimple trace of Frobenius on the Bernstein block of $\tilde{\chi}_r$. See also Remark~\ref{rmk::clarify-dependence-on-llc+}.
\end{unnumrmk}

Roche's theory of Hecke algebra isomorphisms shows us how to rewrite a function in the center of $\mathcal{H}(G_r, I_r, \rho_{\chi_r})$ as a sum of \emph{Bernstein functions} in the center an Iwahori-Hecke algebra associated to an endoscopic group $H_{\chi_r}$; however, we must make some assumptions about $G$ in order to apply this theory without making restrictions to ${\rm{char}}(\resfield)$. These are explained in Remark~\ref{rmk::roche-assumptions}. Haines's formula for Bernstein functions attached to dominant minuscule cocharacters leads to a more concrete formula for $\phironeaug$ by introducing Kazhdan-Lusztig $\widetilde{R}$-polynomials to the expression. The end result of Chapter 2, Proposition~\ref{prop::first-explicit-formula}, is an explicit formula for the coefficients $\phironeaug(I_r^+swI_r^+)$, with $\gamma_{N_r s}$ as in Lemma~\ref{lemma::defin-of-gamma-nrs}:

\begin{unnumprop}
Given a pair $(s,w) \in \torus(k_r) \times \widetilde{W}$, the coefficient $\phironeaug (I_r^+ sw I_r^+)$ can be rewritten as a sum over endoscopic elements in $\dualtorus(\valfield)$ which arise from depth-zero characters $\chi \in \torus(\resfield)^\vee$:
$$
\phironeaug (I_r^+ sw I_r^+) = [I_r: I_r^+]^{-1} \sum_{\kapchi \in K_{q-1}} \gamma_{N_r s}(\kapchi)^{-1} q^{r\ell(w,t_{\lambda(w)})/2} \widetilde{R}_{w,t_{\lambda(w)}}^\chi(Q_r).
$$
\end{unnumprop}

The formula of Proposition~\ref{prop::first-explicit-formula} defines an element of the Hecke algebra $\mathcal{H}(G_r, I_r^+)$, and so there is no doubt that the function exists, subject to the various hypotheses in place. On the other hand, the purpose for considering this function involves the conjectural existence of a distribution $Z_{V_\mu}$ in the stable Bernstein center. When $G=GL_n$, this distribution is known to exist and embeds into $\mathfrak{Z}(G)$. Hence in at least one important example of a split connected reductive group with connected center, the function $\phirone$ appearing in the Test Function Conjecture exists, and its twisted orbital integrals agree with those of the function $\phironeaug$, whose coefficients are specified by the formula in Proposition~\ref{prop::first-explicit-formula}.

Chapter 3 begins the process of simplifying this formula. The first simplification comes from studying the set of endoscopic elements $\kapchi$ such that the functions $Z_{V_\mu} \ast e_{\chi_r}(w) \neq 0$ for a fixed $w \in \widetilde{W}$. An element $\kappa$ in $\dualtorus(\valfield)$ is ``relevant'' to $w = t_\lambda \bar{w}$ if $\lambda(\kappa) = 1$ and $\bar{w}\kappa = \kappa$. Then we prove in Proposition~\ref{prop::relevant-group}:

\begin{unnumprop}
The elements $\kappa \in \dualtorus(\valfield)$ relevant to a fixed $w \in \widetilde{W}$ form a closed subgroup called the $\textbf{relevant subgroup}$ $S_w$.
\end{unnumprop}

We subsequently define a subgroup $S_{w,J} \subseteq S_w$ for any root sub-system $J$ of the ambient system $\Phi(G,T)$ in Definition~\ref{defin::dzrelgrpJ}. The $S_{w,J}$ are diagonalizable algebraic subgroups of $\dualtorus$ defined over $\mathbb{C}$, hence it is natural to consider their character groups $\chars{S_{w,J}}$. Section~\ref{section::lattices-in-char-groups} culminates in the definition of a lattice $L_{w,J} \subset \chars{\dualtorus}$ such that $\chars{S_{w,J}} = \chars{\dualtorus}/L_{w,J}$.

Whereas the groups $S_{w,J}$ are (infinite) algebraic groups, the groups needed for the formula are finite subgroups $\dzrelgrpJ = S_{w,J} \cap K_{q-1}$ of $\dualtorus(\valfield)$. In Section~\ref{section::finite-critical-grps}, we define a finite group $A_{w,J, k_F} \subset \torus(k_F)$ using the lattice $L_{w,J}$. Together, this data describes what happens to a certain sum that will appear in the proof of the main theorem. This is Proposition~\ref{prop::sum-over-group}:

\begin{unnumprop}
Let $s \in \torus(k_r)$, and define $\gamma_{N_r s}$ as above. Then
$$\sum_{\kapchi \in \dzrelgrpJ} \gamma_{N_r s}(\kapchi)^{-1} = \begin{cases}
0, & {\rm{if}}\ N_r(s) \notin A_{w, J, k_F} \\
\vert \dzrelgrpJ \vert, & {\rm{otherwise}}.
\end{cases}
$$
\end{unnumprop}

The second simplification of the formula for $\phironeaug$ comes from the theory of the $\widetilde{R}$-polynomials, defined by Kazhdan and Lusztig, using a formula for these polynomials due to Dyer based on the Bruhat graph and reflection orderings. This is the content of Chapter 4.  Section~\ref{section::reflection-orderings} defines the notion of a reflection ordering $\prec$ on the positive roots of the root system $\Phi$ associated to the Weyl group $W$. The set of paths $u\stackrel{\Delta}{\longrightarrow} v$, for $u$ and $v$ in $W$, through the Bruhat graph whose edges are increasing with respect to $\prec$ is denoted $B_\Phi^\prec(u,v)$. This information determines the $\widetilde{R}$-polynomials according to Theorem~\ref{dyer-R-polynomial-formula}:

\begin{unnumthm}{\rm{(Dyer)}}
Let $\widetilde{W}_{\chi}$ be the extended affine Weyl group of $H_{\chi_r}$, and let $\prec$ be a reflection ordering on $W_{\chi, \rm{aff}}$. Let $Q_r = q^{-r/2} - q^{r/2}$. For any $u, v \in \widetilde{W}_{\chi}$ such that $u \leq_\chi v$ in Bruhat order,
$$\widetilde{R}^\chi_{u,v} (Q_r) = \sum_{\Delta \in B_{\Phi_{\chi, \rm{aff}}}^\prec(u, v)} Q_r^{\ell(\Delta)}.$$
\end{unnumthm}

The main objective of Chapter 4 is to rewrite this formula for use in the proof of the main theorem. Suppose we start with an element $w$ in the extended affine Weyl group $\widetilde{W}$. Such an element has the form $w = t_\lambda \bar{w}$ where $t_\lambda$ is the translation element for the coweight $\lambda$ and $\bar{w}$ is an element of the finite Weyl group. When we want to emphasize that a translation element is the ``translation part'' of some $w \in \widetilde{W}$, we write $t_{\lambda(w)}$. Now suppose further that $w$ is $\mu$-admissible: Haines and Pettet~\cite{haines-pettet2002} showed that such elements satisfy $w \leq t_{\lambda(w)}$ in the Bruhat order on $\widetilde{W}$. Proposition~\ref{prop::finite-reflections-in-interval} shows that for any $\prec$-increasing path from $w$ to $t_{\lambda(w)}$, all edges of the path correspond to \emph{finite} reflections:

\begin{unnumprop}
Let $\mu$ be a dominant minuscule coweight of $\Phi$, and let $(W,S)$ be the finite Weyl group of $\Phi$ inside the affine Weyl group $(W_{\rm aff}, S_{\rm aff})$. Let $T$ be the set of reflections in $W$.

Consider a $\mu$-admissible element $w \leq t_{\lambda (w)}$.  There exists a length-zero element $\sigma$ in $\widetilde{W}$ such that $w, t_{\lambda(w)} \in \sigma W_{\rm{aff}}$. Let $w \stackrel{\Delta}\longrightarrow t_{\lambda (w)}$ be any path in the Bruhat graph $\Omega_{(W_{\rm{aff}}, S_{\rm{aff}})}$. Each reflection in the edge set $E(\Delta) = \{t_1, \ldots t_n\}$ belongs to $T$.
\end{unnumprop}

As a consequence of this proposition, we see that $B_{\Phi(G,T)_{\rm{aff}}}^\prec(w, t_{\lambda(w)})$ has a $\prec$-preserving bijection to a certain interval $B_{\Phi(G,T)}^\prec(w_\lambda^{-1}\wbar, w_\lambda^{-1})$ whose members are paths with edges only in the \emph{finite} Weyl group. See Proposition~\ref{bruhat-interval-isomorphism}. Moreover, for each path $\Delta \in B_{\Phi(G,T)}^\prec(w, t_{\lambda(w)})$, Lemma~\ref{path-root-system-lemma} shows how to construct a root system $J_\Delta \subset \Phi(G,T)$.

All of this comes together in a ``stratified'' version of Dyer's formula, presented in Corollary~\ref{cor::finite-interval-r-polynomial}, which applies to the polynomials $\widetilde{R}_{w, t_{\lambda(w)}}^J(Q_r)$, defined in Chapter~3 (see discussion following Lemma~\ref{lemma::equality-on-strata}):

\begin{unnumcor}
Let $w = t_\lambda \bar{w} \in {\rm{Adm}}_{G_r}(\mu)$ and $J \subseteq \Phi$. Then
$$
\widetilde{R}_{w,t_\lambda}^J (Q_r) = \sum_{J^\prime \subseteq J}\; \sum_{\substack{\Delta \in B_\Phi^\prec (w_\lambda^{-1}\wbar, w_\lambda^{-1})\\ J_\Delta = J^\prime}} Q_r^{\ell(\Delta)}.
$$
\end{unnumcor}

Finally, in Chapter 5 we come to the main result. Several lemmas employ the results from Chapters 3 and 4 to rewrite our original formula for $\phironeaug(I_r^+ sw I_r^+)$. The proof of the theorem requires the following assumptions on $G$ (see Remark~\ref{rmk::roche-assumptions}),
\begin{enumerate}
\item $G$ is a split connected reductive group with connected center,
\item The derived group $G_{\rm{der}}$ is simply-connected, and
\item $W_\chi = W_\chi^\circ$.
\end{enumerate}
For $w \in \widetilde{W}$, $s \in \torus(k_r)$ and $J \subseteq \Phi$, define a symbol $\delta (s, w, J)$ by
$$
\delta(s, w, J) = \begin{cases}
\ 0, & {\rm{if}}\ w\notin {\rm{Adm}}_{G_r}(\mu) \\
\ 0, & {\rm{if}}\ w \in {\rm{Adm}}_{G_r}(\mu)\ {\rm{and}}\ N_r(s) \notin A_{w, J, k_F} \\
\ 1, & {\rm{if}}\ w \in {\rm{Adm}}_{G_r}(\mu)\ {\rm{and}}\ N_r(s) \in A_{w, J, k_F}.
\end{cases}
$$
Let $S_{w, J_\Delta}^{\rm{tors}}$ be the torsion subgroup of $S_{w, J_\Delta}$. Then Theorem~\ref{main-theorem} ties everything together:

\begin{maintheorem}
Let $w \in \widetilde{W}$ and $s \in \torus(k_r)$. Let $d$ be the rank of $T$. Fix a reflection ordering $\prec$ on $\Phi$, and set $c(\Delta) = \left[\ell(w, t_\mu) - \ell(\Delta)\right]/2$. The coefficient of $\phironeaug$ for the $I_r^+$-double coset of $(s,w)$ is given by
$$
(-1)^d \sum_{\Delta \in B_\Phi^{\prec} (w_\lambda^{-1}\wbar, w_\lambda^{-1})} \delta(s, w,J_\Delta) \vert S_{w,J_\Delta}^{\rm{tors}} \cap K_{q-1}\vert (q-1)^{d-{\rm{rank}}(J_\Delta) -1} q^{r c(\Delta)} (1-q^r)^{\ell(\Delta)-d}.
$$
\end{maintheorem}

Corollary~\ref{cor::drinfeld-case} shows how to recover the formula for the Drinfeld case obtained by Haines and Rapoport as a special case of the Main Theorem. Section~\ref{section::implementation-remarks} follows this with remarks on using the formula for calculations and a description of how the formula can be implemented in software. We conclude with two sections that discuss features of some data gathered by computer for cases where $G$ is a general linear group or a general symplectic group. Tables located in the Appendix present the data in full.


\subsection{Preliminaries}
\label{section::preliminaries}

We conclude this Introduction by introducing some terms and notation concerning reductive algebraic groups defined over non-archimedean local fields. As the main purpose of this section is to introduce notation and supplementary facts, almost all of the details are left out, but references are given.

Several unrelated concepts will be represented by similar symbols in this paper, for example, we use $W$ for a finite Weyl group and $W_F$ for a Weil group of a local field $F$; we use $T$ for both a split maximal torus of $G$ and the set of reflections in a Weyl group; and an Iwahori subgroup of an algebraic group $G$ defined over $F$ is denoted $I$, while the inertia subgroup of a Galois group $\rm{Gal}(\bar{F}/F)$ is denoted $I_F$.

\subsubsection{$p$-adic fields}

The following material on $p$-adic fields and related ideas has been drawn from the book by Serre~\cite{serre1979} and the article by Tate~\cite{tate1979} found in the Corvallis proceedings.

The symbol $F$ shall always refer to a non-archimedean local field, which is sometimes also referred to as a \textbf{$p$-adic field}. Let $\mathcal{O}_F$ and $k_F$ denote its ring of integers and its residue field, respectively. The cardinality of $k_F$ shall be denoted $q$. Fix an algebraic closure $\bar{F} \supset F$. The \textbf{Galois group} ${\rm{Gal}}(\bar{F}/F)$ is the profinite topological group of automorphisms of $\bar{F}$ which fix $F$.

Let $\varpi$ be a uniformizer of $\mathcal{O}_F$. Given an element $x = u\varpi^n$, for $u$ a unit, define ${\rm{val}}(x) = n$. A standard convention is to set ${\rm{val}}(0) = \infty$. Thus we get a map ${\rm{val}}: \mathcal{O}_F \rightarrow \{\mathbb{N}, \infty\}$.

For each $r \in \mathbb{N}$, there exists a degree $r$ \textbf{unramified extension} $F_r \supset F$ (see~\cite{serre1979}, III.5), which is unique inside $\bar{F}$. Given such an extension, the algebraic integers inside $F_r$ are denoted $\mathcal{O}_r$ and its residue field is written $k_r$. Its Galois group ${\rm{Gal}}(F_r/F)$ is isomorphic to $\mathbb{Z}/r\mathbb{Z}$. Define the \textbf{norm map} $N_r : F_r \rightarrow F$ by
$$
N_r (z) = \prod_{i = 0}^{r-1} \theta^i(z),
$$
where $\theta$ is a generator of the cyclic group ${\rm{Gal}}(F_r/F)$.

The following definition of a \textbf{Weil group} is drawn from \cite{tate1979}, Section 1.4.1. A Weil group for $F$ is a group $W_F$ embedded in ${\rm{Gal}}(\bar{F}/F)$ whose closure is the Galois group itself. Let $\hat{k} = \cup_{E/F} k_E$ be the union of all residue fields for finite extensions $E/F$. Then $W_F$ consists of the elements which act by Frobenius on $\hat{k}$, i.e., $x \mapsto x^{q^n}$, for $x \in \hat{k}$ and some $n \in \mathbb{Z}$. The Weil group is generated by a \textbf{geometric Frobenius element} $\Phi_F$ and fits into a short exact sequence
$$
1 \longrightarrow I_F \longrightarrow W_F \longrightarrow \mathbb{Z} \longrightarrow 1,
$$
where $I_F$ is the \textbf{inertia subgroup}. It is a fact that a Weil group exists for all $p$-adic fields, and moreover, this group is unique up to isomorphism. One of the consequences of Local Class Field Theory (see \cite{serre1979}, XIII) is the existence of a \textbf{reciprocity map} $\tau_F : W_F \rightarrow F^\ast$. For any extension $E/F$, there is an isomorphism $r_E: E^\ast \rightarrow W_E^{\rm{ab}}$.

Tate goes on to define the \textbf{Weil-Deligne group} $W_F^\prime$ associated to $W_F$ in~\cite{tate1979}, Definition 4.1.1. The representation theory of Weil-Deligne groups is a major feature of the Local Langlands Conjecture, but we will not need the details in what follows.

\subsubsection{Root systems}

We briefly introduce the notions of a root system and its Weyl group, the latter of which is a type of Coxeter group. Much more will be said about Coxeter groups in Section~\ref{section::coxeter-group-background}. There are many excellent references for this subject, such as Humphreys~\cite{humphreys1990}, Bourbaki~\cite{bourbaki4-6}, and the more recent book by Bj{\"o}rner and Brenti\cite{bjorner-brenti2005}, which focuses on combinatorics.

Let $V$ be a Euclidean space endowed with an inner product $\langle\ ,\ \rangle$. Given a vector $\alpha \in V$, define the reflection $s_\alpha$ with respect to the hyperplane in $V$ orthogonal to $\alpha$. That is,
$$
s_\alpha (x) = x - \frac{2\langle x, \alpha\rangle}{\langle \alpha,\alpha\rangle}\alpha
$$
A \textbf{root system} $\Phi$ is a finite set of vectors $\alpha \in V$ such that $\Phi \cap \mathbb{R}\alpha = \{\alpha, -\alpha\}$ and $s_\alpha (\Phi) = \Phi$. A root system can be partitioned into a disjoint union of positive and negative roots, written $\Phi = \Phi^+ \cup \Phi^-$, by choosing a basis $\Delta \subset \Phi$ whose corresponding reflections generate the group $W = \langle s_\alpha \mid \alpha \in \Phi \rangle$, which is called the \textbf{Weyl group} of $\Phi$. All Coxeter groups admit a partial ordering called the Bruhat order. Furthermore, Coxeter groups have a length function $\ell$. A difference of lengths $\ell(v) - \ell(u)$ is sometimes written $\ell(u,v)$ as in~\cite{bjorner-brenti2005}.

Given a root $\alpha \in \Phi$, define its \textbf{coroot} $\alpha^\vee$ by
$$
\alpha^\vee = \frac{2\alpha}{\langle \alpha, \alpha\rangle}.
$$
The set of such coroots forms the \textbf{dual root system} $\Phi^\vee$.

A \textbf{weight} of $\Phi$ is a vector $\lambda \in V$ such that $\langle \lambda, \alpha^\vee\rangle \in \mathbb{Z}$ for all $\alpha \in \Phi$. The set of weights forms a lattice, which we sometimes denote $X$. A $\textbf{coweight}$ of $\Phi$ is $\eta \in V$ such that $\langle\eta, \alpha\rangle \in \mathbb{Z}$ for all $\alpha \in \Phi$; these too form a lattice, sometimes denoted $Y$.

Given an irreducible root system $\Phi$, there is an associated \textbf{affine root system} $\Phi_{\rm{aff}}$ which has a corresponding Coxeter group $W_{\rm{aff}}$ called the \textbf{affine Weyl group}. Suppose $(W,S)$ is the (finite) Weyl group of $\Phi$. Let $\alpha_0$ denote the highest root of $\Phi$, and let $s_0 = t_{\alpha_0^\vee} s_{\alpha_0}$. Then the set of Coxeter generators $S_{\rm{aff}}$ of the affine Weyl group is the union of $S \cup \{s_0\}$.

The affine Weyl group can be further enlarged to the \textbf{extended affine Weyl group} $\widetilde{W}$, defined as the semidirect product $\widetilde{W} = Y \rtimes W$. See Macdonald's book~\cite{macdonald2003}, Section 2.1, for additional details.

Reflections define orthogonal hyperplanes in $V$ which chop the space up into \textbf{alcoves}, as described in \cite{bourbaki4-6}, Chapter 5. A choice of basis for $\Phi$ determines a Weyl chamber of $V$, and we choose the \textbf{fundamental alcove} $\mathcal{C}$ as the alcove in this chamber whose closure contains the origin. In the situation laid out in Section~\ref{section::summary-of-thesis}, consider the apartment corresponding to $T$ in the Bruhat-Tits building of $G$. (See \cite{tits1979} for more information about buildings. ) Our choice of Borel subgroup $B \supset T$  determines a basis of the root system $\Phi(G, T)$, and $\mathcal{C}$ is the unique alcove in the $B$-positive Weyl chamber inside the apartment whose closure contains the origin. The affine Weyl group is generated by reflections through the walls of $\mathcal{C}$. Let $\Omega[\mathcal{C}]$ denote the subgroup of $\widetilde{W}$ which stabilizes $\mathcal{C}$. Then we get a second realization of $\widetilde{W}$ as the semidirect product $W_{\rm{aff}} \rtimes \Omega[\mathcal{C}]$.

\subsubsection{Reductive algebraic groups over $p$-adic fields}

So far we have encountered Galois and Weil groups, along with several variants of Weyl groups. We conclude these preliminaries by giving some definitions pertaining to linear algebraic groups, i.e., Zariski closed subgroups of a general linear group viewed as an algebraic variety, focusing on the case of (split) reductive groups defined over local fields. For the general theory of linear algebraic groups, see for example the books by Borel~\cite{borel1991} and Humphreys~\cite{humphreys1975}. The structure theory of reductive groups over local fields is highly developed; see the survey by Tits~\cite{tits1979} as a starting point.

Let $G$ denote a connected reductive algebraic group that is split over $F$. Its group of $F$-rational points, denoted $G(F)$, has a neighborhood basis of compact open subgroups; and moreover, $G(F)$ is unimodular, hence we may speak of a choice of Haar measure on $G(F)$.

As in Section~\ref{section::summary-of-thesis}, we fix a Borel subgroup $B$ and let $T$ denote the split maximal torus inside $B$. The pair $(G,T)$ determines a root system $\Phi = \Phi(G,T)$ whose positive roots are denoted $\Phi^+$. Let $U$ denote the unipotent radical of $B$. Then
$$
B = TU = T \prod_{\alpha \in \Phi^+} U_\alpha,
$$
where $U_\alpha \subset G(\mathcal{O}_F)$ is normalized by $T$; see \cite{tits1979} Section 1. The \textbf{Iwahori subgroup} with respect to this configuration is the subgroup of $G(\mathcal{O}_F)$ that maps onto $B(\resfield)$ via the ``mod $\varpi$'' map. The unique pro-unipotent subgroup of $I$ is called its \textbf{pro-unipotent radical} $I^+$. The subgroup $I^+$ is a pro-$p$ group.

The \textbf{character group} of $\torus$ is $\chars{\torus} = {\rm{Hom}}_F(\torus, \mathbb{G}_m)$. This group can be thought of as the weight lattice of $\Phi(G,T)$. On the other hand, the \textbf{cocharacter group} $\cochars{\torus} = {\rm{Hom}}_F(\mathbb{G}_m, \torus)$ is the coweight lattice of $\Phi(G,T)$. A cocharacter $\lambda \in X_* (T)$ is \emph{dominant} if $\langle \lambda, \alpha \rangle \ge 0$ for all $\alpha \in \Phi^+$, and it is \emph{minuscule} if $\langle \lambda, \alpha \rangle \in \{ -1, 0, 1\}$. (See also~\cite{bourbaki4-6}, VI.1 Exercise 24.) Throughout this paper, the letter $\mu$ will be reserved for a dominant, minuscule cocharacter in $\cochars{\torus}$ with respect to $\Phi(G,T)$. The set of dominant cocharacters is written $\cochars{\torus}_{\rm{dom}}$.

Our torus $T$ determines a unique \textbf{dual torus} $\dualtorus$ defined over $\mathbb{C}$ whose character group $\chars{\dualtorus}$ is the free abelian group $\cochars{T}$.

Taking the view of a split connected algebraic group $G$ over $F$ as an $\mathcal{O}_F$-affine group scheme, we obtain data related to the above for each unramified extension $F_r/F$. We write $G_r$ for $G(F_r)$, $I_r$ for the corresponding Iwahori subgroup, etc.

Given an unramified extension $F_r/F$, there is a norm $N_r : F_r \rightarrow F$ as described above. The Galois action of ${\rm{Gal}}(F_r/F)$ on $F_r$ can be extended to an action of the group on $\torus(\mathcal{O}_r)$. The resulting map is also written $N_r : \torus(\mathcal{O}_r) \rightarrow \torus(\mathcal{O}_F)$. Observe that for each character $\xi : \torus(\mathcal{O}_F) \rightarrow \valfield^\times$, we may use the norm to get a new character $\xi_r = \xi \circ N_r$. This is an important type of character on $T(\mathcal{O}_r)$.

\section{Test functions with pro-$p$ Iwahori level structure}
\label{chapter::test-functions}

As stated in the Introduction, our main result is a new formula for the coefficients of a function $\phironeaug$ whose twisted orbital integrals agree with those of the test function $\phirone$ with pro-$p$ Iwahori level structure, at least in the cases of general linear groups and general symplectic groups. This chapter summarizes the results needed to define test functions as they arise from the Bernstein center, before going on to develop a first explicit formula for $\phironeaug$ using various data about depth-zero characters. Subsequent chapters will translate this version of the formula into something based on the combinatorics of Coxeter groups.



\subsection{Background on representation theory}
\label{section::background-rep-theory}

This section collects definitions and results concerning smooth representations of reductive algebraic groups defined over a non-archimedean local field. We also give some background on  representations of Weil groups, which involves defining the Langlands dual group of $G$, in order to explain the Local Langlands Correspondence.

For reference, we recommend the Fields Institute book~\cite{cunningham-nevins2009}, which introduces smooth representations of $p$-adic groups, representations of Weil groups, and the Local Langlands Correspondence in a single volume.

\subsubsection{Smooth representations of $p$-adic groups}

In this section, let $G$ be a connected reductive algebraic group defined over a $p$-adic field $F$. Let $V$ be a complex vector space. 

\begin{defin}
A \textbf{smooth representation} of $G(F)$ is a homomorphism \\ 
${\pi: G(F) \rightarrow Aut(V)}$
such that for every $v \in V$ there exists a compact open subgroup $K \subset G(F)$ such that $v \in V^K$, where
$$
V^K = \{v \in V \mid \pi(k)\cdot v = v, \forall k \in K\}.
$$
The category of smooth representations of $G$ is denoted $\mathfrak{R}(G)$.
\end{defin}

Let $C_c^\infty (G)$ denote the space of locally constant, compactly supported functions $f : G(F) \rightarrow \mathbb{C}$. If $K$ is a compact open subgroup of $G(F)$, let $\mathcal{H}(G,K)$ denote the $K$-biinvariant functions in $C_c^\infty$, that is, $f \in \mathcal{H}(G,K)$ satisfies
$$
f(k_1 x k_2) = f(x),\ {\rm{for}}\ k_1, k_2 \in K\ {\rm{and}}\ x \in G.
$$
The \textbf{Hecke algebra} $\mathcal{H}(G)$ is the union $\cup_K \mathcal{H}(G,K)$ ranging over all compact open subgroups $K$ of $G(F)$, equipped with a convolution integral:
$$
(f \ast h)(x) = \int_G f(g)h(g^{-1}x) dg,
$$
where $dx$ is a fixed normalization of Haar measure. It is well-known that $\mathfrak{R}(G)$ is equivalent to the category of non-degenerate left $\mathcal{H}(G)$-modules.

\begin{defin}
Let $G$ be a connected reductive group with maximal torus $\torus$, and let $I$ be the Iwahori subgroup as specified above. The \textbf{Iwahori-Hecke algebra} $\mathcal{H}(G, I)$ is the set of functions $f \in C_c^\infty(G)$ which are invariant on $I - I$-double cosets, with an algebra structure given by convolution.
\end{defin}

Let $P = MN$ be a Levi decomposition of a parabolic subgroup $P$ of $G(F)$. Given a representation $(\sigma,V)$ of $M$, we have an \textbf{induced representation} $\rm{Ind}_P^G (\sigma)$ in $\mathfrak{R}(G)$. The representation space of $\rm{Ind}_P^G (\sigma)$ is
$$
\{f: G \rightarrow V \mid f(hg) = \sigma(h) f(g),\: \forall h \in P, g \in G\}
 $$
The case where $P$ is a Borel subgroup, in which case it is denoted $B$, is particularly important. Here $B = TU$ for a maximal torus $\torus$, and so forming induced representations is a method for producing smooth representations of $G$ from characters on a torus. In fact, we generally consider \emph{normalized} induced representations, where $\sigma$ is twisted by a square-root of the modulus function $\delta_P$. We write $i_P^G (\sigma) = \rm{Ind}_P^G(\delta_P^{1/2}\sigma)$.

The category $\mathfrak{R}(G)$ can be further understood in terms of induced representations arising from supercuspidal representations of Levi subgroups. (This is part of Harish-Chandra's philosophy of cusp forms.) Much of the following terminology comes from the theory of Bushnell-Kutzko types~\cite{bushnell-kutzko1997}, though our presentation mostly follows~\cite{roche2009}, Section 1.7.

Let $M$ be a Levi subgroup of $G$, and let $\sigma$ be a supercuspidal representation of $M$. The pair $(M, \sigma)$ is called a \textbf{cuspidal pair}. There is a conjugacy relation on cuspidal pairs: Given $g \in G$, let $L = gMg^{-1}$ and define ${^g}\sigma$ by ${^g}\sigma(x) = \sigma(g^{-1}x g)$; then the resulting cuspidal pair $(L, {^g}\sigma)$ belongs to $(M, \sigma)_G$, the $G$-conjugacy class of $(M, \sigma)$. It turns out to be more useful to consider the coarser \textbf{inertial equivalence} relation: The pairs $(M, \sigma)$ and $(L, \tau)$ are equivalent if there exists $g \in G$ such that $L = gMg^{-1}$ and $\tau \cong {^g}\sigma \otimes \eta$ for some unramified character $\eta$ on the group $L$. Let $\mathfrak{s} = [M, \sigma]_G$ denote an inertial equivalence class, and call the set of all inertial equivalence classes for $G$, denoted $\mathfrak{B}(G)$, the \textbf{Bernstein spectrum}.

Each inertial equivalence class $\mathfrak{s}$ gives rise to a full subcategory $\mathfrak{R}_{\mathfrak{s}} (G)$ of $\mathfrak{R}(G)$. The objects are described in terms of subquotients of induced representations. Specifically, let $\Pi$ be a smooth representation of $G$. Then $\Pi \in \mathfrak{R}_{\mathfrak{s}} (G)$ if and only if every irreducible subquotient $\pi$ of $\Pi$ has inertial support equal to $\mathfrak{s}$, i.e., if there is a cuspidal pair $(M, \sigma) \in \mathfrak{s}$ such that $\pi$ is a subquotient of $i_P^G(\sigma\eta)$ for $P = MN$ a Levi decomposition and $\eta$ an unramified character on $M$.

\begin{thm} \emph{(Bernstein Decomposition)}
The category of smooth representations of $G$ decomposes as
$$
\mathfrak{R}(G) = \prod_{\mathfrak{s} \in \mathfrak{B}(G)} \mathfrak{R}_{\mathfrak{s}} (G).
$$
\end{thm}

\begin{proof}
This result is originally due to Bernstein. See also~\cite{roche2009}, Theorem 1.7.3.1.
\end{proof}

\subsubsection{The Local Langlands Correspondence}
\label{section::LLC}

The primary reference for this section is Borel's article on automorphic L-functions~\cite{borel1979} from the Corvallis proceedings.

The \textbf{L-group} of a connected reductive group $G$ is ${^L}G = \widehat{G} \rtimes W_F$, where $\widehat{G}$ is the connected reductive group defined over $\mathbb{C}$ determined by the dual root datum $(\cochars{\torus}, \Phi^\vee, \chars{\torus}, \Phi)$. When $G$ is split, the action of the Weil group on $\widehat{G}$ is trivial; so in this case we may use $\widehat{G}$ in place of ${^L}G$.

In the representation theory of the Weil-Deligne group $W_F^\prime$, L-groups play the role of automorphisms of a linear space, that is, we consider homomorphisms \\
$\varphi: W_F^\prime \rightarrow  {^L}G$.  More specifically, we consider ``admissible'' homomorphisms in the sense specified in \cite{haines2014}, Section 4. Following the discussion in \cite{haines2014} Section 5.1, we restrict an admissible homomorphism $\varphi$ on $W_F^\prime$ along the proper embedding $W_F \hookrightarrow W_F^\prime$ to get an admissible homomorphism $\lambda$ on the Weil group, where here ``admissible'' is in the sense of the footnote on p. 131 of \cite{haines2014}. The $\widehat{G}$-conjugacy class of an admissible homomorphism $\lambda$ on $W_F$, denoted $(\lambda)_{\widehat{G}}$, is called an \textbf{infinitesimal character}.

The \emph{Local Langlands Correspondence} (LLC) predicts a finite-to-one relationship between the set of $\widehat{G}$-conjugacy classes of admissible homomorphisms of the Weil-Deligne group into the $L$-group, written $\Phi(G/F)$ and the set of smooth irreducible representations of $G(F)$, written $\Pi(G/F)$, satisfying desiderata given in \cite{borel1979}, which we will not recall here. Given $\pi \in \Pi(G/F)$, its \textbf{Langlands parameter} in $\Phi(G/F)$ is denoted $\varphi_\pi$.

The LLC is a theorem in several cases important for our present purpose. First, it is a theorem for all tori, as we recall in Section~\ref{section::llc-for-tori}. Also, in a major breakthrough, Harris and Taylor proved the LLC for $GL_n$, which is among the split connected reductive groups with connected center being considered here.


\subsection{Test functions via the Bernstein center}

The precise statement of the Test Function Conjecture of Haines and Kottwitz relies on a significant amount of machinery from the study of bad reduction of Shimura varieties that will not be covered here; however, we will provide enough detail to define the test function $\phi_r$ at the heart of the conjecture. The primary reference for this section is~\cite{haines2014}.

Although the Bernstein center is initially defined in categorical terms, there are three concrete ways to describe it, each of which will play into the present approach to test functions. The first part of this section explains each of these alternative descriptions. Then we apply this theory to define the test functions. The definition relies on the LLC+ conjecture in order to embed a distribution in the \emph{stable} Bernstein center $\mathfrak{Z}^{\rm{st}}(G)$ as an element of the usual Bernstein center $\mathfrak{Z}(G)$.

We emphasize that our present objective is to study a test function $\varphi_r$ in the context of the Test Function Conjecture and not to explain the conjecture itself. Several important concepts and objects are mentioned in this section with only enough exposition to lead us to a definition of $\varphi_r$. The reader is encouraged to read the relevant parts of \cite{haines2014} for the full story.

\subsubsection{The Bernstein Center}

\begin{defin}
The \textbf{Bernstein center} $\berncent$ of a connected reductive algebraic group $G$ defined over a $p$-adic field is the center of the category $\mathfrak{R}(G)$, i.e., the endomorphism ring of the identity functor. An element $\xi \in \berncent$ is a family of morphisms $\xi_A: A \rightarrow A$ such that for any morphism $f: A\rightarrow B$ the following diagram commutes
{
\large
\[
\xymatrix{
A \ar[d]_{\xi_A} \ar[r]^f & B \ar[d]^{\xi_B} \\
A \ar[r]^f & B
}
\]
}
\end{defin}
\smallskip

The first concrete realization of $\berncent$ is as an algebra of certain distributions. A \textbf{distribution} is a linear map $D: C_c^\infty(G) \rightarrow \valfield$. Given $f \in C_c^\infty(G)$, one can define a new function $D \ast f$; see \cite{haines2014} Section 3.1. If $D$ is ``essentially compact,'' then $D \ast f \in C_c^\infty(G)$. Lemma 4.1 and Corollary 4.2 of \cite{haines2014} show that the set of $G$-invariant, essentially compact distributions on $G$ form a commutative and associative $\valfield$-algebra $(\mathcal{D}(G)_{\rm{ec}}^G, \ast)$.

Second, $\berncent$ is isomorphic to an inverse limit of centers of Hecke algebras. Given a compact open subgroup $J$ of $G$, consider the center $\mathcal{Z}(G, J)$ of the Hecke algebra $\mathcal{H}(G, J)$. This is an algebra under convolution, and we choose a Haar measure $dx_J$ such that $\rm{vol}_{dx_J} (J) = 1$. Let $1_J$ denote the characteristic function of the subgroup $J$. For $J^\prime \subset J$, there is a corresponding morphism of algebras $Z(G, J^\prime) \rightarrow Z(G, J)$ given by $z_{J^\prime} \mapsto z_{J^\prime} \ast_{dx_{J^\prime}} 1_J$. So we can form the inverse limit $\varprojlim_J \mathcal{Z}(G, J)$; it is a fact that $\mathfrak{Z}(G) \cong \varprojlim_J \mathcal{Z}(G, J)$.

The final realization of the Bernstein center uses the inertial equivalence classes defined in Section~\ref{section::background-rep-theory}. For $\mathfrak{s} = [M, \sigma]_G \in \mathfrak{B}(G)$, let
$$
\mathfrak{X}_\mathfrak{s} = \{ (L, \tau)_G \mid (L, \tau)_G \sim (M, \sigma)_G\},
$$
that is, the set of $G$-conjugacy classes of cuspidal pairs encompassed by the inertial equivalence class $\mathfrak{s}$. Now define a disjoint union
$$
\mathfrak{X}_G = \coprod_{\mathfrak{s} \in \mathfrak{B}(G)} \mathfrak{X}_{\mathfrak{s}}.
$$
This set can be given a variety structure. The Bernstein center is isomorphic to the ring of regular functions $\valfield[\mathfrak{X}_G]$.

\begin{thm}
In summary, the Bernstein center $\berncent$ satisfies the following isomorphisms,
$$
\berncent \cong (\mathcal{D}(G)_{\rm{ec}}^G, \ast) \cong \varprojlim_J \mathcal{Z}(G, J) \cong \valfield[\mathfrak{X}_G].
$$
\end{thm}

\begin{proof}
This is all in \cite{haines2014} Section 3.
\end{proof}

Recall that given a distribution $Z \in \berncent$ and a subgroup $J \subset G$, the element $Z \ast 1_J$ belongs to the Hecke algebra $\mathcal{H}(G,J)$. As such, we can study the action of $Z \ast 1_J$ on $\pi^J$ for a representation $\pi \in \mathfrak{R}(G)$ viewed as a $\mathcal{H}(G)$-module.

The inertial equivalence class $\mathfrak{s}$ supporting a smooth representation $\pi \in \mathfrak{R}(G)$ is a point in the variety $\mathfrak{X}_G$. Viewing $Z \in \mathfrak{Z}(G)$ as a regular function on this variety, we may define a scalar $Z(\pi)$ as the value of $Z$ at $\mathfrak{s}$.

\begin{prop}
Let $\pi$ be a finite-length smooth representation. For every compact open subgroup $J \subset G$, $Z \ast 1_J$ acts on $\pi^J$ by $Z(\pi)$.
\end{prop}

\begin{proof}
This statement is \cite{haines2014} Corollary 4.3(a).
\end{proof}

\subsubsection{Definition of a test function}

Test functions may be defined in terms of distributions coming from the Bernstein center. We construct a particular distribution $Z_{V_\mu}$ in the \emph{stable} Bernstein center $\mathfrak{Z}^{\rm{st}}(G)$ attached to a representation $(r_\mu, V_\mu)$ of the Langlands dual group ${^L}G$ determined by a dominant minuscule cocharacter $\mu$ in $\cochars{\torus}$. The test function is obtained by convolving this distribution with the characteristic function of the level structure subgroup of $G$.

Recall from Section~\ref{section::LLC} that restricting an admissible homomorphism $\varphi$ on $W_F^\prime$ gives us an admissible homomorphism $\lambda$ on $W_F$.

\begin{defin}
\label{defin::semisimple-trace}
Let $(r,V)$ be a complex, finite-dimensional representation of ${^L}G$. Given a geometric Frobenius element $\Phi \in W_F$ and an admissible homomorphism $\lambda: W_F \rightarrow {^L}G$, define the \textbf{semisimple trace} by
$$
{\rm{tr}}^{\rm{ss}}(\lambda(\Phi), V) = {\rm{tr}}(r\lambda(\Phi), V^{r\lambda (I_F)}).
$$
\end{defin}

An infinitesimal character $(\lambda)_{\widehat{G}}$ defines an element of a certain variety $\mathfrak{Y}$ (see~\cite{haines2014} Chapter 5). Assuming the LLC+ conjecture \cite{haines2014}, Section 5.2, the Bernstein variety $\mathfrak{X}_G$ has a quasi-finite surjection onto $\mathfrak{Y}$ when $G$ is split.

\begin{prop}
\label{prop::reg-func-semisimple-trace}
The map $\lambda \mapsto {\rm{tr}}^{\rm{ss}}(\lambda(\Phi), V)$ defines an element $Z_V \in \mathfrak{Z}^{\rm{st}}(G)$ as a regular function on $\mathfrak{Y}$ given by
$$
Z_V((\lambda)_{\widehat{G}}) = {\rm{tr}}^{\rm{ss}}(\lambda(\Phi), V)
$$
\end{prop}

\begin{proof}
This is part of statement is Proposition 4.28 of~\cite{haines2014}.
\end{proof}

\begin{rmk}
If our split connected reductive group $G$ defined over $F$ satisfies the LLC+ conjecture, then there is an injective homomorphism $\mathfrak{Z}^{\rm{st}}(G) \rightarrow \mathfrak{Z}(G)$. In this case, the distributions $Z_V$ may be viewed as elements of $\mathfrak{Z}(G)$.
\end{rmk}

\begin{thm} \emph{(The Theorem of the Highest Weight)} Let $G$ be a linear reductive group over an algebraically closed field $K$. An irreducible, finite-dimensional $K$-representation has a unique $\textbf{highest weight}$. Every dominant weight of the root system of $G$ is the highest weight of such a representation, which is unique up to isomorphism.
\end{thm}

\begin{proof}
See~\cite{humphreys1975}, Theorem 31.3. Although the reference states the theorem for semisimple groups, the argument can be applied to reductive groups.
\end{proof}

Thus, given a dominant minuscule cocharacter $\mu \in \cochars{\torus}$, there exists a highest weight representation $(r_\mu, V_{\mu}) \in {\rm{Rep}}(\widehat{G})$, which is unique up to isomorphism.

\begin{defin}
\label{defin::test-function}
Let $G$ be a split connected reductive group defined over a non-archimedean local field $F$ with split maximal torus $\torus$. Consider a degree $r$ unramified extension $F_r/F$. Denote the $F_r$-rational points by $G_r$. The residue field of $F_r$ has cardinality $q^r$. Let $\mu$ be a dominant cocharacter of $\torus$ and $K_r$ a compact open subgroup of $G_r$. Let $t_\mu$ denote the translation element in  $\widetilde{W}$ corresponding to $\mu$. Finally, define the \textbf{test function with $K_r$-level structure} for $(G_r, \mu)$ to be
$$
\phi_r = q^{r\ell(t_\mu)/2}(Z_{V_\mu} \ast 1_{K_r})
$$
\end{defin}

Observe that $\phi_r$ lies in $\mathcal{Z}(G_r,K_r)$ by the theory of the Bernstein center. The terminology ``$K_r$-level structure'' routinely appears in articles on the theory of Shimura varieties with bad reduction at a place dividing $p$. The group we have called $K_r$ corresponds to the compact open subgroup $K_p \subset G(\mathbb{Q}_p)$, which is a part of a Shimura datum. As described in the Introduction, the Test Function Conjecture~\cite{haines2014}, Conjecture 4.30, predicts that $\phi_r$ can be used to prove a certain formula for the semi-simple Lefschetz number via the Langlands-Kottwitz method.

\begin{defin}
The \textbf{test function with $I_r^+$-level structure} is 
$$
\phirone = q^{r\ell(t_\mu)/2}(Z_{V_\mu} \ast 1_{I_r^+}).
$$
\end{defin}


\subsection{Data associated to a depth-zero character}
\label{section::depth-zero-characters}

We now begin preparations to rewrite $\phirone$ via depth-zero characters on $\torus(\mathcal{O}_r)$. As we shall see in the next section, $\phirone$ can be expressed as a sum indexed by the group of depth-zero characters $\xi$ on $\torus(\mathcal{O}_r)$ by considering certain idempotents in the Hecke algebra $\mathcal{H}(G_r)$.

\begin{rmk}
\label{rmk::warning-on-different-weyl-grp-notation}
Many of the definitions and results in this section can be found in {\rm{\cite{roche1998}}}; however, we have followed the notation used in {\rm{\cite{haines-rapoport2012}}}. The danger for confusion is mostly with variants of Weyl groups. Specifically, we use $W$ for the finite Weyl group and $\widetilde{W}$ for the extended affine Weyl group, whereas Roche denotes these groups $\overline{W}$ and $W$ respectively.
\end{rmk}

\subsubsection{First properties}
\label{section::first-properties-dz-chars}

\begin{defin}
Let $\torus$ be a split maximal torus in $G$. A \textbf{depth-zero character} on $\torus(\algintF)$ is a smooth character $\chi: \torus(\algintF) \rightarrow \valfield^{\times}$ that factors through $\torus(\resfield)$. The resulting character $\torus(k_F) \rightarrow \valfield^{\times}$ is also denoted $\chi$. Similarly, a depth-zero character on $\torus(\algintR)$ is a character $\xi$ that factors through $\torus(k_r)$. Let $\torus(\resfield)^\vee$ and $\torus(k_r)^\vee$ denote the sets of depth-zero characters on the groups $\torus(\mathcal{O}_F)$ and $\torus(\mathcal{O}_r)$, respectively.
\end{defin}

Next, we will associate a root system $\Phi_\chi$, called a \textbf{$\chi$-root system}, to a character $\chi$ on $\torus(\mathcal{O}_F)$. For this application, $\chi$ need not have depth zero.

\begin{prop}
The set $\Phi_\chi = \{ \alpha \in \Phi \mid {\chi \circ \alpha^\vee\vert}_{{\mathcal{O}_F^\times}} = 1\}$ is a root system. 
\end{prop}

\noindent This statement also appears in Roche~\cite{roche1998} and Goldstein's thesis~\cite{goldstein1990}.

\begin{proof}
We say that a subset $J$ of a root system $\Phi$ is \emph{closed} if it satisfies the following condition: If $\alpha, \beta \in J$ and $\alpha + \beta \in \Phi$, then $\alpha + \beta \in J$. A subset $J$ of $\Phi$ is \emph{symmetric} if $\alpha \in J$ implies $-\alpha \in J$. Following Bourbaki~\cite{bourbaki4-6}, it suffices to show that $\Phi_\chi$ is a closed, symmetric subset of $\Phi$.

\emph{Closed:} Suppose $\alpha, \beta \in \Phi_\chi$ and $\alpha + \beta \in \Phi$. By direct calculation,
$$
\chi \circ (\alpha+\beta)^\vee(x) = \left(\chi(\alpha^\vee(x))\right)\left(\chi(\beta^\vee(x))\right) = 1.
$$

\emph{Symmetric:} Suppose $\alpha \in \Phi_\chi$. Then $\chi \circ (-\alpha^\vee) (x) = \chi(\alpha^\vee(x)^{-1}) = 1$.
\end{proof}

\begin{lemma}
Let $F_r/F$ be an unramified extension of local fields. The norm map $N_r : F_r \rightarrow F$ restricts to a surjective map $N_r : \mathcal{O}_r^\times \rightarrow \mathcal{O}_F^\times$.
\end{lemma}

\begin{proof}
See \cite{serre1979}, Proposition V.2.3 and its corollary.
\end{proof}

Recall that the norm induces a map $N_r : \torus(\mathcal{O}_r) \rightarrow \torus(\mathcal{O}_F)$. Given $\chi \in \torus(\resfield)^\vee$ and an unramified extension $F_r/F$, define a new character $\chi_r = \chi \circ N_r$ on $\torus(\mathcal{O}_r)$.

\begin{lemma}
The character $\chi_r = \chi \circ N_r : \torus(\mathcal{O}_r) \longrightarrow \mathbb{C}^\times$ has depth zero.
\end{lemma}

\begin{proof}
The norm $N_r : \mathcal{O}_r^\times \rightarrow \mathcal{O}_F^\times$ descends to $N_r : k_r \rightarrow k_F$ as explained in \cite{serre1979}, Section V.2. This induces a commutative diagram on points of the torus:
{
\[
\xymatrix{
\torus(\mathcal{O}_r) \ar[d] \ar[r]^{N_r} & \torus(\mathcal{O}_F) \ar[d] \\
\torus(k_r) \ar[r]^{N_r} & \torus(k_F)
}
\]
}
Because $\chi$ is depth zero, it factors through $\torus(k_F)$. Therefore $\chi \circ N_r : \torus(\mathcal{O}_r) \rightarrow \mathbb{C}^\times$ factors through $\torus(k_r)$ by composing the induced map $\chi : \torus(k_F) \rightarrow \mathbb{C}^\times$ with the lower route through the above diagram.
\end{proof}

\begin{prop}
\label{prop::root-system-equality}
Let $\chi : \torus(\mathcal{O}_F) \rightarrow \mathbb{C}^\times$ be a depth-zero character, and let $\chi_r$ be the associated depth-zero character on $\torus(\mathcal{O}_r)$. Then $\Phi_\chi = \Phi_{\chi_r}$ as subsystems of the ambient root system $\Phi = \Phi(G,T)$.
\end{prop}

\begin{proof}
For any $\alpha \in \Phi$, consider the cocharacter $\alpha^\vee$ in $\cochars{T}$. Since $G$ is split, the maximal torus $\torus$ is $F$-isomorphic to a direct product $\mathbb{G}_m \times \cdots \times \mathbb{G}_m$ of rank equal to ${\rm{rank}}(\torus)$. Galois groups act coordinate-wise on $\torus(F_r)$ and $\torus(F)$, and the following diagram commutes:
{
\[
\xymatrix{
F_r^\times \ar[d]_{\alpha^\vee} \ar[r]^{N_r} & F^\times \ar[d]^{\alpha^\vee} \\
\torus(F_r) \ar[r]^{N_r} & \torus(F)
}
\]
}
Therefore, we have an equality,
$$
(\chi_r \circ \alpha^\vee)(z) = (\chi \circ \alpha^\vee) (N_r(z)),
$$
for any $\alpha \in \Phi$ and $z \in \mathcal{O}_r^\times$.

The norm map $N_r : F_r \rightarrow F$ restricts to a surjection $N_r : \mathcal{O}_r^\times \rightarrow \mathcal{O}_F^\times$. It is then clear that $\chi_r(\alpha^\vee(z)) = 1$ if and only if $\chi(\alpha^\vee(N_r (z))) = 1$.

So, if $\alpha \in \Phi_\chi$, we can conclude that $\chi_r(\alpha^\vee(z)) = 1$ for all $z \in k_r$, i.e., that $\alpha \in \Phi_{\chi_r}$. Similar logic gives the reverse inclusion.
\end{proof}

\begin{prop}
\label{prop::iwahori-prop-quotient}
Let $\torus$ be the split maximal torus in a fixed Borel subgroup $B$ of $G$. Let $I_r$ be the Iwahori subgroup of $G(\mathcal{O}_r)$ that maps onto $B(k_r)$ modulo $\varpi$. Then there is an isomorphism ${\torus(k_F) \cong I/I^+}$. Similarly, $\torus(k_r) \cong I_r/I_r^+$ for the analogous subgroups of $G(F_r)$.
\end{prop}

\begin{proof}
The isomorphism is a consequence of the factorization of $I$ (resp. $I_r$) into a product of torus elements and unipotent elements. See for example  Goldstein's thesis \cite{goldstein1990}, Chapter 2.
\end{proof}

Using $\torus(\resfield) \cong I/I^+$, we can extend a depth-zero character $\chi$ on $\torus(\resfield)$ to a character $\rho_\chi$ on $I$ which is trivial on $I^+$. There is a character $\rho_{\chi_r}$ on $I_r$ similarly derived from $\chi_r$.

We conclude these opening remarks by defining some variants of the Weyl group associated to a depth-zero character. What follows is essentially reproduced from \cite{haines-rapoport2012} Section 9.1.

Let $N_G(T)$ be the normalizer of $T(F)$ in $G(F)$. Recall that the Weyl group of $\Phi$ is $W = N_G(T)/T(F)$, and its extended affine Weyl group is $\widetilde{W} = N_G(T)/T(\mathcal{O}_F)$. The groups $N_G(T)$, $W$ and $\widetilde{W}$ all act on the depth-zero characters by conjugation, e.g., for $w \in W$ and $t \in \torus(\mathcal{O}_F)$, define the $W$-action by ${^w}\chi(t) = \chi(w^{-1}t w)$. Let
$$
W_\chi = \{ w \in W \mid {^w}\chi = \chi\}.
$$

From the definitions, we have surjections $N_G(T) \rightarrow \widetilde{W} \rightarrow W$ and that $W_\chi$ is a subgroup of $W$. Let $\widetilde{W}_\chi$ be the preimage of $W_\chi$ in $\widetilde{W}$, and let $N_\chi$ be the preimage of $W_\chi$ in $N_G(T)$.

Let $\Phi_{\chi, {\rm{aff}}} = \{ a = \alpha + k \mid \alpha \in \Phi_\chi, k \in \mathbb{Z}\}$ be the affine root system arising from $\Phi_\chi$. Let $W_\chi^\circ = \langle s_\alpha \mid \alpha \in \Phi_\chi \rangle$ and $W_{\chi, {\rm{aff}}} = \langle s_a \mid a \in \Phi_{\chi, {\rm{aff}}}\rangle$.

In conclusion, let us state Lemma 9.1.1 of \cite{haines-rapoport2012}, whose proof is due to Roche~\cite{roche1998}.

\begin{lemma}
\label{lemma::chi-weyl-groups}
\begin{enumerate}
\item The group $W_{\chi, {\rm{aff}}}$ is a Coxeter group, whose set of generators $S_{\chi, {\rm{aff}}}$ are the reflections associated to the simple roots of $\Phi_{\chi, {\rm{aff}}}$.
\item There is a canonical decomposition $\widetilde{W}_\chi = W_{\chi, {\rm{aff}}} \rtimes \Omega_\chi$, where $\Omega_\chi$ is the subset of $\widetilde{W}_\chi$ which fixes the base alcove of $\Phi_\chi$. The Bruhat order $\leq_\chi$ and length function $\ell_\chi$ of $W_{\chi, {\rm{aff}}}$ can be extended to $\widetilde{W}_\chi$ such that $\Omega_\chi$ consists of the length-zero elements.
\item If $W_\chi^\circ = W_\chi$, then $W_{\chi, {\rm{aff}}}$ (resp. $\widetilde{W}_\chi$) is the affine (resp. extended affine) Weyl group associated to $\Phi_\chi$.
\end{enumerate}
\end{lemma}

\subsubsection{Hecke algebras and their isomorphisms}

Given a character $\xi$ on $\torus(k_r)$, extend to $\rho_\xi$ on $I_r$ using $\torus(k_r) \cong I_r / I_r^+$ as before. We define the subalgebra $\mathcal{H}(G_r, I_r, \rho_{\xi}) \subset \mathcal{H}(G_r)$ consisting of functions $f$ such that
$$
f(xgy) = \rho_{\xi}(x)^{-1} f(g) \rho_{\xi}(y)^{-1}
$$
where $x, y \in I_r$ and $g \in G_r$. Roche refers to such an $f$ as a \emph{$\rho_{\xi}^{-1}$-spherical function}. Iwahori and Matsumoto~\cite{iwahori-matsumoto1965} gave an explicit presentation for certain Iwahori-Hecke algebras, which generalizes to algebras such as $\mathcal{H}(G_r, I_r, \rho_{\xi})$ as described in works by Goldstein~\cite{goldstein1990}, Morris~\cite{morris1993}, and Roche~\cite{roche1998}. Roche introduced an approach to Hecke algebra isomorphisms using endoscopic groups, which is advantageous in the present situation; however, Goldstein's isomorphism would be sufficient as it specifically covers the case of depth-zero characters for split reductive groups.

We begin by introducing the Hecke algebra attached to a general Coxeter group. These groups will be denoted $(\mathcal{W}, \mathcal{S})$ to differentiate them from the finite Weyl group $(W,S)$ of $G$. The Hecke algebra is defined by making a parameter choice for the following general construction:

\begin{thm}
Let $(\mathcal{W},\mathcal{S})$ be a Coxeter system and $A$ a commutative ring with unity. There is a unique associated $A$-algebra $\mathcal{H}$ based on a free $A$-module $\mathcal{E}$ having basis $T_w$ for $w \in \mathcal{W}$, with parameters $a_s, b_s \in \mathcal{S}$, subject to the relations
$$
T_s T_w =\begin{cases}
T_{sw}, & {\rm{if}}\ \ell(sw) > \ell(w) \\
a_s T_w + b_s T_{sw}, & {\rm{otherwise}}.
\end{cases}
$$
\end{thm}

\begin{proof}
See \cite{humphreys1990}, Sections 7.1-7.3.
\end{proof}

If we set $a_s = q-1$ and $b_s = q$ for all $s \in \mathcal{S}$, then we get the \textbf{Hecke algebra} $\mathcal{H}(\mathcal{W},\mathcal{S})$ as in \cite{humphreys1990}, Section 7.4. If $W_{\rm{aff}}$ is the affine Weyl group of $G$, then the Iwahori-Matsumoto isomorphism for the Iwahori-Hecke algebra is
$$
\mathcal{H}(G, I) \cong \mathcal{H}(W_{\rm{aff}}, S_{\rm{aff}})\ \tilde{\otimes}\ \mathbb{C}[\Omega],
$$
where the notation $\tilde{\otimes}$ refers to a twisted tensor product and whose multiplication on simple tensors is given by
$$
(T_w \otimes T_\sigma) \cdot (T_{w^\prime} \otimes T_{\sigma^\prime}) = T_w T_{\sigma w^\prime \sigma^{-1}} \otimes T_{\sigma \sigma^\prime}
$$
for $w, w^\prime \in W_{\rm{aff}}$ and $\sigma, \sigma^\prime \in \Omega$. The isomorphism $\widetilde{W} \cong W_{\rm{aff}} \rtimes \Omega$ enables us to view the simple tensors $T_w \otimes T_\sigma$ as a basis for $\mathcal{H}(G,I)$ indexed by $\widetilde{W}$.

Following Roche, there are two ways to generalize this isomorphism. One version of the Hecke algebra isomorphism (see \cite{roche1998}, Theorem 6.3) shows directly that
$$
\mathcal{H}(G_r, I_r, \rho_{\chi_r}) \stackrel{\sim}{\longrightarrow} \mathcal{H}(W_{\chi_r, {\rm{aff}}}, S_{\chi_r, {\rm{aff}}})\ \tilde{\otimes}\ \mathbb{C}[\Omega_{\chi_r}].
$$
However, we will follow the second approach, which defines an endoscopic group $H_{\chi_r}$ and shows that $\mathcal{H}(G_r, I_r, \rho_{\chi_r})$ is isomorphic to the Iwahori-Hecke algebra of this endoscopic group with a suitably chosen Iwahori subgroup $I_{H_r}$. Then the original Iwahori-Matsumoto isomorphism gives a presentation of the Hecke algebra in terms of a basis indexed by $\widetilde{W}_{\chi_r}$.

Before explaining Hecke algebra isomorphisms according to Roche, we make the following remark to clarify what assumptions are being made.

\begin{rmk}
\label{rmk::roche-assumptions}
The following will be enforced from now on:
\begin{enumerate}
\item As before, $G$ is a split connected reductive group with connected center,
\item The derived group $G_{\rm{der}}$ is simply-connected (see Section~\ref{section::lattices-in-char-groups}), and
\item $W_\chi = W_\chi^\circ$.
\end{enumerate}
Under these conditions, we may avoid the restrictions on ${\rm{char}}(\resfield)$ made by Roche in~\cite{roche1998} to prove Hecke algebra isomorphisms for characters with positive depth.  Other restrictions are needed to ensure $W_\chi = W_\chi^\circ$. The theory of Hecke algebra isomorphisms associated to depth-zero characters holds without any restriction on ${\rm{char}}(\resfield)$.

Proposition~\ref{prop::w-chi-equality} states that $W_\chi = W_\chi^\circ$ holds for general linear groups and general symplectic groups without restrictions, these two important examples satisfy the above criteria.
\end{rmk}



\begin{prop}
\label{prop::w-chi-equality}
Suppose $G$ is a split connected reductive group with connected center defined over $F$. If $G = GL_n$ or $G = GSp_{2n}$, then $W_\chi = W_\chi^\circ$ without restriction on residue characteristic.
\end{prop}

\begin{proof}
The proof for all split connected groups with connected center can be extracted from pages 395-397 of \cite{roche1998}, but this comes at the cost of some restrictions on ${\rm{char}}(\resfield)$. The cases of $GL_n$ and $GSp_{2n}$ are proved to be independent of such restrictions in an unpublished manuscript of Haines and Stroh.
\end{proof}

We are ready to give Roche's definition of the dual group $H_{\chi_r}$. In fact, Roche defines two groups $\tilde{H}_{\chi_r}$ and $H_{\chi_r}$; however, if $G$ has connected center and $W_\chi = W_\chi^\circ$, then $\tilde{H}_{\chi_r} = H_{\chi_r}$. See~\cite{roche1998}, Section 8, for the complete story.

Let $H_{\chi_r}$ be the split connected reductive group over $\mathcal{O}_r$ associated to the root datum $(\chars{T}, \Phi_{\chi_r}, \cochars{T}, \Phi_{\chi_r}^\vee)$. By Proposition~\ref{prop::root-system-equality}, $\Phi_{\chi_r} = \Phi_\chi$. Consequently, $W_\chi^\circ$ is the Weyl group for $H_{\chi_r}$, while $\widetilde{W}_\chi$ is its extended affine Weyl group by Lemma~\ref{lemma::chi-weyl-groups}. We may assume $T$ is the split maximal torus inside $H_{\chi_r}$, and there is an Iwahori subgroup $I_{H_r} \subset H_{\chi_r}$ determined by the positive roots $\Phi_{\chi}^+$. Thus we come to consider the Iwahori-Hecke algebra $\mathcal{H}(H_{\chi_r}, I_{H_r})$. When considering this algebra, we normalize Haar measure for the convolution integral such that ${\rm{vol}}(I_{H_r}) = 1$.


\begin{thm}
The algebras $\mathcal{H}(H_{\chi_r}, I_{H_r})$ and $\mathcal{H}(G_r, I_r, \rho_{\chi_r})$ are isomorphic via a family of support-preserving isomorphisms.
\end{thm}

\begin{proof}
This is \cite{roche1998} Theorem 8.2.
\end{proof}

We make a specific choice of isomorphism among the family established by the theorem, following the presentation in Haines-Rapoport~\cite{haines-rapoport2012}, Section 9. Recall that $N_\chi$ is the inverse image of $\widetilde{W}_\chi$ arising from the surjective map $N \rightarrow \widetilde{W}$.

\begin{lemma}
\label{lemma::existence-breve-chi}
Let $\tilde{\chi}_r$ denote an extension of $\chi_r$ to $\torus(F_r)$. Then $\tilde{\chi}_r$ extends to a character $\breve{\chi}_r$ on $N_{\chi_r}$ if and only if $\tilde{\chi}_r$ is $W_{\chi_r}$-invariant.
\end{lemma}

\begin{proof}
This is \cite{haines-rapoport2012}, Lemma 9.2.3.
\end{proof}

For a fixed choice of uniformizer $\varpi$, Remark 9.2.4 loc. cit. defines a specific $W_{\chi_r}$-invariant extension of $\chi_r$, called the \emph{$\varpi$-canonical extension}, by
$$
\tilde{\chi}_r^\varpi (\nu(\varpi) t_0) = \chi(t_0),
$$
for all $\nu \in \cochars{\torus}$ and $t_0 \in \torus(\mathcal{O}_r)$. \textbf{From now on let $\breve{\chi}_r$ be the character on $N_{\chi_r}$ determined according to Lemma~\ref{lemma::existence-breve-chi}.}

For each $w \in \widetilde{W}_\chi$, fix a choice of $n_w \in N_\chi$ such that $n_w \mapsto w$ under the surjection $N_\chi \rightarrow \widetilde{W}_\chi$. Still following Haines-Rapoport, define $[I_r n_w I_r]_{\breve{\chi}_r}$ in $\mathcal{H}(G_r, I_r, \rho_{\chi_r})$ to be the ``unique element in $\mathcal{H}(G_r, I_r, \rho_{\chi_r})$ which is supported on $I_r n_w I_r$ and whose value at $n_w$ is $\breve{\chi}_r^{-1}(n_w)$. Note that $[I_r n_w I_r]_{\breve{\chi}_r}$ depends only on $w$, not on the choice of $n \in N_\chi$ mapping to $w \in \widetilde{W}_\chi$.'' (\cite{haines-rapoport2012}, p. 766.)

Recall that Haar measure on $\mathcal{H}(G_r, I_r, \rho_{\chi_r})$ is normalized so that ${\rm{vol}}(I_r^+) = 1$, while the measure on $\mathcal{H}(H_{\chi_r}, I_{H_r})$ is normalized so that ${\rm{vol}}(I_{H_r}) = 1$.

\begin{lemma}
\label{lemma::hecke-algebra-mapping-coefficients}
The isomorphism $\Psi_{\breve{\chi}_r}: \mathcal{H}(G_r, I_r, \rho_{\chi_r}) \stackrel{\sim}{\longrightarrow} \mathcal{H}(H_{\chi_r}, I_{H_r})$ maps
$$
q^{-r\ell(w)/2}[I_r n_w I_r]_{\breve{\chi}_r} \longmapsto q^{-r\ell_\chi(w)/2}[I_{H_r} n_w I_{H_r}].
$$
\end{lemma}

\begin{proof}
This is quoted from \cite{haines-rapoport2012}, Theorem 9.3.1, but the result comes from \cite{roche1998}.
\end{proof}

The isomorphism $\Psi_{\breve{\chi}_r}$ will be referred to as ``the'' Hecke algebra isomorphism between these algebras. It depends on the choice of $\varpi$ used to construct $\breve{\chi}_r$.

Let $\tilde{\chi} : T(F) \rightarrow \valfield^\times$ be an extension of $\chi$. There is a corresponding inertial equivalence class $\mathfrak{s}_\chi = [T(F), \tilde{\chi}]_G$, and hence a Bernstein component $\mathfrak{R}_{\mathfrak{s}_\chi} (G)$, which we refer to as a \emph{depth-zero principal series component} of $\mathfrak{R}(G)$. If $\chi = 1$, i.e., the trivial character, the inertial class is denoted $\mathfrak{s}_1$. The component $\mathfrak{R}_{\mathfrak{s}_1}(G)$ is the \emph{unramified principal series component}.

\begin{prop}
\label{prop::equivalence-of-repn-cat}
The isomorphism $\Psi_{\breve{\chi}_r}$ sets up an equivalence of categories $\mathfrak{R}_{\mathfrak{s}_\chi} (G_r) \cong \mathfrak{R}_{\mathfrak{s}_1} (H_{\chi_r})$ under which $i_{B_r}^{G_r}(\tilde{\chi}_r^\varpi \eta)$ corresponds to $i_{B_{H_r}}^{H_{\chi_r}}(\eta)$, where $\eta$ is an unramified character.
\end{prop}

\begin{proof}
This statement is Proposition 9.3.3(2) in \cite{haines-rapoport2012}.
\end{proof}

Using the statements and results from Roche and Haines-Rapoport, we have established an isomorphism between the Hecke algebra $\mathcal{H}(G_r, I_r, \rho_{\chi_r})$ and the Iwahori-Hecke algebra $\mathcal{H}(H_{\chi_r}, I_{H_r})$. Then by the Iwahori-Matsumoto isomorphism,
$$
\mathcal{H}(H_{\chi_r}, I_{H_r}) \cong \mathcal{H}(W_{{\chi_r}, {\rm{aff}}}, S_{{\chi_r}, {\rm{aff}}})\ \tilde{\otimes}\ \mathbb{C}[\Omega_{\chi_r}].
$$
The right-hand side of this expression is called a \emph{twisted affine Hecke algebra}. It has a basis $\{T_w\}_{w \in \widetilde{W}_\chi}$, where each of the $T_w$ is invertible. 

Recall from Proposition~\ref{prop::root-system-equality} that $\Phi_\chi = \Phi_{\chi_r}$. It follows that the Coxeter systems $(W_{\chi_r, {\rm{aff}}}, S_{\chi_r, {\rm{aff}}})$ and $(W_{\chi, {\rm{aff}}}, S_{\chi, {\rm{aff}}})$ are identical; however, we could make different parameter choices $a_s = q^{r} -1, b_s = q^r$ and $a_s = q-1, b_s = q$, respectively for $s \in S_{\chi, {\rm{aff}}}$, which would yield two different twisted affine Hecke algebras that share a common basis $\{T_w\}$ indexed by $\widetilde{W}_{\chi_r}$.

In what follows, we would like to work with data defined in terms of characters $\chi \in \torus(\resfield)^\vee$. If we set the parameters $a_s = q^{r} -1, b_s = q^r$ with respect to the Coxeter system $(W_{\chi, {\rm{aff}}}, S_{\chi, {\rm{aff}}}) = (W_{\chi_r, {\rm{aff}}}, S_{\chi_r, {\rm{aff}}})$, then we get the isomorphism
$$
\mathcal{H}(H_{\chi_r}, I_{H_r}) \cong \mathcal{H}_r(W_{{\chi}, {\rm{aff}}}, S_{{\chi}, {\rm{aff}}})\ \tilde{\otimes}\ \mathbb{C}[\Omega_{\chi}],
$$
where the subscript ``r'' in $\mathcal{H}_r$ is meant to remind the reader that the parameters for the Hecke algebra of the Coxeter system depend on $r$, even though we are working with the basis indexed by $\widetilde{W}_\chi$.

The inversion formula for the basis elements determines the $\widetilde{R}$-polynomials, which arise from Kazhdan-Lusztig theory, for the group $(W_{\chi, {\rm{aff}}}, S_{\chi, {\rm{aff}}})$. This notion extends to the extended affine Weyl group $\widetilde{W}_\chi$.

\begin{defin}
\label{defin::r-poly-from-hecke-algebras}
Let $Q_r = q^{-r/2} - q^{r/2}$. For the twisted affine Hecke algebra $\mathcal{H}_r$ associated to $\widetilde{W}_\chi$, there is a family of polynomials $\widetilde{R}^{\chi}_{x,y}(Q_r)$, with $x, y \in \widetilde{W}_\chi$, determined by the inversion formula for a normalized basis element $\tilde{T}_{w,r} = q^{-r\ell_\chi(w)/2} T_w$ in $\mathcal{H}_r$. The polynomials are defined by
$$
\tilde{T}_{w^{-1}, r}^{-1} = \sum_{x \in \widetilde{W}_\chi} \tilde{R}^{\chi}_{x,w}(Q_r) \tilde{T}_{x,r}.
$$
\end{defin}

We will revisit these polynomials in Chapter 4, where they are defined in terms of an abstract Coxeter system.

\subsubsection{The LLC for Tori}
\label{section::llc-for-tori}

The Local Langlands Correspondence is a theorem in the case of tori. While the LLC is true for general tori, we consider only the split case. Therefore ${^L}T$ can be viewed as $\widehat{\torus}$, the torus determined by the dual root datum.

\begin{thm}
\label{thm::LLC-for-Tori}
\emph{(LLC for Tori: Split Case)} Let $\torus$ be a split torus over $F$ and $W_F$ the Weil group of $F$. Then there is a correspondence
$${\rm{Hom}}(W_F, \dualtorus(\valfield)) = {\rm{Hom}}(T(F), \valfield^\times).$$
Denote the Langlands parameter of a character $\xi : T(F) \rightarrow \valfield^\times$ by $\varphi_{\xi} : W_F \rightarrow \dualtorus(\valfield).$
\end{thm}

\begin{proof}
For details see Yu~\cite{yu2009}.
\end{proof}

Let $\chi : \torus(\mathcal{O}_F) \rightarrow \valfield^\times$ be a depth-zero character, and let $\tilde{\chi}$ denote an extension to $T(F)$.  Let $\tau_F$ denote the Artin reciprocity map from local class field theory. For any $\nu \in \chars{\dualtorus} = \cochars{\torus}$, Theorem~\ref{thm::LLC-for-Tori} yields the following commutative diagram:
{
\[
\xymatrix{
W_F \ar[d]_{\nu\circ\tau_F} \ar[r]^{\varphi_{\tilde{\chi}}} & \dualtorus(\valfield) \ar[d]^\nu \\
\torus(F) \ar[r]^{\tilde{\chi}} & \valfield^\times
}
\]
}

\begin{lemma}
Consider $\chi : \torus(\mathcal{O}_F) \rightarrow \valfield^\times$ and an extension $\tilde{\chi}$ to $\torus(F)$.
\begin{enumerate}
\item $\tilde{\chi}$ is an unramified character if and only if its Langlands parameter $\varphi_{\tilde{\chi}}$ is trivial on the inertia subgroup $I_F \subset W_F$.
\item The restriction of $\varphi_{\tilde{\chi}}$ to $I_F$ depends only on $\chi$. This restriction is denoted $\varphi_{\chi}$.
\item $\chi$ is depth-zero if and only if $\varphi_{\chi}$ is trivial on the subgroup $\tau_F^{-1} (1 + \varpi \mathcal{O}_F) \subset I_F$. In this case, $\chi$ is determined by the value of $\varphi_\chi$ on an element $x \in I_F$ whose image $\tilde{x} \in \mathcal{O}_F^\times$ under $\tau_F$ projects to a generator of the multiplicative group of the residue field $k_F$.
\end{enumerate}
\end{lemma}

\noindent We will give a proof of the lemma; however, these statements also appear in Roche's article \cite{roche1998} in the discussion following his Theorem 8.2.

\begin{proof}
Recall that a character is \emph{unramified} if it is trivial on
$$
{^\circ}T = \{ t \in T(F) \mid \rm{val}_F (\nu(t)) = 0, \forall \nu \in X^*(T)\}.
$$
For split groups over non-archimedean local fields, this group is the maximal compact open subgroup ${^\circ}T = T(\mathcal{O}_F)$. The Artin map gives a surjection $\tau_F : I_F \rightarrow \mathcal{O}_F^\times$.

Suppose $\tilde{\chi}$ is an unramified character. If $z \in I_F$, then $\nu(\tau_F(z)) \in \torus(\mathcal{O}_F)$, which implies $\nu(\varphi_{\tilde{\chi}} (z) ) =1$ for all $\nu \in X^*(T)$. Therefore, $\varphi_{\tilde{\chi}} (z) = 1$. Conversely, suppose $t \in T(\mathcal{O}_F)$. Then it can be written as a product $t = \prod_i t_i = \prod_i \nu_i(\tau_F(z_i))$ for some $z_i \in I_F$ and cocharacters $\nu_i \in \cochars{\torus}$, because the $\mathcal{O}_F$-points of $T$ are generated by the images of its cocharacters applied to $\mathcal{O}_F^\times$. We have
$$
\tilde{\chi}(t) = \prod_i \tilde{\chi}(\nu_i(\tau(z_i))) = \prod \nu_i(\varphi_{\tilde{\chi}} (z_i)) =1,
$$
and $\varphi_{\tilde{\chi}}$ is trivial on each $z_i$ by assumption. This proves the first statement.

To prove the second statement, consider any two extensions $\tilde{\chi}_1$, $\tilde{\chi}_2$ to $T(F)$. We will show that the parameters $\varphi_{\tilde{\chi}_1}$ and $\varphi_{\tilde{\chi}_2}$ agree on the inertia subgroup, i.e., $\varphi_{\tilde{\chi}_1} \vert_{I_F} = \varphi_{\tilde{\chi}_2}\vert_{I_F}$.

Choose any $z \in I_F$. Then for all $\nu \in \cochars{\torus} = \chars{\dualtorus}$ and $i=1,2$:
$$
\nu(\varphi_{\tilde{\chi}_i}(z)) = \tilde{\chi}_i (\nu(\tau_F(z))).
$$
Since $\tau_F$ maps $I_F$ onto $\mathcal{O}_F^\times$, $\nu(\tau_F(z)) \in \torus(\mathcal{O}_F)$. Therefore, $\tilde{\chi}_i (\nu(\tau_F(z))) = \chi(\nu(\tau_F(z)))$ for $i = 1,2$. Translating this back into a statement about Langlands parameters gives
$$
\nu(\varphi_{\tilde{\chi}_1}(z)) = \chi(\nu(\tau_F(z))) = \nu(\varphi_{\tilde{\chi}_2}(z)).
$$
Since this holds for all $\nu$, we conclude $\varphi_{\tilde{\chi}_1} \vert_{I_F} = \varphi_{\tilde{\chi}_2}\vert_{I_F}$, as our choice of $z \in I_F$ was arbitrary.

Finally, recall that $\chi$ is depth-zero if it factors through $T(k_F)$. Pick any $z$ in $\tau_F^{-1} (1 + \varpi \mathcal{O}_F)$. For all $\nu \in X^*(\widehat{T})$, $\nu(\varphi_{\chi} (z)) = \chi (\nu(\tilde{z}))$ where $\tilde{z} = \tau_F (z)$ belongs to $1 + \varpi \mathcal{O}_F$. Thus if $\chi$ is depth-zero, all $\nu(\varphi_{\chi} (z)) = 1$, i.e., $\varphi_{\chi} (z) = 1$. The relation $\nu(\varphi_{\chi} (z)) = \chi (\nu(\tilde{z}))$ also implies the converse direction, namely that if $\varphi_\chi$ is trivial on $\tau_F^{-1} (1 + \varpi \mathcal{O}_F) \subset I_F$ then $\chi$ is a depth-zero character. 

Choose $x \in I_F$ such that $\tilde{x} = \tau(x) \in \mathcal{O}_F^\times$ projects to a generator of $k_F^\ast$ under the reduction map. For any $t \in T(\mathcal{O}_F)$, there is an expression
$$
\chi(t) = \chi\left( \prod_i \nu_i (\tilde{x}^{n_i}) \right) = \prod_i \chi(\nu_i (\tilde{x}))^{n_i},
$$
where the $\nu_i$ are a basis of $\cochars{\torus}$. Therefore, we know the value $\chi(t)$ if we know the values $\chi(\nu_i(\tilde{x}))$. This proves the third statement of the lemma.
\end{proof}

\begin{defin}
\label{endoscopic-element-definition}
Fix a choice of $x \in I_F$ such that $\tau_F(x) \in \mathcal{O}_F^\times$ projects to a generator of $k_F^\ast$, and consider any depth-zero character $\chi$ on $\torus(\mathcal{O}_F)$. Define the \textbf{endoscopic element} in $\widehat{T}(\valfield)$ associated to $\chi$ by $\kappa_\chi = \varphi_\chi (x) $.
\end{defin}

\begin{cor}
Fix a choice of $x \in I_F$ such that $\tau_F(x) \in \mathcal{O}_F^\times$ projects to a generator of $k_F^\ast$. The (finite) group of depth-zero characters $\torus(k_F)^\vee$ is isomorphic to the group of endoscopic elements $\kapchi$ determined by $x$.
\end{cor}

\begin{prop}
\label{prop::dz-endoscopic-elements-equal-kernel}
Let $q$ denote the cardinality of the residue field $k_F$. Let $K_{q-1}$ denote the kernel of the map $\dualtorus(\valfield) \rightarrow \dualtorus(\valfield)$ defined by $\kappa \mapsto \kappa^{q-1}$. The group of depth-zero endoscopic elements is $K_{q-1}$; that is, there is $\chi \in \torus(\resfield)^\vee$ such that $\kappa = \kappa_\chi$ if and only if $\kappa \in K_{q-1}$.
\end{prop}

\begin{proof}
Suppose $\kappa = \kappa_\chi$ for some $\chi \in \torus(\resfield)^\vee$. Then $\kappa_\chi = \varphi_\chi(x)$ for some $x \in I_F$ such that $\tilde{x} = \tau_F(x)$ descends to a generator of $k_F^\ast$. For any $\nu \in \cochars{\torus} = \chars{\dualtorus}$,
$$
\chi(\nu(\tilde{x}))^{q-1} = \nu(\kappa_\chi^{q-1}).
$$
Because $\chi$ descends to a character on $\torus(\resfield)$, the values $\chi(\nu(\tilde{x}))$ are all $(q-1)$-th roots of unity. Therefore, $\nu(\kappa_\chi^{q-1}) = 1$ for all $\nu \in \chars{\dualtorus}$, which implies $\kappa_\chi^{q-1} = 1$. That is, $\kappa_\chi \in K_{q-1}$.

Now suppose we start with $\kappa \in K_{q-1}$. We will produce a depth-zero character $\chi$ on $\torus(\resfield)$ such that $\kappa = \kappa_\chi$.

Since $\torus$ is split, $\dualtorus$ is also split and so is isomorphic to a product of copies of $\mathbb{G}_m$. We think of $\kappa = {\rm{diag}}(\kappa_1, \ldots, \kappa_d)$, where each coordinate $\kappa_i$ is a $(q-1)$-th root of unity.  Fix a choice of $x \in I_F$ such that $\tau_F(x)$ descends to a primitive $(q-1)$-th root $\zeta$ generating $\resfield^\ast$. Now $\kappa_i = \zeta^{n_i}$ for some $0 \le n_i < q-1$.

For any $\nu \in \cochars{\torus}$, the element $\nu(\tau_F(x))$ can be projected into $\torus(\resfield)$ and thought of as a diagonal matrix ${\rm{diag}}(\zeta^{r_1}, \ldots, \zeta^{r_d})$. Let $\chi_0$ be the depth-zero character such that
$$
\chi_0(\nu(\tau_F(x))) = \prod_{i=1}^d \zeta^{r_i n_i}.
$$
Then by construction, for any $\nu \in \cochars{\torus} = \chars{\dualtorus}$ we have
$$
\chi_0(\nu(\tau_F(x))) = \nu(\kappa).
$$
By Theorem~\ref{thm::LLC-for-Tori}, $\chi_0(\nu(\tau_F(z))) = \nu(\varphi_{\chi_0}(z))$ for all $z \in W_F$. So in particular $\nu(\varphi_{\chi_0}(x)) = \nu(\kappa)$ for all $\nu$. This shows $\kappa = \kappa_{\chi_0}$.
\end{proof}

\begin{rmk}
We sometimes use the notation $\dualtorus(\valfield)^{\rm{dz}}$ in place of $K_{q-1}$ to remind the reader of the connection with depth-zero endoscopic elements.
\end{rmk}

Later, we will encounter values $\chi^{-1}_r (s)$ for depth-zero $\chi$ and $s \in \torus(k_r)$. The following lemma rephrases this scalar in terms of endoscopic elements.

\begin{lemma}
\label{lemma::defin-of-gamma-nrs}
Let $s \in \torus(k_r)$. Then there exists a character $\gamma_{N_r s} \in \chars{\dualtorus}$ such that $\chi_r (s) = \gamma_{N_r s}(\kapchi)$ in $\valfield^\times$ for all $\chi \in \torus(\resfield)^\vee$.
\end{lemma}

\begin{proof}
Let $\tilde{x} = \tau_F(x)$ where $x \in I_F$ the element determining the correspondence between $\chi$ and $\kapchi$. Let $\{\omega_i\}$ be a basis for $\cochars{\torus}$. There are integers $n_i$ such that $N_r(s) = \prod_i \omega_i(\tilde{x})^{n_i}$ in $\torus(\resfield)$. Let $\gamma_{N_r s} = \sum_i n_i\omega_i.$ Then
$$
\chi(\gamma_{N_r s} (\tilde{x})) = \prod_i \chi(\omega_i (\tilde{x})^{n_i}) = \left(\sum_i n_i \omega_i\right)(\kapchi) = \gamma_{N_r s} (\kapchi).
$$
\end{proof}


\subsection{A first formula for $\phirone$}
\label{section::first-formula}

It looks difficult to compute a test function, in the sense of the Langlands-Kottwitz method, from the definition given in terms of distributions in the Bernstein center. Recall that the test function is written $\phirone$ when it has $I_r^+$-level structure. This section shows how to use the various data coming from depth-zero characters to give a more explicit formula for a function $\phironeaug$ whose twisted orbital integrals are identical to those of the test function $\phirone$, as described in the Introduction.

The first step is to write $\phirone = q^{r \ell(t_\mu)/2} (Z_{V_\mu} \ast 1_{I_r^+})$ as a sum indexed by the depth-zero characters of $\torus(\mathcal{O}_r)$. Its twisted orbital integrals vanish at the summands corresponding to characters $\xi \in \torus(k_r)^\vee$ which are not norms of characters in $\torus(\resfield)^\vee$. The Hecke algebra isomorphism $\Psi_{\breve{\chi}_r}: {\mathcal{H}(G_r, I_r, \rho_{\chi_r}) \stackrel{\sim}{\longrightarrow} \mathcal{H}(H_{\chi_r}, I_{H_r})}$ shows that summands indexed by norms $\chi_r = \chi \circ N_r$ map to sums of Bernstein functions in the center of $\mathcal{H}(H_{\chi_r}, I_{H_r})$. We apply an explicit formula for Bernstein functions attached to dominant minuscule cocharacters to get the formula for $\phironeaug$.

\subsubsection{Definition of $\phironeaug$}


Recall from Section~\ref{section::first-properties-dz-chars} that a depth-zero character $\xi \in \torus(k_r)^\vee$ extends to a character $\rho_\xi$ on $I_r$ which is trivial on $I_r^+$. Similarly, recall that $1_K$ refers to the characteristic function of a subgroup $K \subseteq G_r$. Let us extend this notation:
$$
1_K^{\rho} (x) = \begin{cases}
\rho(x)^{-1}, & {\rm{if}}\ x \in K, \\
0, & {\rm{otherwise.}}
\end{cases}
$$

\begin{lemma}
\label{lemma::idempotents}
For $\xi \in \torus(k_r)^\vee$, define $e_\xi = \rm{vol}(I_r)^{-1} 1_{I_r}^{\rho_\xi}$ in $\mathcal{H}(G_r)$. The elements $e_{\xi}$ are idempotents satisfying $1_{I_r^+} = \sum_{\xi \in \torus(k_r)^\vee} e_{\xi}$.
\end{lemma}

\begin{proof}
Let $dz$ denote Haar measure on $G_r$ normalized such that $\rm{vol}(I_r^+) = 1$. We will compute $e_\xi \ast e_\xi$ from the definition of the convolution integral, namely
$$
(e_\xi \ast e_\xi)(g) = \int_{G_r} e_\xi(gz^{-1})e_\xi(z) dz.
$$
Notice that if $g \notin I_r$, then $(e_\xi \ast e_\xi)(g) = 0$: if $z \notin I_r$, then $e_\xi(z) = 0$ in the integrand; otherwise $e_\xi(gz^{-1}) = 0$. Assuming $g \in I_r$, the integral becomes
$$
(e_\xi \ast e_\xi)(g) = \int_{I_r} \rho_\xi(gz^{-1})^{-1} \rho_\xi(z)^{-1} \rm{vol}(I_r)^{-2} dz.
$$
Since $\rho_\xi(gz^{-1})^{-1} \rho_\xi(z)^{-1} = \rho_\xi(g)^{-1}\rho_\xi(z)\rho_\xi(z)^{-1} = \rho_\xi(g)^{-1}$, we get
$$
(e_\xi \ast e_\xi)(g) = \rho_\xi(g)^{-1}\rm{vol}(I_r)^{-2}\int_{I_r} dz.
$$
It follows that $e_\xi \ast e_\xi = e_\xi$, verifying the first claim.

Next we show that $1_{I_r^+} = \sum_{\xi \in \torus(k_r)^\vee} e_{\xi}$. First, observe that if $z \notin I_r$, then both sides are zero. Second, suppose $z \in I_r\backslash I_r^+$. Then $1_{I_r^+} (z) = 0$, and
$$
\sum_{\xi \in T(k_r)^\vee} e_{\xi}(z) = \rm{vol}(I_r)^{-1} \sum_{\xi \in T(k_r)^\vee} \rho_\xi(z)^{-1}.
$$
But $\sum_{\xi \in T(k_r)^\vee} \rho_\xi(z)^{-1} = 0$: we can find $\xi_0 \in \torus(k_r)^\vee$ such that $\rho_{\xi_0} (z) \neq 1$ and write
$$
\sum_{\xi \in \torus(k_r)^\vee} \rho_\xi(z)^{-1} = \sum_{\xi \in \torus(k_r)^\vee} (\rho_{\xi_0}\rho_\xi)(z)^{-1} = \rho_{\xi_0}(z)^{-1} \sum_{\xi \in \torus(k_r)^\vee} \rho_\xi(z)^{-1}.
$$
Thus we are reduced to looking at $z \in I_r^+$. Here we have
$$
\sum_{\xi \in \torus(k_r)^\vee} \rho_\xi(z)^{-1} = \sum_{\xi \in \torus(k_r)^\vee} {\rm{vol}}(I_r)^{-1} = \vert \torus(k_r)\vert [I_r : I_r^+]^{-1} = 1.
$$
\end{proof}

\begin{lemma}
\label{lemma::orbital-intergrals-zero}
Suppose $\xi \in \torus(k_r)^\vee$ is not a norm, that is, there is not a $\chi \in \torus(\resfield)^\vee$ such that $\xi = \chi \circ N_r$. Then all twisted orbital integrals at $\theta$-semisimple elements vanish on functions in $\mathcal{H}(G_r, I_r, \rho_\xi)$.
\end{lemma}

\begin{proof}
This is Lemma 10.0.4 of~\cite{haines2012}.
\end{proof}

\begin{cor}
\label{cor::equality-of-orbital-integrals}
The function $\phironeaug \in C_c^\infty(G_r)$ defined by
$$
\phironeaug = [I_r : I_r^+]^{-1} q^{r\ell(t_{\mu})/2} \sum_{\chi \in \torus(k_F)^\vee} Z_{V_\mu} \ast e_{\chi_r},
$$
satisfies ${\rm{TO}}_{\delta\theta} (\phirone) = {\rm{TO}}_{\delta\theta}(\phironeaug)$.
\end{cor}

\begin{proof}
Recall that Haar measure on $\mathcal{H}(G_r, I_r^+)$ is normalized such that ${\rm{vol}}(I_r^+) = 1$, while Haar measure on $\mathcal{H}(G_r, I_r, \rho_{\chi_r})$ is normalized to have ${\rm{vol}}(I_r) = 1$. The function $\phironeaug \in \mathcal{H}(G_r, I_r^+)$ is defined by summing up functions $Z_{V_\mu} \ast e_{\chi_r} \in \mathcal{H}(G_r, I_r, \rho_{\chi_r})$; thus we must account for the different normalizations when rewriting $\phirone$ using Lemma~\ref{lemma::idempotents}. Thus we have the intermediate result
$$
{\rm{TO}}_{\delta\theta} (\phirone) = {\rm{TO}}_{\delta\theta}\left([I_r : I_r^+]^{-1} q^{r\ell(t_{\mu})/2} \sum_{\xi \in \torus(k_r)^\vee} Z_{V_\mu} \ast e_{\xi}\right).
$$
But Lemma~\ref{lemma::orbital-intergrals-zero} says that ${\rm{TO}}_{\delta\theta} (Z_{V_\mu} \ast e_{\xi}) = 0$ if $\xi$ is not a norm, which means there is no $\chi \in \torus(\resfield)^\vee$ such that $\xi = \chi_r$. This shows ${\rm{TO}}_{\delta\theta} (\phirone) = {\rm{TO}}_{\delta\theta}(\phironeaug)$.
\end{proof}

\begin{rmk}
\label{rmk::clarify-dependence-on-llc+}
Recall that the definition of a test function $\phi_r$ invokes the LLC+ conjecture to view the distribution $Z_{V_\mu}$ as an element of the Bernstein center. Let us give an unconditional definition for $\phi_{r,1}^\prime$ that agrees with Definition~\ref{defin::test-function} whenever LLC+ holds.

For each $\chi_r$, define $Z_{V_\mu} \ast e_{\chi_r}$ to be the unique function in the center $\mathcal{Z}(G_r, I_r, \rho_{\chi_r})$ which acts on the $\rho_{\chi_r}$-isotypical component of all $i_{B_r}^{G_r}(\tilde{\chi}_r)$ in $\mathfrak{R}_{[T, \tilde{\chi}_r]}(G_r)$ by the scalar ${\rm{tr}}^{\rm{ss}}(\varphi_{\tilde{\chi}_r}(\Phi),V_\mu)$. Of course, the Langlands parameter $\varphi_{\tilde{\chi}_r}$ exists by the LLC for Tori. Sum these functions as before to get a new definition of $\phi_{r,1}^\prime$ that satisfies ${\rm{TO}}_{\delta\theta} (\phirone) = {\rm{TO}}_{\delta\theta}(\phironeaug)$ as in Corollary~\ref{cor::equality-of-orbital-integrals}.
\end{rmk}

\subsubsection{Bernstein functions for dominant minuscule cocharacters}
\label{section::bernstein-functions}

The center of an Iwahori-Hecke algebra can be described in terms of \emph{Bernstein functions} indexed by dominant cocharacters of the maximal torus. In certain cases, such as the case of dominant and minuscule cocharacters of interest here, Haines proved explicit formulas for Bernstein functions \cite{haines2000}, \cite{haines2000b}, \cite{haines-pettet2002}. Our presentation states Haines's formula for the Iwahori-Hecke algebra $\mathcal{H}(H_{\chi_r}, I_{H_r})$.

Recall that $\mathcal{H}(H_{\chi_r}, I_{H_r})$ is isomorphic to the twisted affine Hecke algebra $\mathcal{H}_r (W_{\chi, {\rm{aff}}}, S_{\chi, {\rm{aff}}})\ \tilde{\otimes}\ \mathbb{C}[\Omega_\chi]$, which we sometimes denote $\mathcal{H}_r$. In the notation of Definition~\ref{defin::r-poly-from-hecke-algebras}, this algebra has a normalized basis $\{\tilde{T}_{w,r} \mid w \in \widetilde{W}_\chi\}$. If $w = t_\lambda$ is a translation element, we write $\widetilde{T}_{\lambda,r}$ instead of $\widetilde{T}_{t_\lambda, r}$.

For any $\lambda \in \cochars{\torus}$, there exist dominant $\lambda_1, \lambda_2 \in \cochars{T}$ such that $\lambda = \lambda_1 - \lambda_2$. Let $\Theta_{\lambda,r} = \widetilde{T}_{\lambda_1, r} \widetilde{T}_{\lambda_2, r}^{-1}.$ This element is independent of the choice of $\lambda_1$ and $\lambda_2$. 

\begin{defin}
For $\lambda \in \cochars{T}$, the \textbf{Bernstein function} attached to a Weyl orbit $M_\chi = W_\chi\cdot\lambda$ is defined by
$$
z_{M_\chi,r} = \sum_{\eta \in M_\chi} \Theta_{\eta,r}.
$$
If $\lambda$ is dominant, we denote the function associated to $M_\chi = W_\chi\cdot\lambda$ by $z_{\lambda,r}$.
\end{defin}

\begin{thm}
The center of $\mathcal{H}(H_{\chi_r}, I_{H_r})$ is the free $\mathbb{Z}[q^{r/2}, q^{-r/2}]$-module generated by the elements $z_{\lambda, r}$, where $\lambda$ ranges over all dominant cocharacters of $\torus$.
\end{thm}

\begin{proof}
The theorem is due to Bernstein, but this statement is \cite{haines2000}, Theorem 2.3, adapted to the algebra $\mathcal{H}(H_{\chi_r}, I_{H_r})$.
\end{proof}


\begin{defin}
\label{defin::mu-admissible-set}
Let $\mu$ be a dominant coweight of $\Phi$. An element $w \in \widetilde{W}$ is called \textbf{$\mu$-admissible} if $x \leq t_\lambda$ for some $\lambda$ in the $W$-orbit of $\mu$.
\end{defin}

When $\Phi$ is the root system of $G_r$, the $\mu$-admissible set is denoted ${\rm{Adm}}_{G_r}(\mu)$. Similarly, if $\mu_\chi$ is a dominant coweight of $\Phi_\chi$, the root system of $H_{\chi_r}$ under our hypotheses, elements $w \in \widetilde{W}_\chi$ such that $w \leq_\chi t_\lambda$ for $\lambda \in W_\chi \mu_\chi$ form the set ${\rm{Adm}}_{H_r}(\mu_\chi)$.

\begin{prop}
\label{prop::supp-bern-fns}
Let $\mu_\chi$ be a dominant, minuscule cocharacter of $\torus$ with respect to $\Phi_\chi$. The support of $z_{\mu_\chi, r}$ equals ${\rm{Adm}}_{H_r}(\mu_\chi)$ as subsets of $\widetilde{W}_\chi$.
\end{prop}

\begin{proof}
Apply Proposition 4.6 of~\cite{haines2000b} to the Iwahori-Hecke algebra $\mathcal{H}(H_{\chi_r}, I_{H_r})$, which has a basis indexed by $\widetilde{W}_\chi$, the extended affine Weyl group of $H_{\chi_r}$.
\end{proof}

For any $w \in \widetilde{W}$, let $\lambda(w)$ denote the cocharacter of $\torus$ such that $w = t_{\lambda(w)} \bar{w}$ via the isomorphism $\widetilde{W} \cong \cochars{\torus} \rtimes W$.

\begin{prop}
Suppose $\mu$ is a minuscule cocharacter and $w = t_{\lambda(w)} \bar{w} \in \widetilde{W}$ is $\mu$-admissible. Then $w \leq t_{\lambda(w)}$.
\end{prop}

\begin{proof}
This is \cite{haines-pettet2002}, Corollary 3.5.
\end{proof}

\begin{thm}
\label{thm::bernstein-fn-formula}
Let $\mu_\chi$ be a dominant minuscule cocharacter of $\torus$ with respect to $\Phi_\chi$. Set $Q_r = q^{-r/2} - q^{r/2}$. The Bernstein function $z_{\mu_\chi, r}$ in the center of $\mathcal{H}(H_{\chi_r}, I_{H_r})$  is given by the formula,
$$
z_{\mu_\chi, r} = \sum_{w \in {\rm{Adm}}_{H_r}(\mu_\chi)} \widetilde{R}^\chi_{w, t_{\lambda(w)}}(Q_r) \tilde{T}_{w,r}.
$$
\end{thm}

\begin{proof}
This is Theorem 4.3 of\cite{haines2000b} applied in the case of $\mathcal{H}(H_{\chi_r}, I_{H_r})$. A different proof appears in~\cite{haines-pettet2002}.
\end{proof}
\subsubsection{An explicit formula via Bernstein functions}

The material on Hecke algebra isomorphisms and Bernstein functions developed in earlier sections can be used to give an explicit formula for the coefficients of $\phironeaug$. We begin by expressing the image of $Z_{V_\mu} \ast e_{\chi_r}$ under the Hecke algebra isomorphism $\Psi_{\breve{\chi}_r}$ in terms of Bernstein functions in the center of $\mathcal{H}(H_{\chi_r}, I_{H_r})$.

\begin{lemma}
\label{lemma::chi-weights-disjoint-union}
Let $\rm{Wt}(\chi) = \{ \lambda \in W\mu \mid \lambda(\kapchi) = 1\}$. Let $\{\mu_\chi^i\}$ denote the subset of $W\mu$ consisting of elements which are dominant and minuscule with respect to $\Phi_\chi$. Then
$\rm{Wt}(\chi)$ is the disjoint union of the orbits $W_\chi \mu_\chi^i$.
\end{lemma}

\begin{proof}
The orbits $W_\chi \mu_\chi^i$ are all disjoint, because there is a unique dominant element in each $W_\chi$-orbit. We have $W\mu \supseteq \coprod_i W_\chi \mu_\chi^i$, because $\mu$ is among the $\mu_\chi^i$ and $W_\chi \subseteq W$. On the other hand, given an element $\eta \in W\mu$, there must exist some index $k$ such that $\eta \in W_\chi \mu_\chi^k$; just consider the orbit $W_\chi \eta$, which must contain a unique element dominant with respect to $\Phi_\chi$. So $W\mu = \coprod_i W_\chi \mu_\chi^i$.

Now pick any $\lambda \in {\rm{Wt}}(\chi)$. Then $\lambda \in W_\chi \mu_\chi^k$ for some $k$. For any other $\eta \in W_\chi \mu_\chi^k$, there is $u \in W_\chi$ such that $\eta = u^{-1}\lambda$. But $u \in W_\chi$ satisfies $u\kappa_\chi = \kappa_\chi$, so:
$$
\eta (\kappa_\chi) = u^{-1}\lambda(\kappa_\chi) = \lambda(u \kappa_\chi) = \lambda(\kappa_\chi) = 1.
$$
This shows that any $W_\chi$-orbit containing an element of ${\rm{Wt}}(\chi)$ lies entirely within ${\rm{Wt}}(\chi)$. It follows that ${\rm{Wt}}(\chi)$ is a finite, disjoint union of $W_\chi$-orbits.
\end{proof}

\begin{lemma}
\label{lemma:mu-chi-admissible-sets-disjoint}
Suppose ${\rm{Wt}}(\chi) = \coprod_i W_\chi \mu_\chi^i$ as in Lemma~\ref{lemma::chi-weights-disjoint-union}. Then the $\mu_\chi^i$-admissible sets ${\rm{Adm}}_{H_{\chi_r}}(\mu_\chi^i)$ are disjoint subsets of $\widetilde{W}$.
\end{lemma}

\begin{proof}
The translation elements $t_\lambda$ for $\lambda \in W\mu$ may be written $t_\lambda = \sigma_\lambda \bar{w}_\lambda \in \widetilde{W}$ for some length-zero element $\sigma_\lambda$ and $\bar{w}_\lambda \in W$. The $\sigma_\lambda$ are unique.

Choose any $\mu_\chi^i$. It is dominant and minuscule with respect to $\Phi_\chi$, hence any $w \in {\rm{Adm}}_{H_{\chi_r}}(\mu_\chi^i)$ satisfies $w \leq_\chi t_{\lambda(w)}$. By definition of Bruhat order on extended affine Weyl groups, this means the length-zero part of $w$ is $\sigma_{\lambda(w)}$. But since the $\sigma_\lambda$ are unique and the $W_\chi \mu_\chi^i$ orbits are disjoint, we cannot have $w \in {\rm{Adm}}_{H_{\chi_r}}(\mu_\chi^k)$ for any $k \neq i$.
\end{proof}

We shall sometimes speak of $w \in \widetilde{W}$ as though it were an element of $G_r$ by using a set-theoretic embedding of the extended affine Weyl group into $N_G(T)(F_r)$, which depends on a choice of uniformizer $\varpi$ in $F_r$. The following definition comes from \cite{haines-rapoport2012}, Section 2.

\begin{defin}
\label{defin::set-theoretic-embedding}
For $w \in \widetilde{W}$, we have $w = t_\lambda \bar{w}$ for $\lambda \in \cochars{\torus}$ and $\bar{w} \in W$. Let $\varpi$ be the fixed uniformizer in $F_r$, and set $\varpi^\lambda = \lambda(\varpi)$. The set-theoretic embedding $i_\varpi : \widetilde{W} \hookrightarrow G_r$ is defined by $t_\lambda \mapsto \varpi^{-\lambda}$ and mapping $\bar{w}$ to a fixed representative in $N_{G(\mathcal{O}_r)} (\torus(\mathcal{O}_r))$.
\end{defin}

\begin{lemma}
Let $I_r$ be an Iwahori subgroup of $G_r$ with pro-unipotent radical $I_r^+$. Then $G_r$ decomposes into disjoint double-cosets indexed by pairs $(s,w) \in \torus(k_r) \times \widetilde{W}$,
$$
G_r = \coprod_{(s,w)} I_r^+ sw I_r^+.
$$
\end{lemma}

\begin{proof}
This is a variant of the Bruhat-Tits decomposition which describes $G_r$ in terms of $I_r$-double cosets. The proof of this decomposition uses the relation $\torus(k_r) \cong I_r/I_r^+$; see \cite{haines-rapoport2012}, Section 2, for an explanation. Both the statement and proof depend on the set-theoretic embedding of $\widetilde{W}$ into $G(F_r)$ given in Definition~\ref{defin::set-theoretic-embedding}.
\end{proof}

I thank my advisor, Thomas Haines, for supplying the following lemma and its proof.

\begin{lemma}
\label{lemma::image-of-phi-r-chi}
Let $\chi_r$ be a depth-zero character on $\torus(\mathcal{O}_r)$ arising from a depth-zero character $\chi$ on $\torus(\mathcal{O}_F)$, and fix the $\varpi$-canonical extensions $\tilde{\chi}_r^\varpi$ and $\breve{\chi}_r$ to $\torus(F_r)$ and $N_G(T)(F_r)$ respectively. Let $\varpi^{-\lambda} \bar{w}$ be an element of $N_G(T)(F_r)$, which is the image of $w = t_\lambda \bar{w}$ under the set-theoretic embedding $\widetilde{W} \hookrightarrow G(F_r)$. For simplicity, assume $W_{\chi_r} = W_{\chi_r}^\circ$.
\begin{enumerate}
\item If $\varpi^{-\lambda} \bar{w}$ lies in the support of any non-zero function $\phi \in \mathcal{H}(G_r, I_r, \rho_{\chi_r})$, then ${^{\bar{w}}}\chi = \chi$, or equivalently, $\bar{w} \kappa_\chi = \kappa_\chi$.
\item Suppose $\mu \in \cochars{\torus} = \chars{\dualtorus}$ is dominant and minuscule. Under the Hecke algebra isomorphism $\Psi_{\breve{\chi}_r}$, $Z_{V_\mu} \ast e_{\chi_r}$ goes to the sum $\sum_{\mu_\chi} z_{\mu_\chi, r}$, where $\mu_\chi$ ranges over dominant representatives of $W_\chi$-orbits of $\nu \in W\mu$ such that $\nu(\kappa_\chi) = 1$.
\item If $\varpi^{-\lambda} \bar{w}$ lies in the support of any function of the form $Z_{V_\mu} \ast e_{\chi_r}$ then $\lambda(\kappa_\chi) = 1$.
\item If $Z_{V_\mu} \ast e_{\chi_r} (\varpi^{-\lambda}\bar{w}) \neq 0$, then $t_\lambda \bar{w} \in {\rm{Adm}}_{G_r}(\mu)$.
\end{enumerate}
\end{lemma}

\begin{proof} (T. Haines) To prove the first statement, we may assume $\phi = [I_r \varpi^{-\lambda} \bar{w} I_r]_{\breve{\chi}_r}$. For any $t \in \torus(\mathcal{O}_r)$ we have
$$
\chi_r^{-1}(t) \phi(\varpi^{-\lambda} \bar{w}) = \phi(\varpi^{-\lambda} \bar{w} ({^{\bar{w}^{-1}}}t)) = {^{\bar{w}}}\chi_r^{-1} (t) \phi(\varpi^{-\lambda} \bar{w}).
$$
Since $\phi(\varpi^{-\lambda} \bar{w}) \neq 0$, we conclude $\chi_r (t) = {^{\bar{w}}}\chi_r (t)$ for all $t \in \torus(\mathcal{O}_r)$, i.e., ${^{\bar{w}}}\chi_r = \chi_r$.

Now let us prove the second statement. The function $Z_{V_\mu} \ast e_{\chi_r}$ lies in the center of the Hecke algebra $\mathcal{H}(G_r, I_r, \rho_{\chi_r})$ and acts by a non-zero scalar only on the $\rho_{\chi_r}$-isotypical components of representations in the Bernstein block for the inertial class $[T, \tilde{\chi}_r^\varpi]_G$. In fact, for any unramified character $\eta$ on $\torus(F_r)$, unwinding the definition shows that it acts on $i_{B_r}^{G_r} (\tilde{\chi}_r^\varpi)^{\rho_{\chi_r}}$ by the scalar
$$
{\rm{Tr}}(r_\mu \varphi_{\tilde{\chi}_r^\varpi \eta} (\Phi), V_\mu^{\kappa_\chi}).
$$
By~\cite{haines-rapoport2012}, Lemma 10.1.1, the image $\Psi_{\breve{\chi}_r} (Z_{V_\mu} \ast e_{\chi_r})$ acts by this scalar on the Iwahori-fixed vectors $i_{B_{H_r}}^{H_{\chi_r}} (\eta)^{I_{H_r}}$ in the unramified principal series of $H_{\chi_r}$. In order to prove $\Psi_{\breve{\chi}_r} (Z_{V_\mu} \ast e_{\chi_r}) = \sum_{\mu_\chi} z_{\mu_\chi, r}$ , it will suffice to show that $\sum_{\mu_\chi} z_{\mu_\chi, r}$ acts by this same scalar on such representations, in which case the functions are equal because they determine the same regular function on the Bernstein variety $\mathfrak{X}_{H_r}$.

As $\widehat{T}$-representations we have $V_\mu^{\kappa_\chi} = \oplus_{\mu_\chi} V_{\mu_\chi}^{H_{\chi_r}}$, the sum of highest weight representations for the dual group of $H_{\chi_r}$, hence the scalar can be written as
$$
\sum_{\mu_\chi} {\rm{Tr}}(r_{\mu_\chi} \varphi_{\tilde{\chi}_r^\varpi \eta} (\Phi), V_{\mu_\chi}^H).
$$
Note that $\varphi_{\tilde{\chi}_r^\varpi \eta} (\Phi) = \varphi_\eta (\Phi)$. View the function $z_{\mu_\chi}$ as a regular function on the variety $\dualtorus/W_\chi$, whose points correspond to $W_\chi$-invariant unramified characters,
$$
z_{\mu_\chi,r} : \eta \mapsto \sum_{\lambda \in W_\chi \mu_\chi} \lambda(\eta).
$$
Looking at the weight space decomposition, $\sum_{\lambda \in W_\chi \mu_\chi} \lambda(\eta) = {\rm{Tr}}(r_{\mu_\chi} \varphi_{\eta} (\Phi), V_{\mu_\chi}^H)$. Summing over the $\mu_\chi$ shows that $\sum_{\mu_\chi} z_{\mu_\chi, r}(\eta) = {\rm{Tr}}(r_\mu \varphi_{\tilde{\chi}_r^\varpi \eta} (\Phi), V_\mu^{\kappa_\chi}).$

Now use the second statement and the fact that $\Psi_{\breve{\chi}_r}$ is support-preserving to see that $\varpi^{-\lambda} \bar{w}$ must lie in the support of some $z_{\mu_\chi, r}$. Since $\mu_\chi$ is minuscule, we must have in that case that $\lambda \in W_\chi \mu_\chi$. But then $\lambda(\kappa_\chi) = 1$ by Lemma~\ref{lemma::chi-weights-disjoint-union}.

Finally, suppose that $Z_{V_\mu} \ast e_{\chi_r} (\varpi^{-\lambda}\bar{w}) \neq 0$. Using earlier work and the support-preserving property of Hecke algebra isomorphisms, we have that $z_{\mu_\chi, r} (w) \neq 0$ for some $\mu_\chi$. The support of $z_{\mu_\chi,r}$ is ${\rm{Adm}}_{H_{\chi_r}}(\mu_\chi)$, i.e., $w \in {\rm{Adm}}_{H_{\chi_r}}(\mu_\chi)$. So it is enough to show that ${\rm{Adm}}_{H_{\chi_r}}(\mu_\chi) \subset {\rm{Adm}}_{G_r}(\mu)$ as subsets of $\widetilde{W}$.

The base alcove for a (based) root system $\Phi$ is the set of points in $\cochars{\torus} \otimes \mathbb{R}$ on which all positive affine roots take positive values. Since the positive affine roots attached to $\Phi_\chi$ are a subset of the positive affine roots for $\Phi$, we deduce an inclusion of base alcoves $\textbf{a} \subset \textbf{a}_\chi$ associated to $\Phi$, and $\Phi_\chi$, respectively. Furthermore, every alcove for $\Phi$ is contained in a unique alcove for $\Phi_\chi$.

Given $w = t_\lambda \bar{w} \in {\rm{Adm}}_{H_{\chi_r}}(\mu_\chi)$, we know that $w \leq_\chi t_\lambda$, thus there is a sequence of alcoves $w\textbf{a}_\chi = \textbf{a}_\chi^{0}, \textbf{a}_\chi^{1}, \ldots, \textbf{a}_\chi^{l} = t_\lambda \textbf{a}_\chi$, such that if $\textbf{a}_\chi^{i-1}$ and $\textbf{a}_\chi^{i}$ are separated by an affine hyperplane $H$, then $\textbf{a}_\chi$ and $\textbf{a}_\chi^{i-1}$ are on the same side of $H$. From this we get a sequence of alcoves $w\textbf{a} = \textbf{a}^{0}, \textbf{a}^{0}, \ldots, \textbf{a}^{l} = t_\lambda \textbf{a}$, such that whenever $\textbf{a}^{i-1}$ and $\textbf{a}^{i}$ are separated by an affine hyperplane $H$, then $\textbf{a}$ and $\textbf{a}^{i-1}$ are on the same side of $H$. It follows that $w \leq t_\lambda$, which implies $w \in {\rm{Adm}}_{G_r}(\mu)$.
\end{proof}

The function $\phironeaug$ is an element of the Hecke algebra $\mathcal{H}(G_r, I_r^+)$, hence it can  be described by specifying the complex values taken on $I_r^+$-double cosets of $G_r$ indexed by pairs in $\torus(k_r) \times \widetilde{W}$. The values $\phironeaug(I_r^+ sw I_r^+)$ are called the \textbf{coefficients} of $\phironeaug$. The following proposition gives an explicit formula for the coefficients of $\phironeaug$, which is sufficient for our purposes by Corollary~\ref{cor::equality-of-orbital-integrals}. We will revisit this proposition in Chapter 5, where we use the results of the intervening chapters to develop a combinatorial formula using the following formula as a starting point.

\begin{prop}
\label{prop::first-explicit-formula}
Given a pair $(s,w) \in \torus(k_r) \times \widetilde{W}$, the coefficient $\phironeaug (I_r^+ sw I_r^+)$ can be rewritten as a sum over endoscopic elements in $\dualtorus(\valfield)$ which arise from depth-zero characters $\chi \in \torus(\resfield)^\vee$:
$$
\phironeaug (I_r^+ sw I_r^+) = [I_r: I_r^+]^{-1} \sum_{\kapchi \in K_{q-1}} \gamma_{N_r s}(\kapchi)^{-1} q^{r\ell(w,t_{\lambda(w)})/2} \widetilde{R}_{w,t_{\lambda(w)}}^\chi(Q_r).
$$
\end{prop}

\begin{proof}
The proof is a straightforward application of the work done throughout this chapter. We start with the definition of $\phironeaug$ from Corollary~\ref{cor::equality-of-orbital-integrals},
$$
\phironeaug = [I_r : I_r^+]^{-1} q^{r\ell(t_\mu)/2} \sum_{\chi \in \torus(\resfield)^\vee} Z_{V_\mu} \ast e_{\chi_r}.
$$
Then by Lemma~\ref{lemma::image-of-phi-r-chi},
$$
\phironeaug = [I_r : I_r^+]^{-1} q^{r\ell(t_\mu)/2} \sum_{\chi \in \torus(\resfield)^\vee} \sum_{\mu_\chi} \Psi_{\breve{\chi}_r}^{-1} (z_{\mu_\chi, r}).
$$

The orbits $W_\chi \mu_\chi$ are disjoint as $\mu_\chi$ ranges over the set of elements in $W\mu$ which are dominant and minuscule with respect to $\Phi_\chi$. The $\mu_\chi$-admissible sets in $\widetilde{W}$ are disjoint by Lemma~\ref{lemma:mu-chi-admissible-sets-disjoint}, hence the supports of each $z_{\mu_\chi, r}$ are disjoint. Thus for each $w \in {\rm{Adm}}_{G_r}(\mu)$, there is a unique $\mu_\chi \in W\mu$ such that $w \in \rm{supp}(z_{\mu_\chi, r})$. Therefore,
$$
\phironeaug (I_r^+ sw I_r^+) = [I_r : I_r^+]^{-1} q^{r\ell(t_\mu)/2} \sum_{\chi \in \torus(\resfield)^\vee} \Psi_{\breve{\chi}_r}^{-1} (z_{\mu_\chi, r})(sw).
$$
Now, we apply Haines's formula for $z_{\mu_\chi, r}$ and Lemma~\ref{lemma::hecke-algebra-mapping-coefficients}. Recall that
$$
z_{\mu_\chi, r} (w) = \widetilde{R}_{w, t_{\lambda(w)}}^\chi (Q_r) \widetilde{T}_{w,r}(w),
$$
and $\widetilde{T}_{w,r}(w) = q^{-r\ell_\chi(w)/2} [I_{H_r} w I_{H_r}] (w).$ Therefore,
$$
\Psi_{\breve{\chi}_r}^{-1} (z_{\mu_\chi, r})(sw) = q^{-r\ell(w)/2 + r\ell_\chi(w)/2} q^{-r\ell_\chi(w)/2} \widetilde{R}_{w, t_{\lambda(w)}}^\chi (Q_r) [I_r^+ sw I_r^+]_{\breve{\chi}_r} (sw).
$$
But $[I_r^+ sw I_r^+]_{\breve{\chi}_r} (sw) = \chi^{-1}_r(s)$. So, we conclude that
$$
\phironeaug (I_r^+ sw I_r^+) = [I_r : I_r^+] ^{-1} \sum_{\chi \in \torus(\resfield)^\vee} \chi_r^{-1}(s) q^{r(\ell(t_{\lambda(w)})-\ell(w))/2}\widetilde{R}_{w, t_{\lambda(w)}}^\chi (Q_r),
$$
where we have used that the conjugates of a translation element all have the same length; that is, $\ell(t_\mu) = \ell(t_{\lambda(w)})$ for all $\lambda \in W\mu$.

The set of depth-zero characters $\torus(\resfield)^\vee$ is in bijective correspondence with the set of endoscopic elements $\kapchi \in \dualtorus(\valfield)$ arising from depth-zero characters; this bijection is determined by a fixed element of the inertia subgroup $x \in I_F$ whose image $\tau_F(x)$ projects to a generator of $k_F^\ast$. Proposition~\ref{prop::dz-endoscopic-elements-equal-kernel} shows that this subset of $\dualtorus(\valfield)$ equals $K_{q-1}$.

Following the notation in \cite{bjorner-brenti2005}, we write $\ell(w, t_{\lambda(w)})$ for the difference in lengths $\ell(t_{\lambda(w)}) - \ell(w)$.

Finally, for each $s \in \torus(k_r)$ the equation $\gamma_{N_r s}(\kapchi) = \chi_r (s)$ holds for all depth-zero characters $\chi$. In conclusion,
$$
\phironeaug (I_r^+ sw I_r^+) = [I_r : I_r^+]^{-1} \sum_{\kapchi \in K_{q-1}} \gamma_{N_r s} (\kapchi)^{-1} q^{r\ell(w, t_{\lambda(w)})/2} \widetilde{R}_{w, t_{\lambda(w)}}^\chi (Q_r)
$$
\end{proof}


\section{Groups of endoscopic elements in the dual torus}

Consider the endoscopic elements $\kapchi$ in $\dualtorus(\valfield)$ arising from depth-zero characters on $\torus(\mathcal{O}_F)$, which index the summation in our formula for $\phironeaug$. We shall see that it is useful to determine the $\kappa_\chi$ such that $Z_{V_\mu} \ast e_{\chi_r} (w) \neq 0$ for a fixed $w \in \widetilde{W}$.

We begin by defining a closed subgroup $S_w \subset \dualtorus(\valfield)$ associated to a fixed $w = t_{\lambda} \wbar \in \widetilde{W}$. This ``relevant subgroup'' is an \emph{infinite} diagonalizable algebraic group. In exchange for working with a more complex object, we gain access to the theory of diagonalizable groups and tori defined over an algebraically closed field. The endoscopic elements needed for the combinatorial formula for $\phironeaug$ comprise the ``depth-zero relevant subgroup'' $\dzrelgrp$, which is a finite subgroup of $\dualtorus(\valfield)$.

The final section of this chapter contains results to be used later in the statement and proof of the main theorem in Chapter~\ref{chapter::main-theorem}. First, we identify an analogue of the ``critical index torus'' used by Haines and Rapoport in the Drinfeld case; this object is used to determine which $s \in \torus(k_r)$ contribute to nontrivial coefficients $\phironeaug(I_r^+ sw I_r^+)$ in terms of the $k_F$-points of a subtorus of $T$. Second, we look at the order of certain subgroups ${\dzrelgrpJ \subset \dzrelgrp}$ which arise from root sub-systems $J \subseteq \Phi$.


\subsection{Background on diagonalizable algebraic groups}
\label{section::diag-groups-background}

Let us recall various properties of diagonalizable algebraic groups over an algebraically closed field, that is, linear algebraic groups which are isomorphic to a closed subgroup of the diagonal torus in some general linear group. In fact, we will only consider such groups over $\mathbb{C}$.

\begin{defin}
An algebraic group over $\mathbb{C}$ is \textbf{diagonalizable} if it isomorphic to a closed subgroup of the diagonal group $D(n, \mathbb{C})$ for some $n$.
\end{defin}

Recall that a connected diagonalizable group is a torus.

\begin{thm}
Let $G$ be a diagonalizable group over $\mathbb{C}$. Then $G = A \times H$, where $A$ is a torus over $\mathbb{C}$ and $H$ is a finite group.
\end{thm}

\begin{proof}
See \cite{humphreys1975}, Section 16.2.
\end{proof}

Let $D$ be a diagonalizable group defined over $\mathbb{C}$. Recall that the multiplicative group $\mathbb{G}_m$ consists of the nonzero elements of the affine space $\mathbb{A}^1$ equipped with the group law $(x,y) \mapsto xy$. The group of $\mathbb{C}$-rational points of $\mathbb{G}_m$ is $\mathbb{C}^\times$. The \textbf{character group} of $D$ is defined by $\chars{D} = {\rm{Hom}}_{\mathbb{C}}(D, \mathbb{C}^\times)$, while its \textbf{cocharacter group} is $\cochars{D} = {\rm{Hom}}_{\mathbb{C}}(\mathbb{C}^\times, D)$.

\begin{defin}
A \textbf{lattice} is a free subgroup of $\chars{D}$ or $\cochars{D}$ generated over a linearly independent set.
\end{defin}

When considering a torus we will sometimes choose a basis for the cocharacter group $\cochars{T}$ and then give coordinates in terms of that basis. For example, we write $\mu = (1,0, \ldots, 0)$ for the cocharacter $\mu$ of the diagonal torus in $GL_d$ given by the formula
$$
\mu(z) = {\rm{diag}}(z,1,\ldots,1).
$$

\begin{thm}
\label{thm::cat-antiequiv}
There is a categorical anti-equivalence between diagonalizable algebraic groups and abelian groups, which arises from the contravariant functor sending a diagonalizable group $D$ to its character group $\chars{D}$.
\end{thm}

\begin{proof}
See \cite{waterhouse1979}, Section 2.2, for example.
\end{proof}

The following corollaries follow directly from Theorem~\ref{thm::cat-antiequiv}.

\begin{cor}
Let $D$ be a diagonalizable algebraic group over $\mathbb{C}$. If its character group $\chars{D}$ has torsion, then $D$ is not connected, i.e., $D$ is not a torus.
\end{cor}

\begin{cor}
\label{cor::contain-by-annihilation}
Let $D$ be a diagonalizable subgroup of $\dualtorus(\valfield)$, and let $L$ be the lattice in $\chars{\dualtorus}$ such that $\chars{D} = \chars{\dualtorus}/L$. Then $\kappa \in \dualtorus(\valfield)$ is annihilated by $L$ if and only if $\kappa \in D$.
\end{cor}

We conclude this background section with two useful lemmas.

\begin{lemma}
\label{lemma::chars-separate-points}
Let $D$ be a diagonalizable group. Given two distinct points $x$ and $y$ in $D$, there exists a character $\eta \in \chars{D}$ such that $\eta(x) \neq \eta(y)$.
\end{lemma}

\begin{proof}
See \cite{humphreys1975}, Section 16.1.
\end{proof}

\begin{lemma}
\label{lemma::chars-of-intersected-groups}
Suppose $D_1$ and $D_2$ are diagonalizable subgroups of $\widehat{\torus}(\valfield)$ such that $\chars{D_i} = \chars{\dualtorus}/L_i$ for lattices $L_1$ and $L_2$. Then the character group of $D_1 \cap D_2$ is
$$
\chars{D_1 \cap D_2} = \chars{\dualtorus}/\langle L_1, L_2 \rangle.
$$
\end{lemma}

\begin{proof}
The group $D_1 \cap D_2$ is the largest common subgroup of $D_1$ and $D_2$. Under the categorical anti-equivalence between diagonalizable groups and their character groups, its character group $\chars{D_1 \cap D_2}$ is the largest common quotient of $\chars{D_1}$ and $\chars{D_2}$. But as $\langle L_1, L_2\rangle$ is the smallest lattice containing both $L_1$ and $L_2$, we conclude $\chars{D_1 \cap D_2} = \chars{\dualtorus}/\langle L_1, L_2\rangle$.
\end{proof}


\subsection{The relevant group of an admissible element}

For any $w \in \widetilde{W}$, there is an expression $w = t_{\lambda} \wbar$ obtainable via the isomorphism $\widetilde{W} \cong \cochars{\torus} \rtimes W$. We give certain conditions based on the data $\lambda$ and $\wbar$ to define a diagonalizable group $\relgrp$ in $\dualtorus(\valfield)$. This infinite group contains a finite subgroup $\dzrelgrp$ consisting of those $\kappa \in \relgrp$ such that $\kappa = \kapchi$ for some depth-zero character $\chi$ on $\torus(\mathcal{O}_F)$. Given a root sub-system $J \subseteq \Phi$, there are analogous groups $\relgrpJ$ and $\dzrelgrpJ$. We specify lattices $L_{w,J} \subset \chars{\dualtorus}$ such that $\chars{S_{w,J}} = \chars{\dualtorus}/L_{w,J}$, thereby giving information about $S_{w,J}$ through Theorem~\ref{thm::cat-antiequiv}.

\subsubsection{Definition of $S_w$ and $S_{w,J}$}

\begin{lemma}
Let $\bar{w} \in W$. Then $\kappa_{^{\wbar}\chi} = \wbar\kapchi$ for all characters $\chi$ on $\torus(\mathcal{O}_F)$. If in addition $\wbar \in W_\chi$, then $\wbar\kapchi = \kapchi$.
\end{lemma}

\begin{proof}
Given $\wbar \in W$, the LLC for Tori implies ${^{\wbar}} \chi (\nu(\tau_F(x))) = \nu(\kappa_{{^{\wbar}} \chi})$ for all $\nu \in \cochars{\torus} = \chars{\torus}$. On the other hand,
$$
{^{\wbar}}\chi (\nu(\tau_F(x))) = \chi (\wbar^{-1} \nu(\tau_F(x))) = (\wbar^{-1}\cdot \nu)(\kapchi) = \nu(\wbar\kapchi).
$$
Thus $\nu(\wbar\kapchi) = \nu(\kappa_{{^{\wbar}}\chi})$ for all $\nu$. We conclude that $\wbar\kapchi = \kappa_{{^{\wbar}}\chi}$ by Lemma~\ref{lemma::chars-separate-points}.

For the final statement, apply the definition $W_\chi = \{\bar{w} \in W \mid {^{\bar{w}}\chi} = \chi\}$.
\end{proof}

\begin{defin}
\label{defin::relevant-endoscopic-elements}
An endoscopic element $\kappa \in \dualtorus(\valfield)$ is \textbf{relevant} to $w \in {\widetilde{W}}$ if $\wbar\kappa = \kappa$ and $\lambda(\kappa) = 1$. 
\end{defin}

\begin{prop}
\label{prop::relevant-group}
The elements $\kappa \in \dualtorus(\valfield)$ relevant to a fixed $w \in \widetilde{W}$ form a closed subgroup called the \textbf{relevant subgroup} $S_w$.
\end{prop}

\begin{proof}
First, we verify that relevant $\kappa$ satisfy the group axioms. It is obvious that $\kappa = 1$, corresponding to the trivial character, belongs to $S_w$ for any $w \in \widetilde{W}$. If ${\kappa}_1, {\kappa}_2 \in S_w$, then it is enough to observe that
$$
\wbar\cdot ({\kappa}_1 {\kappa}_2) = (\wbar\cdot \kappa_1)(\wbar\cdot\kappa_2) = {\kappa}_1{\kappa}_2
$$
and
$$
\lambda({\kappa}_1{\kappa}_2) = \lambda({\kappa}_1)\lambda({\kappa}_2) = 1.
$$
In particular, $\wbar\cdot \kappa^{-1} = \kappa^{-1}$.

This subgroup is \emph{closed} because it is the intersection of $\rm{ker}(\lambda)$ and the fixed-point set of $\wbar$, each of which is a Zariski-closed subset of $\dualtorus(\valfield)$.
\end{proof}

The relevant subgroup for an element $w$ is not a torus in general. See Example~\ref{example::rel-group-not-a-torus} which uses an explicit realization of the character group $\chars{\relgrp}$ to produce torsion elements.


\begin{lemma}
\label{lemma::support-of-phi-r-chi}
Let $w \in \widetilde{W}$. Then $Z_{V_\mu} \ast e_{\chi_r} (w) \neq 0$ implies $w \in {\rm{Adm}}_{G_r}(\mu)$ and that $\kapchi$ is relevant to $w$.
\end{lemma}

\begin{proof}
This is a rephrasing of two statements proved in Lemma~\ref{lemma::image-of-phi-r-chi}.

\end{proof}

\begin{defin}
The subgroup of $\relgrp$ consisting of $\kapchi$ associated to $\chi \in \torus(\resfield)^\vee$ will be called the \textbf{depth-zero relevant subgroup} of $w$. It is denoted $\dzrelgrp$.
\end{defin}

Lemma~\ref{lemma::support-of-phi-r-chi} says that if $Z_{V_\mu} \ast e_{\chi_r}(w) \neq 0$ then $w$ is $\mu$-admissible and $\kappa_\chi \in \dzrelgrp$.

\begin{defin}
\label{defin::dzrelgrpJ}
Let $J$ be a root sub-system of $\Phi$. For $\alpha \in J$, consider $\alpha^\vee$ as a character on $\dualtorus$. Define $S_{w,J} \subseteq S_w$ by
$$
S_{w,J} = S_w \cap \left(\bigcap_{\alpha \in J} \rm{ker}(\alpha^\vee)\right).
$$
Let $\dzrelgrpJ$ denote the subset of $\relgrpJ$ whose elements arise as endoscopic elements corresponding to $\chi \in \torus(\resfield)^\vee$.
\end{defin}

\begin{prop}
\label{prop::set-version-of-dzrelgrpJ}
Let $w \in \widetilde{W}$, and fix a sub-system $J \subseteq \Phi$. Then
$$
\dzrelgrpJ = \{ \kapchi \in \dzrelgrp \mid \Phi_\chi \supseteq J\}.
$$
\end{prop}

\begin{proof}
If $\kapchi \in \dzrelgrpJ$, then for all $\alpha \in J$ we have $\alpha^\vee (\kapchi) = 1$. It follows that $\chi \circ \alpha^\vee (\tilde{x}) = 1$. By definition this means $\alpha \in \Phi_\chi$, so in total, $J \subseteq \Phi_\chi$.

Conversely, if $\kapchi \in \dzrelgrp$ satisfies $\Phi_\chi \supseteq J$, then $\chi \circ \alpha^\vee (\tilde{x}) = 1$ for all $\alpha \in J$. Again, $\alpha^\vee (\kapchi) = 1$, so that $\kapchi \in \cap_{\alpha \in J} \rm{ker}(\alpha^\vee)$.
\end{proof}

\subsubsection{Lattices in character groups}
\label{section::lattices-in-char-groups}

The relevant subgroup $S_w$ will guide us to an analogue of the critical index torus defined for the Drinfeld case in~\cite{haines-rapoport2012}, Section 6.4. This comes about by realizing the character groups $\chars{\relgrp}$ and $\chars{\relgrpJ}$ as quotients of $\chars{\dualtorus}$ by specifying certain lattices.

The categorical anti-equivalence between diagonalizable groups and their character groups enables us to take advantage of the definition of $S_{w,J}$ as the intersection of two diagonalizable groups. Let $K_J = \bigcap_{\alpha \in J} \rm{ker}(\alpha^\vee)$. In the following diagram, arrows in the left diamond are inclusions while arrows in the right diamond are quotients.
{
\[
\xymatrix{
 & \dualtorus(\valfield)\\
S_w \ar[ur] && K_J \ar[ul] \\
 & S_{w,J} \ar[ul] \ar[ur]
}
\
\xymatrix{
 \\
 \longleftrightarrow \\
}
\
\xymatrix{
 & \chars{\dualtorus} \ar[dl] \ar[dr] \\
\chars{S_w} \ar[dr] && \chars{K_J} \ar[dl] \\
 & \chars{S_{w,J}}
}
\]
}

Quotients of $\chars{\dualtorus}$ correspond to the lattice used to form the quotient. The lattice needed to form the quotient $\chars{S_{w,J}}$ is the lattice generated by the lattices corresponding to $\chars{S_w}$ and $\chars{K_J}$.

\begin{lemma}
\label{lattice-equality-lemma}
Let $w = t_\lambda \bar{w}$ be a $\mu$-admissible element in $\widetilde{W}$. Define a lattice $L_w = \langle w(\nu) - \nu \mid \nu \in \chars{\dualtorus} \rangle$, where $w(\nu) = \lambda + \wbar(\nu)$. Then
$$L_w = \langle \lambda,\: \wbar (\nu) - \nu \mid \nu \in \chars{\dualtorus} \rangle.$$
\end{lemma}

\noindent The $L_w$ so defined is the same as the lattice studied in Section 6.4 of \cite{haines-rapoport2012}.

\begin{proof}
It immediately obvious that $L_w \subseteq  \langle \lambda,\: \wbar \nu - \nu \vert\ \nu \in \chars{\dualtorus} \rangle.$

Let us consider the reverse inclusion. If we choose $\nu = 0$, then $w(0) - 0 = \lambda$ belongs to $L_w$. On the other hand, choosing $\nu = \lambda$ yields $w(\lambda) - \lambda = \bar{w}\lambda$. For any choice of $\nu \in \chars{\dualtorus}$,
$$
w(\nu + \lambda) - (\nu + \lambda) = \bar{w}(\nu) + \bar{w}(\lambda) - \nu.
$$
But we have just seen that $\bar{w}(\lambda) \in L_w$, hence $\bar{w}(\nu) - \nu$ lies in $L_w$ for all $\nu \in \chars{\dualtorus}$. It follows that $L_w \supseteq  \langle \lambda,\: \wbar \nu - \nu \vert\ \nu \in \chars{\dualtorus} \rangle.$
\end{proof}

Let $A$ be a group and $M$ be an $A$-module. The \emph{module of coinvariants} of $M$ is defined as $$
M_A = M/\langle gm - m \mid m \in M,\; g \in A\rangle.
$$
This module satisfies the following universal property: If $N$ is another $A$-module with trivial $A$-action, and there is a surjection $M \rightarrow N$, then there is a unique map $M_A \rightarrow N$.

\begin{prop}
\label{characters-of-relevant-subgroup-proposition}
Let $w = t_\lambda \wbar$ be $\mu$-admissible. Then
$$\chars{\relgrp} = \chars{\dualtorus} / L_w = \chars{\dualtorus}_{\langle\wbar\rangle} / \langle \lambda \rangle.$$
\end{prop}

\begin{proof}
Let $K$ be the intersection of kernels
$$
K = {\rm{ker}}(\lambda) \cap \left(\bigcap_{\nu \in \chars{\dualtorus}} {\rm{ker}}(\bar{w}(\nu) - \nu)\right)
$$
We contend that this set equals $S_w$ in $\dualtorus(\valfield)$.

Recall that $S_w$ is defined as the subgroup of $\dualtorus(\valfield)$ comprising those $\kappa$ such that $\lambda(\kappa) = 1$ and $\bar{w} \kappa = \kappa$. Thus all $\kappa \in S_w$ lie in ${\rm{ker}}(\lambda)$, and for each $\nu \in \chars{\dualtorus}$,
$$
\Big((\bar{w}\cdot\nu)(\kappa)\Big)\nu(\kappa)^{-1} = \nu(\bar{w}^{-1}\kappa)\nu(\kappa)^{-1} = \nu(\kappa)\nu(\kappa)^{-1} = 1.
$$
So $\kappa \in {\rm{ker}}(\wbar(\nu)-\nu)$ for all $\nu \in \chars{\dualtorus}$. It follows that $S_w \subseteq K$.

Conversely, if we choose some $\kappa_0 \in K$, then $\lambda(\kappa_0) = 1$ from the definition, while for each character $\nu$ on $\dualtorus(\valfield)$, $\nu(\bar{w}^{-1} \kappa_0) = \nu(\kappa_0).$ By Lemma~\ref{lemma::chars-separate-points}, $\bar{w}^{-1} \kappa_0 = \kappa_0$. We conclude that $\kappa_0 \in S_w$. This proves $S_w = K$.

For a character $\eta \in \chars{\dualtorus}$, Lemma~\ref{lemma::kernel-lattice-relationship} implies $\chars{{\rm{ker}}\;\eta}= \chars{\dualtorus} / \langle\eta\rangle.$ Lemma~\ref{lemma::chars-of-intersected-groups} shows that the character group of an intersection of kernels equals the quotient of $\chars{\dualtorus}$ by the lattice generated by the characters whose kernels were intersected. In this case,
$$
\chars{K} = \chars{\dualtorus}/\langle \lambda, \bar{w}(\nu) - \nu \mid \nu \in \chars{\dualtorus}\rangle.
$$
Applying $K = S_w$ and Lemma~\ref{lattice-equality-lemma}, we get the desired result: $\chars{S_w} = \chars{\dualtorus}/L_w.$
\end{proof}

\begin{example}
\label{example::rel-group-not-a-torus}
Let $G = GL_4$ and $\mu = (1,1,0,0)$. Let $w = t_{\mu}(132)$, so that $\wbar = (132)$ is a permutation in the symmetric group $W \cong S_3$. Then the module of coinvariants $\cochars{\torus}_{\langle \wbar \rangle}$ is generated by $\bar{\varepsilon}_1, \bar{\varepsilon}_4$, where the $\varepsilon_i$ are the coordinate cocharacters of $\cochars{\torus}$, which form a basis, and $\bar{\varepsilon}_i$ is the image of $\varepsilon_i$ in the module of coinvariants. Consider
$$
\chars{S_w} = \cochars{T}_{\langle \wbar \rangle} / \langle \mu \rangle.
$$
Then as elements of $\chars{S_w}$,
$$
2\bar{\varepsilon}_1 = \bar{\varepsilon}_1 + \bar{\varepsilon}_2 = \bar{\mu} = 0.
$$
Thus $\chars{\relgrp}$ has 2-torsion in this case, and so $\relgrp$ is not a torus.
\end{example}

\begin{lemma}
\label{lemma::kernel-lattice-relationship}
For any root subsystem $J \subseteq \Phi$, let $K_J = \cap_{\alpha \in J} \rm{ker}(\alpha^\vee)$ and let $L_J = \mathbb{Z}\langle \alpha^\vee \mid \alpha \in J\rangle.$ Then $\chars{K_J} = \chars{\dualtorus}/L_J.$
\end{lemma}

\begin{proof}
Choose any $\alpha \in J$. This determines a short exact sequence
$$
1 \longrightarrow {\rm{ker}}(\alpha^\vee) \longrightarrow \dualtorus \longrightarrow \mathbb{G}_m \longrightarrow 1.
$$
The corresponding exact sequence for character groups is
$$
0 \longleftarrow \chars{{\rm{ker}}(\alpha^\vee)} \longleftarrow \chars{\dualtorus} \stackrel{f}{\longleftarrow} \mathbb{Z} \longleftarrow 0,
$$
where $f:1\mapsto \alpha^\vee$. It follows that $\chars{{\rm{ker}}(\alpha^\vee)} = \chars{\dualtorus}/\mathbb{Z}\alpha^\vee$.

Now consider all $\alpha \in J$ at once. Repeated application of Lemma~\ref{lemma::chars-of-intersected-groups} shows that $\chars{K_J} = \chars{\dualtorus}/L_J.$
\end{proof}

\begin{cor}
\label{cor::chars-of-swj}
Let $L_{w,J}$ be the lattice of $\chars{\dualtorus}$ generated by $L_w$ and $L_J$ as defined above. Then
$\chars{S_{w,J}} = \chars{\dualtorus}/L_{w,J}.$
\end{cor}

A root sub-system $J \subseteq \Phi$ determines a subgroup $W_J = \langle s_\alpha \mid \alpha \in J^+\rangle$ of $W$.

\begin{prop}
\label{prop::structure-of-lwj}
Let $w = t_\lambda \bar{w} \in {\rm{Adm}}_{G_r}(\mu)$ and $J \subseteq \Phi$. If $\bar{w} \in W_J$, then $L_{w,J} = \langle \lambda, \alpha^\vee \mid \alpha \in J^+\rangle$.
\end{prop}

\begin{proof}
By its construction, $L_{w,J} = \langle \lambda + \bar{w}(\nu) - \nu, \alpha^\vee \mid \nu \in \chars{\dualtorus}, \alpha \in J^+\rangle$. Lemma~\ref{lattice-equality-lemma} implies $L_{w,J} = \langle \lambda, \bar{w}(\nu) - \nu, \alpha^\vee \mid \nu \in \chars{\dualtorus}, \alpha \in J^+\rangle$, so that $L_{w,J}$ clearly contains $\langle \lambda, \alpha^\vee \mid \alpha \in J^+\rangle$.\

In order to prove the reverse inclusion, it is enough to show that $\bar{w}(\nu) - \nu$ lies in the span of the coroots of $J^+$ for all $\nu \in \chars{\dualtorus}$. 

Since $W_J$ is a reflection subgroup, we can find an expression $\bar{w} = s_1 \cdots s_m$ where the $s_i$ are reflections in $W_J$. Choose any $\nu \in \chars{\dualtorus}$. Observe that
$$
(s_1 \cdots s_m)(\nu) = (s_2 \cdots s_m)(\nu) - \langle (s_2\cdots s_m)(\nu), \alpha_1\rangle \alpha_1^\vee,
$$
where $\alpha_1$ is the positive root in $J$ corresponding the to reflection $s_1$. Of course, a similar formula can be applied to the term $(s_2 \cdots s_m)(\nu)$, leading to the expression
$$
\bar{w}(\nu) = \nu - \sum_{k=2}^m \langle (s_k \cdots s_m)(\nu), \alpha_{k-1}\rangle \alpha_{k-1}^\vee.
$$
Therefore, $\bar{w}(\nu) - \nu \in \langle \alpha^\vee \mid \alpha \in J^+ \rangle.$ 
\end{proof}

%

The \textbf{fundamental group} of a reductive algebraic group is defined by
$$
\pi_1(G) = \cochars{\torus}/\mathbb{Z}\Phi^\vee.
$$
A reductive algebraic group $H$ is \textbf{simply-connected} if $\pi_1 (H) = 1$.

\begin{cor}
\label{cor::rank-of-swj}
Suppose that $G$ satisfies the assumptions of Remark~\ref{rmk::roche-assumptions}, in particular, that $G_{\rm{der}}$ is simply connected. Let $w = t_\lambda \bar{w}$ be $\mu$-admissible, and let $J \subseteq \Phi$ be a root subsystem. Suppose $\bar{w} \in W_J$. Then ${\rm{rank}}(S_{w,J}) = {\rm{rank}}(\dualtorus) - {\rm{rank}}(J) - 1$.
\end{cor}

\begin{proof}
By assuming that $G_{\rm{der}}$ is simply connected, we have that $\cochars{\torus} / \mathbb{Z}\Phi^\vee$ is torsion-free. So any torsion element of $\cochars{\torus} / \mathbb{Z}J^\vee$ must come from the quotient $\mathbb{Z}\Phi^\vee / \mathbb{Z}J^\vee$.

We claim that $\lambda \notin \mathbb{Z}\Phi^\vee$. Fix a basis $\Delta$ of $\Phi$.  Without loss of generality, we may assume that $\lambda$ is dominant because the coroot lattice is stable under the $W$-action. A nonzero dominant coweight is in particular a fundamental coweight, which corresponds to some simple root $\alpha_i \in \Delta$. For any $\alpha_j \in \Delta$, $\langle \lambda, \alpha_j\rangle = \delta_{ij}$.

Suppose that $\lambda \in \mathbb{Z}\Phi^\vee$, so that we have an expression
$$
\lambda = \sum_{\alpha_j \in \Delta} c_{\alpha_j} \alpha_j^\vee,\ {\rm{where}}\ c_{\alpha_j} \in \mathbb{Z}.
$$
We apply the simple reflection $s_i$ to $\lambda$ in two ways. First,
$$
s_i \lambda = \lambda - \langle \lambda, \alpha_i \rangle \alpha_i^\vee = \lambda - \alpha_i^\vee.
$$
On the other hand,
$$
s_i \lambda = s_i \Big( \sum_{\alpha_j \in \Delta} c_{\alpha_j} \alpha_j^\vee\Big) = \sum_{\alpha_j \in \Delta} c_{\alpha_j} s_i \alpha_j^\vee = -c_{\alpha_i}\alpha_i^\vee + \sum_{\alpha_j \in \Delta \setminus \{\alpha_i\}} d_j \alpha_j^\vee,
$$
where we have used that $s_i$ permutes the simple coroots $\alpha_j^\vee$ for $j \neq i$. The coefficients $d_{\alpha_j}$ are the permuted $c_{\alpha_j}$. Now combine these two different forms of $s_i \lambda$ to get a relation
$$
(2c_{\alpha_i} - 1)\alpha_i^\vee + \sum_{\alpha_j \in \Delta \setminus \{\alpha_i\}} (c_{\alpha_j} - d_{\alpha_j}) \alpha_j^\vee = 0.
$$
This is a linear combination of the basis elements of $\mathbb{Z}\Phi^\vee$, hence all coefficients must be zero. This implies $2c_{\alpha_i} = 1$, which is impossible since $c_{\alpha_i} \in \mathbb{Z}$.


Because $\lambda \notin \mathbb{Z}\Phi^\vee$, it is not a torsion element of $\cochars{\torus} / \mathbb{Z}J^\vee$, and consequently $\chars{\dualtorus}/L_{w,J} = \left(\chars{\dualtorus}/L_J\right)/\mathbb{Z}\lambda$ has rank equal to $\Big({\rm{rank}}(\dualtorus) - {\rm{rank}}(J) - 1\Big)$, because ${\rm{rank}}(L_J) = {\rm{rank}}(J)$.



\end{proof}

\begin{rmk}
\label{rmk::assume-gder-sc}
The assumption that $G_{\rm{der}}$ be simply connected still admits the $GL_n$ and $GSp_{2n}$ cases. Moreover, Corollary~\ref{cor::rank-of-swj} is only used to find the rank of certain $S_{w, J}$ at the very end of the proof of the main theorem (Theorem~\ref{main-theorem}); the rest of the proof is independent of this assumption.
\end{rmk}

When $G$ is $GL_n$, $GSp_{4}$ or $GSp_{6}$, we can give a precise description of the structure of $\chars{S_{w,J}}$ as a finitely generated abelian group. More specifically, in these cases we can say something about how torsion appears, rather than only finding the rank of the group.

\begin{lemma}
\label{lemma::rel-grp-gln-gsp2n}
Suppose $G=GL_n$, $GSp_{4}$ or $GSp_{6}$. Consider $S_{w,J}$ for $\mu$-admissible $w$ and a root subsystem $J \subseteq \Phi$. The Smith form of $\chars{S_{w,J}}$ has a single invariant factor $c_1$, and
$$
\chars{S_{w,J}} \cong
\begin{cases}
\mathbb{Z}^{{\rm{rank}}(\dualtorus) - {\rm{rank}}(J)-1}, & {\rm{if}}\ c_1 = \pm 1,\\
\mathbb{Z}^{{\rm{rank}}(\dualtorus) - {\rm{rank}}(J)-1} \times \mathbb{Z}/c_1 \mathbb{Z}, & {\rm{otherwise}}.
\end{cases}
$$
\end{lemma}

\begin{proof}
We will handle the $GL_n$ case separately from $GSp_4$ and $GSp_6$.



\noindent\textbf{Case 1: $G = GL_n$.} 
The quotient $\cochars{\torus}/\mathbb{Z}J^\vee$ is the fundamental group of an endoscopic group of $GL_n$. All endoscopic groups have the property that their derived groups are simply connected, so $\cochars{\torus}/\mathbb{Z}J^\vee$ is torsion-free for all $J \subseteq \Phi$. We show how to write down an explicit set of generators.

Let $\varepsilon_1, \ldots \varepsilon_n$ be the coordinate generators of $\cochars{T}$. Coroots $\alpha^\vee \in J^\vee$ have the form $\alpha_{uv}^\vee = \varepsilon_u - \varepsilon_v$ The image of $\varepsilon_i$ in $\cochars{T}/\mathbb{Z}J^\vee$ is denoted $\bar{\varepsilon}_i$.

Let $J = J_1 \coprod \cdots \coprod J_r$ be the irreducible decomposition of $J$. Let $i_k$ be the minimal index among those $i$ such that $\langle \varepsilon_{i}, \alpha \rangle \neq 0$ for $\alpha \in J_k$. Then if $\alpha_{uv} \in J_k$, we have $\bar{\varepsilon}_{i_k} = \bar{\varepsilon}_{u} = \bar{\varepsilon}_{v}$ in the quotient. Let $A$ denote the set of indices $t$ such that $\langle \varepsilon_t, \alpha \rangle = 0$ for all $\alpha \in J$. Then $\cochars{T}/\mathbb{Z}J^\vee$ is generated by the $\bar{\varepsilon}_{i_k}$ and the $\bar{\varepsilon}_t$. Relabel this generating set such that
$$
\cochars{T}/\mathbb{Z}J^\vee = \langle \bar{\varepsilon}_{j_1}, \ldots, \bar{\varepsilon}_{j_m} \rangle.
$$

Let $\lambda = \lambda(w)$ be the translation part of $w$. Then $\lambda = \sum_{\ell=1}^t \varepsilon_{i_\ell}$ in $\cochars{T}$. Let $I_\lambda = \{i_1, \ldots, i_t\}$. There is a map
 $$
 f: I_\lambda \rightarrow \{j_1, \ldots, j_m\}
 $$
defined by sending $\varepsilon_i$ to its image in $\cochars{T}/\mathbb{Z}J^\vee$. Set $c_i = \vert f^{-1}(j_i)\vert$ for $1 \le i \le m$, so $\bar{\lambda} = \sum_{k=1}^m c_i \bar{\varepsilon}_{j_k}$.

Corollary~\ref{cor::chars-of-swj} and Proposition~\ref{prop::structure-of-lwj} show that
$$
\chars{S_{w,J}} = \Big(\cochars{T}/\mathbb{Z}J^\vee\Big) / \langle \lambda \rangle.
$$
Thus $\bar{\lambda} = \sum_{k=1}^m c_i \bar{\varepsilon}_{j_k}$ can be thought of as a $1 \times m$ relation matrix among the generators of $\cochars{\torus}/\mathbb{Z}J^\vee$. The resulting Smith form, which describes the finitely generated abelian group $\chars{S_{w,J}}$, has a single invariant factor $c_1$ and rank equal to $\Big({\rm{rank}}(\dualtorus) - {\rm{rank}}(J) - 1\Big)$ by Corollary~\ref{cor::rank-of-swj}.

\noindent\textbf{Case 2: $G = GSp_{4}$ or $GSp_{6}$.} In this case,
$$
\cochars{\torus} = \langle \varepsilon_0, \varepsilon_1, \ldots, \varepsilon_{n} \rangle,
$$
where $\varepsilon_0$ is the similitude cocharacter. The positive coroots have one of three forms for $i < j$: $\varepsilon_i - \varepsilon_j$, $\varepsilon_i + \varepsilon_j$ or $\varepsilon_i$. For these low-rank groups, one can check all possible $J^\vee$ to confirm that $\cochars{\torus}/\mathbb{Z}J^\vee$ is torsion-free and generated by the images of a subset of the $\varepsilon_i$.

There are six possible $J^\vee$ for $GSp_4$:
\begin{align*}
J_1^\vee &= \{ \pm (\varepsilon_1 - \varepsilon_2)\} &  J_4^\vee &= \{ \pm \varepsilon_2 \} \\
J_2^\vee &= \{ \pm (\varepsilon_1 + \varepsilon_2)\} &  J_5^\vee &= \{ \pm \varepsilon_1\} \coprod \{ \pm \varepsilon_2\} \\
J_3^\vee &= \{ \pm \varepsilon_1 \} &  J_6^\vee &= \Phi^\vee \\
\end{align*}
Each system gives a relations matrix $A_J$ among the generators $\varepsilon_0, \ldots, \varepsilon_n$. For example, in the the case of $J_5^\vee$ the matrix is

\[
A_{J_5} =
\begin{bmatrix}
0 & 1 & 0 \\
0 & 0 & 1 \\ 
\end{bmatrix}
\]

where the columns correspond to generators of $\cochars{\torus}$ and each row is a relation coming from a positive coroot. This matrix has invariant factors $(1,1)$, hence $\cochars{\torus} / \mathbb{Z}J_5^\vee \cong \mathbb{Z}^{3-2} = \mathbb{Z}$ is torsion-free. The relation matrices for $J_1^\vee, \ldots, J_4^\vee$ consist of a single row and have invariant factor $(1)$, so each $\cochars{\torus} / \mathbb{Z}J_i^\vee$ is torsion-free for $i = 1,\ldots, 4$. Meanwhile, we know $\cochars{\torus} / \mathbb{Z}\Phi^\vee$ is torsion-free because $GSp_4$ has a simply-connected derived group.

We use the same approach for $GSp_6$. There are thirty possible $J^\vee$, which split into seven families organized by the type of $J^\vee$. We list each type, followed by the number of members in the family, and then a representative $J^\vee$ of the family.
\begin{align*}
{\rm{Type}}\ A_1 &\ (9) & J^\vee &= \{\pm (\varepsilon_1 - \varepsilon_2)\} \\
{\rm{Type}}\ A_1 \times A_1 &\ (9) & J^\vee &= \{\pm (\varepsilon_1 + \varepsilon_2)\} \coprod \{\pm \varepsilon_3\} \\
{\rm{Type}}\ A_2 &\ (4) & J^\vee &= \{\pm (\varepsilon_1 + \varepsilon_2), \pm (\varepsilon_1 + \varepsilon_3), \pm (\varepsilon_2 - \varepsilon_3)\} \\
{\rm{Type}}\ C_2 &\ (3) & J^\vee &= \{\pm (\varepsilon_1 - \varepsilon_2), \pm (\varepsilon_1 + \varepsilon_2), \pm \varepsilon_1, \pm \varepsilon_2\} \\
{\rm{Type}}\ A_1 \times A_1 \times A_1 &\ (1) & J^\vee &= \{\pm \varepsilon_1\} \coprod \{\pm \varepsilon_2\} \coprod \{\varepsilon_3\} \\
{\rm{Type}}\ C_2\times A_1 &\ (3) & J^\vee &= \{\pm (\varepsilon_1 - \varepsilon_2), \pm (\varepsilon_1 + \varepsilon_2), \pm \varepsilon_1, \pm \varepsilon_2\} \coprod \{\pm \varepsilon_3\} \\
{\rm{Type}}\ C_3 &\ (1) & J^\vee &= \Phi^\vee \\
\end{align*}
Now we can construct relation matrices $A_J$ as before to compute the Smith form of $\cochars{\torus} / \mathbb{Z}J^\vee$ in each case. It turns out that in every case, all invariant factors are units, which means the quotient is torsion-free. Here are some examples.

Let $J^\vee = \{\pm (\varepsilon_1 + \varepsilon_2), \pm (\varepsilon_1 + \varepsilon_3), \pm (\varepsilon_2 - \varepsilon_3)\}$:

\[
A_{J} =
\begin{bmatrix}
0 & 1 & 1 & 0 \\
0 & 1 & 0 & 1 \\
0 & 0 & 1 & -1 \\
\end{bmatrix}
\sim
\begin{bmatrix}
1 & 0 & 0 & 0 \\
0 & -1 & 0 & 0 \\
0 & 0 & 0 & 0 \\
\end{bmatrix}
\]

Let $J^\vee = \{\pm (\varepsilon_1 - \varepsilon_2), \pm (\varepsilon_1 + \varepsilon_2), \pm \varepsilon_1, \pm \varepsilon_2\} \coprod \{\pm \varepsilon_3\}$:

\[
A_{J} =
\begin{bmatrix}
0 & 1 & -1 & 0 \\
0 & 1 & 1 & 0 \\
0 & 1 & 0 & 0 \\
0 & 0 & 0 & 1 \\
\end{bmatrix}
\sim
\begin{bmatrix}
1 & 0 & 0 & 0 \\
0 & 1 & 0 & 0 \\
0 & 0 & 1 & 0 \\
0 & 0 & 0 & 0 \\
\end{bmatrix}
\]

Since $\cochars{\torus}/\mathbb{Z}J^\vee$ is torsion-free in all cases for $GSp_4$ and $GSp_6$, we can now proceed as in the $GL_n$ case. Consider the image of $\lambda = \sum_{\ell=1}^r \varepsilon_{i_\ell}$ in $\cochars{\torus}/\mathbb{Z}J^\vee$ to get a relation among the generators of the quotient, then take the Smith form.
\end{proof}


\subsection{Analogues of the critical index torus}

The final section of this chapter determines an analogue of the group of $\resfield$-points of the ``critical index torus'' $T_{S(w)}$ defined in the Drinfeld case~\cite{haines-rapoport2012}, Section 6.4, wherein this subtorus is defined by specifying that a coordinate in the diagonal torus $\torus$ is always 1 if that coordinate does not belong to a certain subset of indices $S(w)$. There are important differences between the Drinfeld case and the more general case being considered here. First, in the general case we cannot find a \emph{subtorus} to play the role of $T_{S(w)}$, though we can approximate its behavior through certain finite subgroups of $\torus(\resfield)$. Second, whereas the support of $\phironeaug$ can be described in terms of the $\mu$-admissible set and the single group $T_{S(w)}(\resfield)$ in the Drinfeld case, the general case admits the possibility that multiple finite groups $\critgrpJ$ are needed to completely describe those $s \in \torus(k_r)$ whose norms $N_r(s)$ contribute to nonzero coefficients $\phironeaug(I_r^+ sw I_r^+)$. This second point will be elaborated in Chapter 5.

\subsubsection{Stratification of $\dzrelgrp$ by $\chi$-root systems}
\label{subsection::stratification-by-chi-root-systems}

Recall that $\dualtorus(\valfield)^{\rm{dz}}$ is the finite subgroup of $\dualtorus(\valfield)$ comprising the depth-zero endoscopic elements $\kappa_\chi$ associated to $\chi \in \torus(\resfield)^\vee$. Consider any $\kapchi \in \dualtorus(\valfield)^{\rm{dz}}$, not necessarily relevant to $w$. The corresponding character $\chi$ determines a $\chi$-root system $\Phi_\chi \subseteq \Phi$, which may be reducible. We stratify $\widehat{T}(\valfield)^{\rm{dz}}$ by defining
$$
\dualtorus(J) = \left\{\kappa_\chi \in \dualtorus(\valfield)^{\rm{dz}} \mid J = \Phi_\chi\right\},
$$
which yields
$$\dualtorus(\valfield)^{\rm{dz}} = \coprod_{J \subseteq \Phi} \dualtorus(J).$$
Now any subset $A \subseteq \dualtorus(\valfield)^{\rm{dz}}$ can be stratified by setting $A(J) = A \cap \dualtorus(J)$ for each $J \subseteq \Phi$, thus we come to consider the strata $\dzrelgrp(J)$ of the depth-zero relevant group.

%

\begin{cor}
\label{cor::union-of-strata}
With notation as above,
$$
\dzrelgrpJ = \coprod_{J \subseteq J^\prime \subseteq \Phi} \dzrelgrp(J^\prime).
$$
\end{cor}

\begin{proof}
This is a consequence of Proposition~\ref{prop::set-version-of-dzrelgrpJ}.
\end{proof}

\begin{lemma}
\label{lemma::equality-on-strata}
Consider a stratum $\dzrelgrp(J)$. For all characters $\chi \in \torus(\resfield)^\vee$ such that $\kapchi \in \dzrelgrp(J)$: the groups $W_\chi^\circ$ are all isomorphic, the length functions $\ell_\chi$ agree in the obvious sense, and the polynomials $\widetilde{R}_{u,v}^\chi(Q_r)$ are identical for all $u, v \in \widetilde{W}_\chi$.
\end{lemma}

\begin{proof}
The key point is that all $\kapchi \in \dzrelgrp(J)$ have the same $\chi$-root system $\Phi_\chi \subseteq \Phi$. Consequently, the groups $W_\chi^\circ$ are all identical, and so are the length functions $\ell_\chi$. Finally, the Hecke algebras $\mathcal{H}(W_{\chi, {\rm{aff}}}, S_{\chi, {\rm{aff}}})$ all determine the same $\widetilde{R}^\chi$-polynomials.
\end{proof}

The common length function and $\widetilde{R}^\chi$-polynomials associated to points in the stratum $S_{w}^{\rm{dz}}(J)$ are sometimes denoted $\ell_J$ and $\widetilde{R}^J$, respectively.

\subsubsection{Definition of finite critical groups}
\label{section::finite-critical-grps}

The critical index torus $T_{S(w)}$ defined in the Drinfeld case is first understood in terms of geometric data concerning the special fiber of the Shimura variety. Haines and Rapoport prove that $T_{S(w)}$ is the particular subtorus of $\torus$ corresponding to the lattice $L_w \subset \cochars{\torus}$, where $L_w$ is defined as in Lemma~\ref{lattice-equality-lemma}. The group of $k_F$-points $\torus_{S(w)}(\resfield)$ plays a role in determining the support of $\phironeaug$ in the Drinfeld case. See also Section~\ref{section::the-drinfeld-case}. In the more general case considered here, the analogous lattices $L_{w,J}$ do not necessarily correspond to subtori of $\torus$, because $\cochars{\torus}/L_{w,J}$ is not always torsion-free.

\begin{defin}
\label{defin::finite-critical-groups}
Let $w$ be $\mu$-admissible, and let $J \subseteq \Phi$ be a root subsystem. The lattice $L_{w,J}$ from Corollary~\ref{cor::chars-of-swj} can be viewed as a lattice in $\cochars{\torus}$. Define a subgroup of $\torus(\resfield)$ by
$$
A_{w,J,\resfield} = \langle \nu(\resfield) \:\vert\: \nu \in L_{w,J}\rangle,
$$
This is the \textbf{finite critical group} of $w$ and $J$ with respect to $\resfield$.
\end{defin}


\begin{lemma}
Let $\chi$ be a depth-zero character on $\torus(\resfield)$. Then $\kapchi$ lies in $\dzrelgrpJ$ if and only if the restriction of $\chi$ to $A_{w,J,\resfield}$ is the the trivial character.
\end{lemma}

\begin{proof}
Suppose $\kappa_\chi \in \dzrelgrpJ$. We have $\eta(\kappa_\chi) = 1$ for all $\eta \in L_{w,J}$ because $\chars{\dzrelgrpJ} = \chars{\dualtorus}/L_{w,J}$. Thus viewing $\eta$ as an element of $\cochars{\torus}$, we have $\chi(\eta(\tilde{x})) = 1$ for any generator $\tilde{x} = \tau_F(x)$ of $k_F^\ast$, since $\chi(\eta(\tilde{x})) = \eta(\kappa_\chi)$. But $\critgrpJ$ is generated by the elements $\eta(\tilde{x})$, so $\chi$ is trivial on this subgroup of $\torus(\resfield)$.

Conversely, suppose $\chi$ restricts to the trivial character on $\critgrpJ$. For any generator $\tilde{x}$ of $\resfield^\ast$ such that $\tilde{x} = \tau_F (x)$ for $x \in I_F$, the hypothesis implies $\chi(\eta(\tilde{x})) = 1$ for $\eta \in L_{w,J}$. As before, $\chi(\eta(\tilde{x})) = \eta(\kappa_\chi)$, hence $\eta(\kappa_\chi) = 1$ for all $\eta \in L_{w,J}$. Therefore $\kappa_\chi \in \dzrelgrpJ$ by Corollary~\ref{cor::contain-by-annihilation}.
\end{proof}



\subsubsection{Sums of over groups of endoscopic elements}

Let $N_r : \torus(k_r) \rightarrow \torus(k_F)$ be the norm map. Lemma~\ref{lemma::defin-of-gamma-nrs} showed that given $s \in T(k_r)$, we can attach an element $\gamma_{N_r s} \in \cochars{\torus}$ to $N_r(s) \in \torus(k_F)$. The purpose of this section is to determine the possible values of sums
$$
\sum_{\kapchi \in \dzrelgrpJ} \gamma_{N_r s}(\kapchi)^{-1}.
$$
Such sums will appear in Chapter 5 as a result of grouping terms in an expression for $\phironeaug$.  The finite critical groups introduced in Definition~\ref{defin::finite-critical-groups} are the key determining factor for whether a sum is zero.

\begin{prop}
\label{prop::sum-over-group}
Let $s \in \torus(k_r)$, and define $\gamma_{N_r s}$ as above. Then
$$\sum_{\kapchi \in \dzrelgrpJ} \gamma_{N_r s}(\kapchi)^{-1} = \begin{cases}
0, & {\rm{if}}\ N_r(s) \notin A_{w, J, k_F} \\
\vert \dzrelgrpJ \vert, & {\rm{otherwise}}.
\end{cases}
$$
\end{prop}

\begin{proof}
First, suppose there exists $\kappa_0 \in \dzrelgrpJ$ such that $\gamma_{N_r s} (\kappa_0) \neq 1$. Then
$$
\sum_{\kapchi \in \dzrelgrpJ} \gamma_{N_r s}(\kapchi)^{-1} = \sum_{\kapchi \in \dzrelgrpJ} \gamma_{N_r s}(\kappa_0 \kapchi)^{-1} = \gamma_{N_r s}(\kappa_0)^{-1} \sum_{\kapchi \in \dzrelgrpJ} \gamma_{N_r s}(\kapchi)^{-1},
$$
implies $\sum_{\kapchi \in \dzrelgrpJ} \gamma_{N_r s}(\kapchi)^{-1} = 0$. But $\gamma_{N_r s}(\kappa_\chi) \neq 1$ for some $\kappa_\chi \in S_{w,J}^{\rm{dz}}$ if and only if $N_r(s) \notin A_{w, J, k_F}$. The second case is obvious, because then $N_r(s) \in A_{w,J,\resfield}$ implies all $\gamma_{N_r(s)}(\kapchi) = 1$
\end{proof}

\section{Combinatorics on increasing paths in a Bruhat graph}
\label{chapter::combinatorics-on-paths}

The explicit formula for Bernstein functions attached to dominant minuscule cocharacters  applied in Chapter 2 introduced a term to our formula connected with the $\widetilde{R}$-polynomials defined by Kazhdan and Lusztig. In fact, we must work with polynomials $\widetilde{R}_{w, t_{\lambda(w)}}^J(Q_r)$ for $w \in {\rm{Adm}}_{G_r}(\mu)$ and a $\chi$-root system $J \subseteq \Phi$. This chapter applies a formula for $\widetilde{R}$-polynomials due to Dyer in order to rewrite $\widetilde{R}_{w, t_{\lambda(w)}}^J(Q_r)$ in a new way, which will simplify our formula for the coefficients of $\phironeaug$.

Our first order of business is to recall some definitions and results pertaining to reflection subgroups and reflection orderings on Coxeter groups. Then, we will take a closer look at the behavior of these objects in the special case of the Bruhat interval between an affine Weyl group element and its translation part, before going on to prove the desired modification to Dyer's formula in the final section.


\subsection{Background on Coxeter groups}
\label{section::coxeter-group-background}

Recall that a \textbf{Coxeter system} is a pair $(\mathcal{W},\mathcal{S})$ composed of a group $\mathcal{W}$ generated by a set of involutions $\mathcal{S}$ subject to braid relations $(s_i s_j)^{m(i,j)} = 1$, where $m(i,j) \in \{ \mathbb{Z}, \infty\}$. The \textbf{length} of an element $w \in \mathcal{W}$ is denoted $\ell(w)$. A \textbf{finite} Coxeter group has all $m(i,j) \in \mathbb{Z}$, and a \textbf{finite rank} Coxeter group has $\vert \mathcal{S} \vert < \infty$. Weyl groups are finite Coxeter groups, while affine Weyl groups are infinite. Both types of groups have finite rank.

Essentially all of the material in this first section is covered in the book by Bj{\"o}rner and Brenti~\cite{bjorner-brenti2005} or in the papers of Dyer cited throughout.

\subsubsection{Bruhat order and Bruhat graphs}

The following statement is Definition 2.1.1 in~\cite{bjorner-brenti2005}.

\begin{defin}
Let $(\mathcal{W}, \mathcal{S})$ be a Coxeter system and let
$$
\mathcal{T} = \{wsw^{-1} \mid w \in \mathcal{W}, s \in \mathcal{S}\} = \bigcup_{w \in \mathcal{W}} w\mathcal{S}w^{-1}
$$
be its set of \textbf{reflections}. Let $u, w \in \mathcal{W}$. Then
  \begin{enumerate}
    \item $u\stackrel{t}{\rightarrow}w$ means that $u^{-1}w = t \in \mathcal{T}$ and $\ell(u) < \ell(w)$.
    \item $u\rightarrow w$ means that $u\stackrel{t}{\rightarrow}w$ for some $t \in \mathcal{T}$.
    \item $u \leq w$ means that there exist $u_i \in \mathcal{W}$ such that $$u = u_0 \rightarrow u_1 \rightarrow \cdots \rightarrow u_k = w.$$
  \end{enumerate}
The \emph{\textbf{Bruhat graph}} $\Omega_{(\mathcal{W},\mathcal{S})}$ is the directed graph whose vertices are the elements of $\mathcal{W}$ and whose edges are given by $u \rightarrow w$. \emph{\textbf{Bruhat order}} is the partial order relation $u \leq w$ on the set $\mathcal{W}$. 
\end{defin}

For elements $x \leq y$ in a Coxeter system $(\mathcal{W}, \mathcal{S})$, a \textbf{path} $\Delta$ from $x$ to $y$, also written $x\stackrel{\Delta}\longrightarrow y$, is a set of edges in $\Omega_{(\mathcal{W},\mathcal{S})}$ that connect the vertices $x$ and $y$. Let $B_{\mathcal{W}}(x,y)$ denote the set of all paths $x\stackrel{\Delta}\longrightarrow y$ through $\Omega_{(\mathcal{W}, \mathcal{S})}$. If $W$ is a Weyl group associated to a root system $\Phi$, this set instead may be written $B_\Phi(x,y)$.

\subsubsection{Reflection subgroups of Coxeter groups}

\begin{defin}
Let $\mathcal{W}$ be a Coxeter group with set of reflections $\mathcal{T}$. Any subgroup $\mathcal{W}^\prime \subset \mathcal{W}$ satisfying $\mathcal{W}^\prime = \langle \mathcal{W}^\prime \cap \mathcal{T} \rangle$ is called a \textbf{reflection subgroup} of $\mathcal{W}$.
\end{defin}

The following is Definition 3.1 of~\cite{dyer1990}.

\begin{defin}
\label{defin::coxeter-gens-of-refl-subgrp}
Let $\mathcal{W}$ be a Coxeter group. For $w \in \mathcal{W}$, let
$$
N(w) = \{ t \in \mathcal{T} \mid \ell(tw) < \ell(w)\}.
$$
If $\mathcal{W}^\prime$ is a subgroup of $\mathcal{W}$, let
$$
\Sigma(\mathcal{W}^\prime) = \Big\{ t \in \mathcal{T} \mid N(t) \cap \mathcal{W}^\prime = \{t\}\Big\}.
$$
\end{defin}

\begin{thm} \label{reflection-subgroups-are-coxeter-systems}
Let $\mathcal{W}^\prime$ be a reflection subgroup of a Coxeter system $(\mathcal{W}, \mathcal{S})$ and let $\mathcal{S}^\prime = \Sigma(\mathcal{W}^\prime)$. Then
  \begin{enumerate}
    \item $\mathcal{W}^\prime \cap \mathcal{T} = \bigcup_{u \in \mathcal{W}^\prime} u \mathcal{S}^\prime u^{-1}$, and
    \item $(\mathcal{W}^\prime, \mathcal{S}^\prime)$ is a Coxeter system.
  \end{enumerate}
\end{thm}

\begin{proof}
This is Theorem 3.3 of~\cite{dyer1990}, where the result is established for general reflection systems. The fact that a reflection subgroup of a Coxeter group if itself a Coxeter group was proved independently by Deodhar \cite{deodhar1989} (see the ``Main Theorem'' proved in Section 3 of the cited paper).
\end{proof}

Dyer~\cite{dyer1991} further proved that the Bruhat graph associated to any reflection subgroup $\mathcal{W}^\prime$ of $(\mathcal{W}, \mathcal{S})$ embeds as a full subgraph of the Bruhat graph $\Omega_{(\mathcal{W}, \mathcal{S})}$. Let $\Omega_{(\mathcal{W}, \mathcal{S})}(\mathcal{W}^\prime)$ denote the full subgraph of $\Omega_{(\mathcal{W}, \mathcal{S})}$ on the vertex set $\mathcal{W}^\prime$.

\begin{thm}
\label{thm::bruhat-graph-of-reflection-subgroups}
Let $\mathcal{W}^\prime$ be a reflection subgroup of $(\mathcal{W}, \mathcal{S})$ and set $\mathcal{S}^\prime = \Sigma(\mathcal{W}^\prime)$.
  \begin{enumerate}
    \item $\Omega_{(\mathcal{W}^\prime, \mathcal{S}^\prime)} = \Omega_{(\mathcal{W}, \mathcal{S})} (\mathcal{W}^\prime)$.
    \item For any $x \in \mathcal{W}$, there exists a unique $x_0 \in \mathcal{W}^\prime x$ such that the map ${\mathcal{W}^\prime \rightarrow \mathcal{W}^\prime x}$ defined by $w \mapsto wx_0$, for $w \in \mathcal{W}^\prime$, is an isomorphism of directed graphs $\Omega_{(\mathcal{W}, \mathcal{S})} (\mathcal{W}^\prime) \rightarrow \Omega_{(\mathcal{W}, \mathcal{S})} (\mathcal{W}^\prime x)$.
  \end{enumerate}
\end{thm}

\begin{proof}
This is Theorem 1.4 of~\cite{dyer1991}.
\end{proof}

Suppose $(\mathcal{W}, \mathcal{S})$ is a finite-rank Coxeter system, i.e., $\mathcal{S}$ has finite cardinality. This is the case for all Coxeter groups considered in this article. There is a root system $\Phi_{\mathcal{W}}$ associated to any such $(\mathcal{W}, \mathcal{S})$ arising from the standard geometric representation of $\mathcal{W}$. See, for example, \cite{bjorner-brenti2005}, Section 4.4.

\begin{lemma}
\label{lemma::reflection-subgroup-gives-root-subsystem}
Let $W$ be a Weyl group associated to a root system $\Phi$ and let $W^\prime$ be any reflection subgroup of $W$. Then the root system $\Phi_{W^\prime}$ is a sub-system of $\Phi$.
\end{lemma}

\begin{proof}
Per~\cite{bjorner-brenti2005}, the root system $\Phi=\Phi_W$ of $(W,S)$ is equal to
$$
\Phi = \{w(\alpha_s) \mid w\in W, s \in S\},
$$
where the $\alpha_s$ form a basis for ambient Euclidean space with dimension equal to $\vert S \vert$.

The reflection group $W^\prime$ is a Coxeter system $(W^\prime, \Sigma(W^\prime)$, hence there exists a root system $\Phi_{W^\prime}$ as above. For each $s^\prime$, the root $a_{s^\prime}$ equals $u(\alpha_s)$ for some $u \in W$ and $s \in S$.  Therefore, for any $w^\prime \in W^\prime$ and $s^\prime \in \Sigma(W^\prime)$, the root $w^\prime(\alpha_{s^\prime})$ has the form $u(\alpha_s)$ for some $u \in W$ and $s \in S$. This proves $\Phi_{W^\prime} \subseteq \Phi$.
\end{proof}

\subsubsection{Reflection orderings}
\label{section::reflection-orderings}

Dyer introduced the notion of a reflection ordering in \cite{dyer1993}. Our presentation also draws from Sections 5.2 and 5.3 of~\cite{bjorner-brenti2005}.

\begin{defin}
Let $(\mathcal{W}, \mathcal{S})$ be a finite-rank Coxeter system, and let $\Phi_{\mathcal{W}}$ be its associated root system. A total ordering $\prec$ on the (possibly infinite) set of positive roots $\Phi_{\mathcal{W}}^+$ is a \textbf{reflection ordering} if for all $\alpha, \beta \in \Phi_{\mathcal{W}}^+$ and $\lambda, \mu \in \mathbb{R}_{> 0}$ such that $\lambda \alpha + \mu \beta \in \Phi_{\mathcal{W}}^+$, we have that either
$$\alpha \prec \lambda \alpha + \mu \beta \prec \beta$$
or
$$ \beta \prec \lambda \alpha + \mu \beta \prec \alpha.$$
\end{defin}

The bijection between the positive roots $\Phi_{\mathcal{W}}^+$ and the set of reflections $\mathcal{T}$ in $\mathcal{W}$ means that a reflection ordering induces a total ordering on $\mathcal{T}$.

\begin{prop}
Let $(\mathcal{W}, \mathcal{S})$ be a finite-rank  Coxeter system, and let $\Phi_{\mathcal{W}}$ be its associated root system. Then there exists a reflection ordering on $\Phi_{\mathcal{W}}^+$.
\end{prop}

\begin{proof}
This first appeared in in \cite{dyer1993}, (2.1) - (2.3), and an alternative proof is given in \cite{bjorner-brenti2005}, Proposition 5.2.1.
\end{proof}

We emphasize that Dyer's theory holds for both finite and infinite Coxeter groups. In what follows, we will repeatedly discuss reflection orderings on reflections in a finite Weyl group \emph{and} reflection orderings on affine reflections in an affine Weyl group. For a finite Weyl group $W$ (resp. an affine Weyl group $W_{\rm{aff}}$), the root system $\Phi_{\mathcal{W}}$ coincides with the root system of $W$ (resp. the affine root system of $W$). See~\cite{humphreys1990}, Sections 6.4 - 6.5.

\begin{defin}
Let $(\mathcal{W}, \mathcal{S})$ be a finite-rank Coxeter group, and fix a reflection ordering $\prec$ on $\Phi_{\mathcal{W}}^+$. Given a path $$\Delta = \{w_0, w_1, \ldots, w_n\}$$ from $u \rightarrow v$ through the Bruhat graph for $(\mathcal{W},\mathcal{S})$, define the \textbf{edge set} of $\Delta$ by
$$
E(\Delta) = \{ w_{i-1}^{-1} w_i \mid 1 \leq i \leq n\}.
$$
The \textbf{descent set} of $\Delta$ with respect to the reflection ordering $\prec$ is defined by
$$
D(\Delta; \prec) = \Big\{ i \in \{1,\ldots, n-1\} : w_i^{-1} w_{i+1} \prec  w_{i-1}^{-1} w_i\Big\}.
$$
\end{defin}

Given a fixed reflection ordering $\prec$ on the reflections of $\mathcal{W}$ and $x \leq y$ in the Bruhat order on $\mathcal{W}$, we denote the set of $\prec$-increasing paths from $x$ to $y$ by
$$
B_{\mathcal{W}}^\prec (x,y) = \{ \Delta \in B_{\mathcal{W}}(x,y) \mid D(\Delta; \prec) = \emptyset \}.
$$


\subsection{Bruhat intervals for admissible elements}

This section shows that any path from a $\mu$-admissible $w$ in $\widetilde{W}$ to its translation part $t_{\lambda(w)}$ consists solely of reflections appearing in the finite Weyl group $W$.

\begin{rmk}
In this section and the next, we will sometimes write $B_{\Phi_{\rm{aff}}}(w, t_{\lambda(w)})$ when working with $w$ and $t_{\lambda(w)}$ in $\widetilde{W}$---as opposed to $W_{\rm{aff}}$. Let us justify this apparent misuse of notation.

Recall how Bruhat order works in the extended affine Weyl group: if $w \leq t_\lambda$, then there exists a length-zero element $\sigma \in \widetilde{W}$ such that $w, t_{\lambda(w)} \in \sigma^{-1} W_{\rm{aff}}$, and $\sigma^{-1} w \leq \sigma^{-1} t_{\lambda(w)}$ in the Bruhat order on $W_{\rm{aff}}$. We are simply writing $B_{\Phi_{\rm{aff}}}(w, t_{\lambda(w)})$ instead of $B_{\Phi_{\rm{aff}}}(\sigma^{-1} w, \sigma^{-1} t_{\lambda(w)})$.
\end{rmk}

\begin{prop}
\label{prop::finite-reflections-in-interval}
Let $\mu$ be a dominant minuscule coweight of $\Phi$, and let $(W,S)$ be the finite Weyl group of $\Phi$ inside the affine Weyl group $(W_{\rm aff}, S_{\rm aff})$. Let $T$ be the set of reflections in $W$.

Consider a $\mu$-admissible element $w \leq t_{\lambda (w)}$.  There exists a length-zero element $\sigma$ in $\widetilde{W}$ such that $w, t_{\lambda(w)} \in \sigma W_{\rm{aff}}$. Let $w \stackrel{\Delta}\longrightarrow t_{\lambda (w)}$ be any path in the Bruhat graph $\Omega_{(W_{\rm{aff}}, S_{\rm{aff}})}$. Each reflection in the edge set $E(\Delta) = \{t_1, \ldots t_n\}$ belongs to $T$.
\end{prop}

\begin{proof}
Let $\mathcal{C}$ denote the base alcove. Recall that $W_{\rm{aff}}$ acts simply transitively on alcoves (see \cite{humphreys1990}, Section 4.5), and let $A_u = u \cdot \mathcal{C}$ for $u \in W_{\rm{aff}}$. Because $w$ belongs to ${\rm{Adm}}(\mu)$, it can be written $w = t_\lambda \bar{w}$ with $w \leq t_\lambda$. (N.B. the $\lambda$ here is $\lambda(w)$ by definition.)

We claim that $A_w$ and $A_{t_\lambda}$ both contain $\lambda$ in their closures. Observe that the fundamental alcove $\mathcal{C}$ and the alcove $\bar{w}\mathcal{C}$ both have the origin in their closure, because $\bar{w}$ is an element of the finite Weyl group. Translating each alcove by $\lambda$ means the translate of the origin lies in the closure of the translated alcoves.

A path $w \stackrel{\Delta}\longrightarrow t_{\lambda}$ is a sequence of (affine) reflections $t_1, \ldots, t_n$ such that
$$w < wt_1 < w t_1 t_2 < \cdots < wt_1 \cdots t_n = t_{\lambda}.$$

Now we further claim that each alcove $A_{w t_1\cdots t_i}$ also contains $\lambda$ in its closure. Let $H_i$ be a hyperplane crossed by going from $A_{w t_1 \ldots t_{i-1}}$ to $A_{w t_1\ldots t_i}$. All such $H_i$ weakly separate $A_w$ from $A_\lambda$. Suppose $\lambda$ did not lie in some $H_i$. Then it would be strictly on one side of $H_i$. This is a contradiction, because $\lambda$ belongs to the closure of both $A_w$ and $A_\lambda$, and these alcoves are separated by $H_i$.


As a matter of notation, given $u \in \widetilde{W}$, let $^u t_i = u t_i u^{-1}$. Then we can rewrite the above sequence as
$$w < {^w t_1} w < {^{wt_1}t_2} {^w t_1}w < \cdots <{^{wt_1\cdots t_{n-1}} t_n} \cdots {^{wt_1}t_2} {^w t_1}  w = t_{\lambda}.$$
The argument above shows that the hyperplane for each affine reflection ${^{wt_1\cdots t_{i-1}} t_i}$ passes through the point $\lambda$. Therefore, the corresponding reflection fixes this point, 
$$
{^{wt_1\cdots t_{i-1}} t_i} (\lambda) = \lambda.
$$
Since $w^{-1} = \bar{w}^{-1} t_{-\lambda}$, the preceding equation can be rewritten
$$
t_1 \cdots t_i \cdots t_1 \bar{w}^{-1}t_{-\lambda} (\lambda) = \bar{w}^{-1}t_{-\lambda}(\lambda)
$$
Using $\bar{w}^{-1} (0) = 0$, we conclude that for each $1 \leq i \leq n$,
$$t_1 \cdots t_i \cdots t_1 (0) = \overline{w}^{-1} (0) = 0.$$
An affine reflection fixes the origin if and only if its translation part is trivial. So the reflection $t_1 \cdots t_i \cdots t_1$ is in the finite Weyl group. It follows that each $t_i$ belongs to $T$.
\end{proof}

\begin{lemma}
\label{lemma::minimal-coset-rep}
Let $w = t_\lambda \bar{w}$ be $\mu$-admissible. There is an element $w_\lambda \in W$ such that $t_\lambda w_\lambda$ has minimal length in the coset $t_\lambda W$, and moreover, for any $x \in W$
$$
\ell(t_\lambda w_\lambda x) = \ell(t_\lambda w_\lambda) + \ell(x).
$$
Finally, for any $x, y \in W$, $t_\lambda w_\lambda x \leq t_\lambda w_\lambda y$ if and only if $x \leq y$.
\end{lemma}

\begin{proof}
Because we know $w \leq t_\lambda$, there is a length-zero element $\sigma \in \widetilde{W}$ such that $w, t_\lambda \in \sigma W_{\rm{aff}}$. We may and do think of $t_\lambda$ and $w$ as elements of $W_{\rm{aff}}$ by multiplying each on the left by $\sigma^{-1}$.

Since $W$ is a finite group, it is clearly possible to attain a minimal value in the set $\{ \ell(t_\lambda x) \mid x \in W\}$. Let $w_\lambda$ denote an element of $W$ such that $\ell(t_\lambda w_\lambda)$ is minimal. In fact, this $w_\lambda$ is unique by the theory of minimal coset representatives applied to the quotient $W_{\rm{aff}}/W$; see \cite{bjorner-brenti2005} Corollary 2.4.5 for example.

By the triangle inequality, for any $x \in W$ the lengths satisfy
$$
\ell(t_\lambda w_\lambda x) \leq \ell(t_\lambda w_\lambda) + \ell(x).
$$
Suppose $\ell(t_\lambda w_\lambda x) < \ell(t_\lambda w_\lambda) + \ell(x)$. Let $S_{\rm{aff}} = \{s_0, s_1, \ldots, s_r\},$ and choose reduced expressions $t_\lambda w_\lambda = s_{i_1}\cdots s_{i_n}$ and $x = s_{j_1}\cdots s_{j_m}$. By assumption of strict inequality, there is an index $j_k$ such that
$$
\ell(t_\lambda w_\lambda s_{j_1}\cdots s_{j_{k-1}} s_{j_k}) < \ell(t_\lambda w_\lambda s_{j_1}\cdots s_{j_{k-1}}).
$$
By the Exchange Condition, multiplying $t_\lambda w_\lambda s_{j_1}\cdots s_{j_{k-1}}$ by $s_{j_k}$ must delete some element in the expression $s_{i_1}\cdots s_{i_n} s_{j_1}\cdots s_{j_{k-1}}.$
Since the expression for $x$ is reduced, $s_{j_k}$ must delete some $s_{i_p}$. This contradicts the assumption that $t_\lambda w_\lambda = s_{i_1}\cdots s_{i_n}$ is reduced. Therefore, $\ell(t_\lambda w_\lambda x) = \ell(t_\lambda w_\lambda) + \ell(x).$

The final statement is a consequence of the above equality and the definition of Bruhat order. If we assume $x \neq y$, then $t_\lambda w_\lambda x < t_\lambda w_\lambda y$ means $\ell(t_\lambda w_\lambda x) < \ell(t_\lambda w_\lambda y)$. But then
$$
\ell(t_\lambda w_\lambda) + \ell(x) < \ell(t_\lambda w_\lambda) + \ell(y),
$$
and so $\ell(x) < \ell(y)$. Then by definition, $x < y$. To prove the opposite implication, run the argument in reverse.
\end{proof}

\begin{prop}
\label{bruhat-interval-isomorphism}
Let $w = t_\lambda \bar{w} \in {\rm{Adm}}_{G_r}(\mu)$. Given a reflection ordering $\prec$ on $W_{\rm{aff}}$, let $\prec$ also denote its restriction to $W$. There is an explicit $\prec$-preserving bijection between $
B_{\Phi_{\rm{aff}}}(w, t_\lambda)$, whose paths go through $\Omega_{(W_{\rm{aff}}, S_{\rm{aff}})}$, and $B_\Phi(w_\lambda^{-1}\wbar, w_\lambda^{-1})$, whose paths go through $\Omega_{(W,S)}$.
\end{prop}

\begin{proof}
Choose any $\Delta \in B_{\Phi_{\rm{aff}}}(w, t_\lambda)$, with edge set $E(\Delta) = \{t_1, \ldots, t_n\}$. That is,
$$
w < wt_1 < w t_1 t_2 < \cdots < w t_1 t_2 \cdots t_{n-1} < t_\lambda.
$$
By the final statement of Lemma~\ref{lemma::minimal-coset-rep}, this chain of inequalities holds if and only if the following chain also holds,
$$
w_\lambda^{-1} \wbar < w_\lambda^{-1} \wbar t_1 < \cdots w_\lambda^{-1} \wbar t_1\cdots t_{n-1} < w_\lambda^{-1}.
$$

Proposition~\ref{prop::finite-reflections-in-interval} shows that the $t_i$ are reflections in the finite Weyl group. Hence the new chain of inequalities defines a path $\Delta^\prime \in B_\Phi (w_\lambda^{-1} \wbar, w_\lambda^{-1})$. The edge sets of each path are identical, and moreover, the edges appear in the same order. Therefore, we have a $\prec$-preserving bijection between the two Bruhat intervals.
\end{proof}


\subsection{A stratified formula for $\widetilde{R}$-polynomials}

The preceding section showed that the set of paths between a $\mu$-admissible element and its translation part increasing with respect to a reflection ordering $\prec$ is in bijection with the $\prec$-increasing paths between certain finite Weyl group elements. In this section, we show that the polynomials $\widetilde{R}_{w, \wtrans}^J(Q_r)$ appearing in the formula for the coefficients of $\phironeaug$ can be written as a sum indexed by these paths.

\subsubsection{Dyer's formula}

This subsection is devoted to justifying Theorem~\ref{dyer-R-polynomial-formula}, Dyer's $\widetilde{R}$-polynomial formula, in the context of this article. We begin by explaining how the polynomials are defined for a general Coxeter system, then we compare this definition to the definition of $\widetilde{R}$-polynomials arising from the inversion formula for basis elements of a twisted affine Hecke algebra (see Definition~\ref{defin::r-poly-from-hecke-algebras}).

Given a Coxeter system $(\mathcal{W},\mathcal{S})$ and $w \in \mathcal{W}$, the \textbf{right descent set} of $w$ is defined as
$$
D_R(w) = \{s \in \mathcal{S} \mid \ell(ws) < \ell(w)\}.
$$

\begin{thm}
Let $(\mathcal{W}, \mathcal{S})$ be a Coxeter system. There is a unique family of polynomials $\{R_{u,v}(q^r)\}_{u,v \in \mathcal{W}}$ satisfying the following conditions:
\begin{enumerate}
\item $R_{u,v}(q^r) = 0$ if $u\nleq v$,
\item $R_{u,v}(q^r) = 1$ if $u = v$,
\item If $s \in D_R(v)$, then
$$
R_{u,v}(q^r) =
\begin{cases}
R_{us,vs}(q^r), & {\rm{if}}\ s \in D_R(u), \\
q^rR_{us,vs}(q^r) + (q^r-1)R_{u,vs}(q^r), & {\rm{if}}\ s \notin D_R(u).
\end{cases}
$$
\end{enumerate}
\end{thm}

\begin{proof}
This is Theorem 5.1.1 of~\cite{bjorner-brenti2005}.
\end{proof}

\begin{prop}
Let $u, v \in \mathcal{W}$, and let $\hat{Q} = q^{r/2} - q^{-r/2}$. There exists a unique polynomial $\widetilde{R}_{u,v}(q^r) \in \mathbb{N}[q]$ such that
$$R_{u,v}(q^r) = q^{r\ell(u,v)/2} \widetilde{R}_{u,v}(\hat{Q}_r).$$
\end{prop}

\begin{proof}
This is Proposition 5.3.1 of~\cite{bjorner-brenti2005}.
\end{proof}

Recall that in Chapter 2 we set $Q = q^{-r/2} - q^{r/2}$. In the definitions given for the Hecke algebra case, the $R$-polynomials and $\widetilde{R}$-polynomials are related by
$$
(-1)^{\ell(u)} (-1)^{\ell(v)} R_{u,v}(q^r) = q^{r\ell(u,v)/2} \widetilde{R}_{u,v}(Q_r),
$$
for $u$ and $v$ in the extended affine Weyl group $\widetilde{W}$, whereas in the definition of the polynomials for general Coxeter groups we have
$$
R_{u,v}(q^r) = q^{r\ell(u,v)/2} \widetilde{R}_{u,v}(\hat{Q}_r).
$$

\begin{lemma}
In the notation defined above,
$$
(-1)^{\ell(u)} (-1)^{\ell(v)} \widetilde{R}_{u,v} (Q_r) = \widetilde{R}_{u,v}(\hat{Q}_r).
$$
\end{lemma}

\begin{proof}
Notice that $\hat{Q}_r = -Q_r$, so it is enough to show that the sign $(-1)^{\ell(u)} (-1)^{\ell(v)}$ is somehow compatible with the individual terms in the polynomial. It will help to rewrite $(-1)^{\ell(u)} (-1)^{\ell(v)} = (-1)^{\ell(v) - \ell(u)} = (-1)^{\ell(u,v)}$.

Next, we recall two facts about $\widetilde{R}$-polynomials arising from the inversion formula for Hecke algebra basis elements. First, ${\rm{deg}}\big(\widetilde{R}_{u,v}(Q_r)\big) = \ell(u,v)$; see~\cite{haines2000b}, Lemma 2.5. Second, the powers of $Q_r$ in $\widetilde{R}_{u,v} (Q_r)$ all have the same parity. This follows from Theorem~\ref{dyer-R-polynomial-formula}, whose proof is independent of the current argument. Suppose we have two paths $\Delta_1, \Delta_2 \in B_{\Phi_{\chi, \rm{aff}}}^\prec(u, v)$, such that the product of edges of $\Delta_1$ is $t_1\ldots t_r$, while the product of edges of $\Delta_2$ is $u_1 \ldots u_s$. Then $$u^{-1}v = t_1\ldots t_r = u_1 \ldots u_s.$$ It is a general fact about Coxeter groups that in this situation $r - s$ must be even; but $r = \ell(\Delta_1)$ and $s = \ell(\Delta_2)$.

An $\widetilde{R}$-polynomial has coefficients $c_n$ in $\mathbb{N}$. Multiplying through by $(-1)^{\ell(u,v)}$ gives
$$
(-1)^{\ell(u,v)} c_n Q_r^{\ell(u,v) - 2n} = (-1)^{2n} c_n \hat{Q}_r^{\ell(u,v)-2n}.
$$
This completes the proof.
\end{proof}

Finally, we can state Dyer's formula~\cite{dyer1993} using the notation of Chapter 2.

\begin{thm}
\label{dyer-R-polynomial-formula}
Let $\widetilde{W}_{\chi}$ be the extended affine Weyl group of $H_{\chi_r}$, and let $\prec$ be a reflection ordering on the reflections in $W_{\chi, \rm{aff}}$. Let $Q_r = q^{-r/2} - q^{r/2}$. For any $u, v \in \widetilde{W}_{\chi}$ such that $u \leq_\chi v$ in Bruhat order,
$$\widetilde{R}^\chi_{u,v} (Q_r) = \sum_{\Delta \in B_{\Phi_{\chi, \rm{aff}}}^\prec(u, v)} Q_r^{\ell(\Delta)}.$$
\end{thm}

\begin{proof}
This statement is Theorem 5.3.4 of~\cite{bjorner-brenti2005}, applied to the case of $\widetilde{R}$-polynomials arising from $\mathcal{H}(H_{\chi_r}, I_{H_r})$ viewed as a twisted affine Hecke algebra.
\end{proof}

\subsubsection{Modifications to Dyer's formula}

Suppose $J$ is a root sub-system of $\Phi$. Let $W_J = \langle s_\alpha \mid \alpha \in J^+\rangle$ be the reflection subgroup of $W$ associated to $J$. The group $W_{J,{\rm{aff}}}$ is the corresponding affine Weyl group. Given a reflection ordering $\prec$ on reflections in $W_{\rm{aff}}$, we induce a reflection ordering on reflections in $W_{J, {\rm{aff}}}$ by restricting the ordering on $\Phi_{\rm{aff}}^+$ to $J_{\rm{aff}}^+$.

Let $w = t_\lambda \bar{w} \in {\rm{Adm}}_{G_r}(\mu)$, and $J \subseteq \Phi$. There is an automorphism $\sigma$ of the base alcove $\mathcal{C}$ such that $w$ and $t_\lambda$ lie in $\sigma W_{\rm{aff}}$, so we may speak of paths $w \stackrel{\Delta}{\longrightarrow} t_\lambda$ through the Bruhat graph $\Omega_{(W_{\rm{aff}}, S_{\rm{aff}})}$.  We consider Bruhat intervals
$$
B_{J_{{\rm{aff}}}}^{\prec} (w, t_\lambda) = \{ w \stackrel{\Delta}{\longrightarrow} t \mid E(\Delta) \subset W_{J,{\rm{aff}}},\ D(\Delta, \prec) = \emptyset \}.
$$

\begin{lemma}
\label{path-root-system-lemma}
For any path $\Delta$ in $B_{J_{\rm{aff}}}^{\prec} (w, t_\lambda)$, there is a unique minimal root subsystem $J_\Delta \subseteq J$ such that all reflections $t_i \in E(\Delta)$ lie in $W_{J_\Delta}$.
\end{lemma}

\begin{proof}
Fix a path $\Delta \in B_{J,{\rm{aff}}}^{\prec} (w, t_{\lambda})$. By Proposition~\ref{prop::finite-reflections-in-interval}, the edges of all paths in $B_{J,{\rm{aff}}}^{\prec} (w, t_{\lambda})$ are finite reflections. Observe that
$$W^\prime =\ \bigcap_{\substack{E(\Delta) \subset V \subset W_J}} V \ =\ \langle t_i \mid t_i \in E(\Delta) \rangle,$$
is the smallest subgroup of $W_J$ containing all of the edges in $\Delta$. Theorem~\ref{reflection-subgroups-are-coxeter-systems} shows that $W^\prime$ is itself a Coxeter group. Therefore, by Lemma~\ref{lemma::reflection-subgroup-gives-root-subsystem}, there is an associated root system $J_\Delta$ whose positive roots are in bijection with the reflections of $W^\prime$.
\end{proof}

\begin{prop}
\label{stratified-r-polynomial-formula}
Let $w=t_\lambda \bar{w} \in {\rm{Adm}}_{G_r}(\mu)$. The polynomial $\widetilde{R}_{w,t_\lambda}^J (Q_r)$, defined with respect to the reflection subgroup $(W_J, \Sigma_J)$, can be rewritten as
$$
\widetilde{R}_{w,t_\lambda}^J (Q_r) = \sum_{J^\prime \subseteq J}\; \sum_{\substack{\Delta \in B_{J}^{\prec} (w,t_\lambda)\\ J_\Delta = J^\prime}} Q_r^{\ell(\Delta)}.
$$
\end{prop}

\begin{proof}
This is a direct consequence of Dyer's formula and the preceding lemmas.
\end{proof}

\begin{lemma}
\label{invariance-of-paths}
Let $w = t_\lambda \bar{w} \in {\rm{Adm}}_{G_r}(\mu)$. For any chain of root systems $J^\prime \subseteq J \subseteq \Phi$, there is an equality of sets
$$
\{\Delta \in B_J^{\prec} (w_\lambda^{-1}\wbar, w_\lambda^{-1}) \mid J_\Delta = J^\prime \} = \{\Delta \in B_\Phi^{\prec} (w_\lambda^{-1}\wbar, w_\lambda^{-1}) \mid J_\Delta = J^\prime\}.
$$
\end{lemma}

\begin{proof}
The reflection subgroups $W_J$ and $W_{J^\prime}$ can be realized as Coxeter groups $(W_J, \Sigma_J)$ and $(W_{J^\prime}, \Sigma_{J^\prime})$. Let $\Omega_{(W,S)}$ denote the Bruhat graph of $(W,S)$ and use analogous notation for graphs of reflection subgroups.

The key observation comes from Theorem~\ref{thm::bruhat-graph-of-reflection-subgroups}: The Bruhat graph of a reflection subgroup is equal to the full subgraph of the Bruhat graph of an ambient Coxeter group having vertices in the reflection subgroup. Symbolically, this says:
$$
\Omega_{(W,S)} (W_{J^\prime}) = \Omega_{(W_{J^\prime}, \Sigma_{J^\prime})} = \Omega_{(W_J, \Sigma_J)} (W_{J^\prime}).
$$
Therefore, the set of paths through $\Omega_{(W,S)}$ associated to $J^\prime$ equals the set of paths through $\Omega_{(W_J, \Sigma_J)}$ associated to $J^\prime$.
\end{proof}

\begin{cor}
\label{cor::finite-interval-r-polynomial}
Let $w = t_\lambda \bar{w} \in {\rm{Adm}}_{G_r}(\mu)$ and $J \subseteq \Phi$. Then
$$
\widetilde{R}_{w,t_\lambda}^J (Q_r) = \sum_{J^\prime \subseteq J}\; \sum_{\substack{\Delta \in B_\Phi^\prec (w_\lambda^{-1}\wbar, w_\lambda^{-1})\\ J_\Delta = J^\prime}} Q_r^{\ell(\Delta)}.
$$
\end{cor}

\begin{proof}
Apply Lemma~\ref{invariance-of-paths} to rewrite the formula of Proposition~\ref{stratified-r-polynomial-formula} with $B_\Phi^\prec (w, t_\lambda)$ instead of $B_J^\prec (w, t_\lambda)$. Then apply Proposition~\ref{bruhat-interval-isomorphism}.
\end{proof}


\section{The combinatorial formula and example calculations}
\label{chapter::main-theorem}

Let us recall what we have done leading up to this final chapter. First, we gave an abstract definition of a test function $\phi_r$ via the Bernstein center and then explicitly described the test function in the case of $I_r^+$-level structure for split connected reductive groups with connected center by applying results on depth-zero characters and invoking Haines's formula for Bernstein functions attached to dominant minuscule cocharacters. We then embarked on two (mostly) independent paths. By looking more closely at the depth-zero endoscopic elements $\kapchi$, we obtained information on which $s \in T(k_r)$ help characterize the nonzero coefficients of $\phironeaug$, plus we simplified the sum in the first explicit formula through a stratification process indexed by $\chi$-root systems. This stratification by $\chi$-root systems was also employed in the subsequent chapter to adapt a formula of Dyer concerning $\widetilde{R}$-polynomials. All of these threads will now come together into the main theorem.

The proof of the Main Theorem amounts to showing that the various stratifications and corresponding summations behave in a compatible way. By swapping sums and subsequently reorganizing terms, objects introduced in Chapters 3 and 4 begin to appear in the formula. Finally, all but one summation drops out of the formula for $\phironeaug(I_r^+ sw I_r^+)$. This remaining sum is over a well-studied combinatorial set connected with the Bruhat order of the \emph{ambient} Weyl group.

After proving the formula, we make some remarks about using it in practice. The chapter concludes with some calculations of interest to the study of Shimura varieties in the cases of $GL_n$ and $GSp_{2n}$.


\subsection{The Main Theorem}

Throughout this section, let $G$ be a split connected reductive group with connected center and whose derived group $G_{\rm{der}}$ is simply-connected, and assume $W_\chi^\circ = W_\chi$ (see Remark~\ref{rmk::roche-assumptions}). Fix a split maximal torus $T \subset G$ of rank $d$. Let $\mu$ be a dominant minuscule cocharacter of $T$. Finally, choose a reflection ordering $\prec$ on $\Phi(G,T)$.

Recall that Proposition~\ref{prop::first-explicit-formula} presented our first explicit formula for the coefficients of $\phironeaug$, the function used in place of ${\phirone = q^{\ell(t_{\mu})/2}(Z_{V_\mu} \ast 1_{I_r^+})}$ when computing twisted orbital integrals:
$$
\phironeaug (I_r^+ sw I_r^+) = [I_r: I_r^+]^{-1} \sum_{\kapchi \in K_{q-1}} \gamma_{N_r s} (\kapchi)^{-1} q^{r\ell(w,\wtrans)/2} \widetilde{R}_{w,\wtrans}^\chi(Q_r).
$$

\subsubsection{Proof of the main theorem}

The notation $I_r^+ sw I_r^+$ treats $w \in \widetilde{W}$ as an element of $G_r$ using the set-theoretic embedding $\widetilde{W} \hookrightarrow G_r$ defined in Definition~\ref{defin::set-theoretic-embedding}.

\begin{lemma}
For $w \in {\rm{Adm}}_{G_r}(\mu)$ and $s \in \torus(k_r)$, we have
$$
\phironeaug (I_r^+ sw I_r^+) = [I_r: I_r^+]^{-1} \sum_{J \subseteq \Phi} \sum_{\kapchi \in \dzrelgrp(J)} \gamma_{N_r s}(\kapchi)^{-1} q^{r\ell(w,\wtrans)/2} \widetilde{R}_{w,\wtrans}^J(Q_r).
$$
\end{lemma}

\begin{proof}
According to Lemma~\ref{lemma::support-of-phi-r-chi}, $Z_{V_\mu} \ast e_{\chi_r} (w) = 0$ if $\kappa_\chi$ is not relevant to $w$, i.e., if $\kappa_\chi \notin S_w^{\rm{dz}}$. Using this, we may replace the sum in the formula from Proposition~\ref{prop::first-explicit-formula} with $\dzrelgrp$. Then $\phironeaug (I_r^+ sw I_r^+)$ equals
$$
[I_r: I_r^+]^{-1} \sum_{\kapchi \in \dzrelgrp} \gamma_{N_r s} (\kapchi)^{-1} q^{r\ell(w,\wtrans)/2} \widetilde{R}_{w,\wtrans}^\chi(Q_r).
$$
Next, stratify $\dzrelgrp$ as described in Section~\ref{subsection::stratification-by-chi-root-systems}, i.e.,
$\dzrelgrp = \coprod_{J \subseteq \Phi} \dzrelgrp(J)$, so that the coefficient equals,
$$
[I_r: I_r^+]^{-1} \sum_{J \subseteq \Phi} \sum_{\kapchi \in \dzrelgrp(J)} \gamma_{N_r s} (\kapchi)^{-1} q^{r\ell(w,\wtrans)/2} \widetilde{R}_{w,\wtrans}^\chi(Q_r).
$$
Finally, recall that the polynomials $\widetilde{R}_{w,\wtrans}^\chi(Q_r)$ are identical for all $\kapchi$ in $\dzrelgrp(J)$ by Lemma~\ref{lemma::equality-on-strata}. Specifically, they are the polynomial $\widetilde{R}_{w,\wtrans}^J(Q_r)$ with $J = \Phi_\chi$. This completes the proof.
\end{proof}

\begin{lemma}
\label{lemma::formula-after-chapter-4}
For $w \in {\rm{Adm}}_{G_r}(\mu)$ and $s \in \torus(k_r)$, $\phironeaug (I_r^+ sw I_r^+)$ equals
$$
[I_r: I_r^+]^{-1} \sum_{J \subseteq \Phi} \sum_{\kapchi \in \dzrelgrp(J)} \gamma_{N_r s}(\kapchi)^{-1} q^{r\ell(w,\wtrans)/2} \sum_{J^\prime \subseteq J}\; \sum_{\substack{\Delta \in B_\Phi^\prec (w_\lambda^{-1}\wbar, w_\lambda^{-1})\\ J_\Delta = J^\prime}} Q_r^{\ell(\Delta)}.
$$
\end{lemma}

\begin{proof}
Replace $\widetilde{R}_{w,\wtrans}^J(Q_r)$ with the formula from Corollary~\ref{cor::finite-interval-r-polynomial}.
\end{proof}

For each path $\Delta$ in $B_\Phi^\prec(w, t_{\lambda(w)})$, we defined a root system $J_\Delta \subseteq \Phi$, which in turn determines a diagonalizable subgroup $S_{w, J_\Delta}$ of $\dualtorus(\valfield)$. Let $S_{w, J_\Delta}^{\rm{tors}}$ denote the torsion subgroup of $S_{w, J_\Delta}$. If $w = t_{\lambda(w)}$, then $B_\Phi^\prec(w, t_{\lambda(w)})$ contains only a single element, the ``empty path,'' whose length is zero and has trivial root system $J_\Delta = \emptyset$; see also Corollary~\ref{cor::main-theorem-codim-0}.

Let $w \in \widetilde{W}$, $s \in \torus(k_r)$, and $J \subseteq \Phi$. Definition~\ref{defin::finite-critical-groups} introduced the finite critical groups $A_{w,J,\resfield}$, and we saw in Proposition~\ref{prop::sum-over-group} that there are consequences of $N_r (s)$ belonging to $A_{w,J,\resfield}$ or not. Define the symbol $\delta(s, w, J)$ by
$$
\delta(s, w, J) = \begin{cases}
0, & {\rm{if}}\ w\notin {\rm{Adm}}_{G_r}(\mu) \\
0, & {\rm{if}}\ w \in {\rm{Adm}}_{G_r}(\mu)\ {\rm{and}}\ N_r(s) \notin A_{w, J, \resfield} \\
1, & {\rm{if}}\ w \in {\rm{Adm}}_{G_r}(\mu)\ {\rm{and}}\ N_r(s) \in A_{w, J, \resfield}.
\end{cases}
$$

Recall that the main theorem holds under the conditions of Remark~\ref{rmk::roche-assumptions}.

\begin{thm}
\label{main-theorem}
Let $w \in \widetilde{W}$ and $s \in \torus(k_r)$. Let $d$ be the rank of $T$. Fix a reflection ordering $\prec$ on $\Phi$, and set $c(\Delta) = \left[\ell(w, t_\mu) - \ell(\Delta)\right]/2$. The coefficient of $\phironeaug$ for the $I_r^+$-double coset of $(s,w)$ is given by
$$
(-1)^d \sum_{\Delta \in B_\Phi^{\prec} (w_\lambda^{-1}\wbar, w_\lambda^{-1})} \delta(s, w,J_\Delta) \vert S_{w,J_\Delta}^{\rm{tors}} \cap K_{q-1}\vert (q-1)^{d-{\rm{rank}}(J_\Delta) -1} q^{r c(\Delta)} (1-q^r)^{\ell(\Delta)-d}.
$$
\end{thm}

\begin{proof}
If $w \notin {\rm{Adm}}_{G_r}(\mu)$, we know $\phi_{r,1}^\prime(I_r^+ sw I_r^+) = 0$ for all $s \in \torus(k_r)$ by combining Corollary~\ref{cor::equality-of-orbital-integrals} and Lemma~\ref{lemma::image-of-phi-r-chi}. Thus we assume $w \in {\rm{Adm}}_{G_r}(\mu)$, and our starting point is the version of the formula given in Lemma~\ref{lemma::formula-after-chapter-4}. The first phase of the proof involves rearranging the four summations therein.

First, observe that the sum indexed by endoscopic elements in the stratum $\dzrelgrp(J)$ does not depend on the choice of $J^\prime \subseteq J$. Exchanging the corresponding summations and moving all terms to the innermost quantity shows that $\phironeaug(I_r^+ sw I_r^+)$ equals
$$
[I_r: I_r^+]^{-1} \sum_{J \subseteq \Phi} \;\sum_{J^\prime \subseteq J}\;\sum_{\kapchi \in \dzrelgrp(J)}\;\sum_{\substack{\Delta \in B_\Phi^{\prec} (w_\lambda^{-1}\wbar, w_\lambda^{-1})\\ J_\Delta = J^\prime}} \gamma_{N_r s}(\kapchi)^{-1} q^{r\ell(w,t_{\lambda(w)})/2}Q_r^{\ell(\Delta)}.
$$
Next, we rewrite the first two sums of this expression as follows: Instead of summing over $J \subseteq \Phi$ and then summing over $J^\prime \subseteq J$; first sum over $J^\prime \subseteq \Phi$ and then $J \supseteq J^\prime$. The new expression is
$$
[I_r: I_r^+]^{-1} \sum_{J^\prime \subseteq \Phi} \;\sum_{J \supseteq J^\prime}\;\sum_{\kapchi \in \dzrelgrp(J)}\;\sum_{\substack{\Delta \in B_\Phi^{\prec} (w_\lambda^{-1}\wbar, w_\lambda^{-1})\\ J_\Delta = J^\prime}} \gamma_{N_r s}(\kapchi)^{-1} q^{r\ell(w,t_{\lambda(w)})/2}Q_r^{\ell(\Delta)}.
$$

The innermost sum does not depend on the choice of $J$ containing a fixed $J^\prime$. Therefore we may move it through the two adjacent summations. That is,
\begin{multline*}
[I_r: I_r^+]^{-1} \sum_{J^\prime \subseteq \Phi} \;\sum_{J \supseteq J^\prime}\;\sum_{\kapchi \in \dzrelgrp(J)}\;\sum_{\substack{\Delta \in B_\Phi^{\prec} (w_\lambda^{-1}\wbar, w_\lambda^{-1})\\ J_\Delta = J^\prime}} \gamma_{N_r s}(\kapchi)^{-1} q^{r\ell(w,t_{\lambda(w)})/2}Q_r^{\ell(\Delta)} \\
 = [I_r: I_r^+]^{-1} \sum_{J^\prime \subseteq \Phi}  \sum_{\substack{\Delta \in B_\Phi^{\prec} (w_\lambda^{-1}\wbar, w_\lambda^{-1})\\ J_\Delta = J^\prime}} q^{r\ell(w,t_{\lambda(w)})/2} Q_r^{\ell(\Delta)} \left( \sum_{J \supseteq J^\prime}\;\sum_{\kapchi \in \dzrelgrp(J)} \gamma_{N_r s}(\kapchi)^{-1}\right).
\end{multline*} 
Corollary~\ref{cor::union-of-strata} simplifies the quantity in parentheses:
$$
[I_r: I_r^+]^{-1} \sum_{J^\prime \subseteq \Phi}  \sum_{\substack{\Delta \in B_\Phi^{\prec} (w_\lambda^{-1}\wbar, w_\lambda^{-1})\\ J_\Delta = J^\prime}} q^{r\ell(w,t_{\lambda(w)})/2} Q_r^{\ell(\Delta)} \left(\sum_{\kapchi \in S_{w, J^\prime}^{\rm{dz}}} \gamma_{N_r s}(\kapchi)^{-1}\right)
$$

The second phase of the proof simplifies the preceding expression by applying our results about paths through the Bruhat graph and sums of character values.

Recall that $\ell(t_\mu) = \ell(t_\lambda)$ for all $\lambda \in W\mu$, so that we may work with $\ell(w, t_\mu)$ in all cases rather than $\ell(w, t_{\lambda(w)})$. Observe that the difference of lengths
$$
\ell(w, t_\mu) = \ell(t_\mu) - \ell(w)
$$
has the same parity as every path length $\ell(\Delta)$. (This a rephrasing of our earlier statement that the orders of terms in $\widetilde{R}$-polynomials all have the same parity.) Thus $c(\Delta) = \left[\ell(w, t_\mu) - \ell(\Delta)\right]/2$ is a nonnegative integer. We also apply Corollary~\ref{prop::sum-over-group} to the quantity $\left(\sum_{\kapchi \in S_{w, J^\prime}^{\rm{dz}}} \gamma_{N_r s}(\kapchi)^{-1}\right)$, which implies
$$
\left(\sum_{\kapchi \in S_{w, J^\prime}^{\rm{dz}}} \gamma_{N_r s}(\kapchi)^{-1}\right) = \delta(s,w,J^\prime) \vert S_{w, J^\prime}^{\rm{dz}} \vert.
$$
Finally, use that $I_r/I_r^+ \cong \torus(k_r)$ and that $\torus$ is a split maximal torus to see
$$
[I_r : I_r^+] = (q^r -1)^d.
$$
The result is that $\phironeaug(I_r^+ sw I_r^+)$ equals
$$
(-1)^d \sum_{J^\prime \subseteq \Phi}\;\sum_{\substack{\Delta \in B_\Phi^{\prec} (w_\lambda^{-1}\wbar, w_\lambda^{-1})\\ J_\Delta = J^\prime}} \delta(s,w,J^\prime) \vert S_{w, J^\prime}^{\rm{dz}} \vert q^{rc(\Delta)} (1-q^r)^{\ell(\Delta)-d}
$$
For simplicity of notation, relabel all $J^\prime$ as $J$,
$$
(-1)^d \sum_{J \subseteq \Phi}\;\sum_{\substack{\Delta \in B_\Phi^{\prec} (w_\lambda^{-1}\wbar, w_\lambda^{-1})\\ J_\Delta = J}} \delta(s,w,J) \vert \dzrelgrpJ\vert q^{rc(\Delta)} (1-q^r)^{\ell(\Delta)-d}.
$$

The third phase of the proof simplifies the double summation in the previous expression.

Recall that there is a well-defined root subsystem $J_\Delta \subseteq \Phi$ associated to each path $\Delta$ in $B_\Phi^\prec(w_\lambda^{-1}\wbar, w_\lambda^{-1})$. This relationship partitions the $\prec$-increasing paths:
$$
B_\Phi^\prec (w_\lambda^{-1}\wbar, w_\lambda^{-1}) = \coprod_{J \subseteq \Phi} \left\{ \Delta \in B_\Phi^\prec (w_\lambda^{-1}\wbar, w_\lambda^{-1}) \mid J_\Delta = J\right\}.
$$ 
For a fixed $J \subseteq \Phi$, suppose that $J_\Delta \neq J$ for all $\Delta \in B_\Phi^\prec(w_\lambda^{-1}\wbar, w_\lambda^{-1})$. Then
$$
\delta(s, w,J) \vert \dzrelgrpJ\vert \sum_{\substack{\Delta \in B_\Phi^{\prec} (w_\lambda^{-1}\wbar, w_\lambda^{-1})\\ J_\Delta = J}} q^{rc(\Delta)} (1-q^r)^{\ell(\Delta)-d} = 0.
$$
On the other hand, for any $J$ where there exist paths $\Delta$ such that $J_\Delta = J$, we have an equality
\begin{multline*}
\delta(s, w,J) \vert \dzrelgrpJ\vert \sum_{\substack{\Delta \in B_\Phi^{\prec} (w_\lambda^{-1}\wbar, w_\lambda^{-1})\\ J_\Delta = J}} q^{rc(\Delta)} (1-q^r)^{\ell(\Delta)-d} \\ 
= \sum_{\substack{\Delta \in B_\Phi^{\prec} (w_\lambda^{-1}\wbar, w_\lambda^{-1})\\ J_\Delta = J}} \delta(s, w,J_\Delta) \vert S_{w, J_\Delta}^{\rm{dz}} \vert q^{rc(\Delta)} (1-q^r)^{\ell(\Delta)-d}
\end{multline*}

\noindent So we have shown that $\phironeaug(I_r^+ sw I_r^+)$ equals
$$
(-1)^d \sum_{\Delta \in B_\Phi^{\prec} (w_\lambda^{-1}\wbar, w_\lambda^{-1})} \delta(s, w,J_\Delta) \vert S_{w, J_\Delta}^{\rm{dz}} \vert  q^{rc(\Delta)} (1-q^r)^{\ell(\Delta)-d},
$$
because $B_\Phi^\prec(w_\lambda^{-1}\wbar, w_\lambda^{-1})$ splits as a disjoint union indexed by $J \subseteq \Phi$.

Finally, let us rewrite $\vert S_{w, J_\Delta}^{\rm{dz}} \vert$. The diagonalizable group $S_{w, J_\Delta} \subseteq \dualtorus(\valfield)$ factors into a direct product of a torus $S_{w, J_\Delta}^\circ$ and a torsion subgroup $S_{w,J_\Delta}^{\rm{tors}}$. Because the group of depth-zero endoscopic elements $S_w^{\rm{dz}}$ in $\dualtorus(\valfield)$ equals the kernel $K_{q-1}$ of the endomorphism on $\dualtorus(\valfield)$ given by multiplication by $(q-1)$,
$$
S_{w, J_\Delta}^{\rm dz} = (S_{w,J_\Delta}^{\rm{tors}} \cap K_{q-1}) \times (S_{w, J_\Delta}^\circ \cap K_{q-1}).
$$
But since ${\rm{rank}}(S_{w, J_\Delta}) = d - {\rm{rank}}(J_\Delta) - 1$ by Corollary~\ref{cor::rank-of-swj}, and $S_{w, J_\Delta}^\circ$ is the connected component, we must have $\vert S_{w, J_\Delta}^\circ \cap K_{q-1} \vert = (q-1)^{d - {\rm{rank}}(J_\Delta) - 1}$; so,
$$
\vert S_{w, J_\Delta}^{\rm{dz}} \vert = \vert S_{w, J_\Delta}^{\rm{tors}} \cap K_{q-1} \vert (q-1)^{d - {\rm{rank}}(J_\Delta) - 1}.
$$
This completes the proof of the main theorem.
\end{proof}


\subsubsection{The Drinfeld Case}
\label{section::the-drinfeld-case}

The formula for test functions in the Drinfeld case found by Haines and Rapoport is a special case of Theorem~\ref{main-theorem}. Their expression depends on the ``set of critical indices'' $S(w)$ associated to a $\mu$-admissible element $w \in \widetilde{W}$ and the corresponding subtorus $T_{S(w)}$ of the split maximal torus $T$ in $G$. Recall the Drinfeld case data: $G = GL_d$, $\mu = (1,0,\dots,0)$, and $k_F \cong \mathbb{F}_p$. Let $e_1 = (1, 0, \ldots)$, $e_2 = (0,1,0,\ldots)$, and so on.

\begin{defin}
In the Drinfeld case, for $w \in {\rm{Adm}}_{G_r}(\mu)$, the set of \textbf{critical indices} is the subset $S(w) \subseteq \{1, \ldots, d\}$ given by
$$
S(w) = \{j \mid w \leq t_{e_j}\}.
$$
The subtorus $T_{S(w)}$ consists of the elements ${\rm{diag}}(t_1, \ldots, t_d) \in T$ such that $t_i = 1$ for all $i \notin S(w)$. See Section 6 of~\emph{\cite{haines-rapoport2012}} for more details.
\end{defin}

\begin{cor} \emph{(Haines-Rapoport, \cite{haines-rapoport2012} 12.2.1)}
\label{cor::drinfeld-case}
With respect to the Haar measure $dx$ on  $G_r$ which gives $I_r^+$ volume equal to 1, the function $\phironeaug$ is given by the formula
$$
\phironeaug(I_r^+ sw I_r^+) =
\begin{cases}
0,\ {\rm{if}}\ w \notin {\rm{Adm}}_{G_r}(\mu) \\
0,\ {\rm{if}}\ w \in {\rm{Adm}}_{G_r}(\mu)\ {\rm{and}}\ N_r(s) \notin T_{S(w)}(k_F) \\
(-1)^{d} (p-1)^{d-\vert S(w)\vert} (1-p^r)^{\vert S(w)\vert-d-1},\ {\rm{otherwise}}. \\
\end{cases}
$$
\end{cor}

\begin{proof}
Admissible elements in the Drinfeld case have the form
$$
w = t_{e_m} (m m_{k-1} \cdots m_1),
$$
where $m > m_{k-1} > \cdots > m_1$. The proof of \cite{haines2000b}, Proposition 5.2, shows that, in this case,
$$
\widetilde{R}_{w, t_{\lambda(w)}}(Q) = Q^{\ell(w, t_{\lambda(w)})}.$$
But Theorem~\ref{dyer-R-polynomial-formula} gives this polynomial in terms of the set $B_\Phi^\prec (w, t_{\lambda(w)})$ for any choice of reflection ordering $\prec$. It follows that $B_\Phi^\prec (w, t_{\lambda(w)})$ consists of a single path of length $\ell(w, t_{\lambda(w)})$. Therefore, if $w \in {\rm{Adm}}_{G_r}(\mu)$ we have
$$
\phironeaug(I_r^+ sw I_r^+) = (-1)^{d} (1-q)^{-d} \delta(s,w,\Phi) \vert \dzrelgrp \vert (1-q)^{\ell(w,t_{\lambda(w)})}.
$$
Haines and Rapoport show that $\ell(w,t_{\lambda}) = \vert S(w)\vert - 1$.

It is clear that in the Drinfeld case $S_w$ is always a torus, hence we can define a torus $T_{S(w)}$ that fits into an exact sequence
$$
1 \rightarrow T_{S(w)} \rightarrow T \rightarrow Q_w \rightarrow 1,
$$
such that $Q_w$ is dual to $\dualtorus/S_w$. We conclude that $\vert S_w^{\rm{dz}} \vert = (p-1)^{d-\vert S(w)\vert}$ by the relations imposed by $\lambda(\kappa) = 1$ and $\wbar(\kappa) = \kappa$.

Now apply the results of Chapter 3: if $w \in {\rm{Adm}}_{G_r}(\mu)$ and $N_r(s) \in T_{S(w)}(k_F)$, then we have
$$
\phironeaug(I_r^+ sw I_r^+) = (-1)^{d} (p-1)^{d-\vert S(w)\vert} (1-q)^{\vert S(w)\vert-d-1},
$$
and the coefficient is zero otherwise.
\end{proof}

\subsubsection{Relationship to test functions with Iwahori level structure}
\label{section::rel-to-fn-iwahori-level}

Let us say something about the relationship between the formula for coefficients $\phironeaug$, which is sufficient for computing twisted orbital integrals of the test function $\phirone$, and the coefficients of the test function $\phi_{r,0}$ with \emph{Iwahori} level structure.

Let $\phi_{r,0} = q^{r\ell(t_\mu)/2} z_{\mu, r}$ be the Kottwitz function, where $\mu$ is a dominant minuscule cocharacter of $G$ as usual and $z_{\mu, r}$ is a Bernstein function in the center of $\mathcal{H}(G_r, I_r)$ as defined in Section~\ref{section::bernstein-functions}. Haines's formula for Bernstein functions of minuscule cocharacters shows that
$$
\phi_{r,0} = q^{r\ell(t_\mu)/2} \sum_{w \in {\rm{Adm}_{G_r}(\mu)}} \widetilde{R}_{w, t_{\lambda(w)}}(Q_r) \tilde{T}_{w, r},
$$
where again notation is the same as in Chapter~\ref{chapter::test-functions}. Its coefficients are
$$
\phi_{r,0}(I_r w I_r) = q^{r\ell(t_\mu)/2} q^{-r\ell(w)/2} \widetilde{R}_{w, t_{\lambda(w)}}(Q_r).
$$
The term $q^{-r\ell(w)/2}$ comes from the normalization of basis elements in the Iwahori-Matsumoto presentation of $\mathcal{H}(G_r, I_r)$. Dyer's formula for $\widetilde{R}$-polynomials, discussed in Chapter~\ref{chapter::combinatorics-on-paths}, implies
$$
\phi_{r,0}(I_r w I_r) = q^{r\ell(t_\mu)/2} q^{-r\ell(w)/2} \sum_{\Delta \in B_{\Phi_{\rm{aff}}}^\prec (w, t_{\lambda(w)})} Q_r^{\ell(\Delta)}.
$$
Let $c(\Delta) = [\ell(w, t_\mu) - \ell{\Delta}]/2$ as in Theorem~\ref{main-theorem} to get
$$
\phi_{r,0}(I_r w I_r) = \sum_{\Delta \in B_{\Phi_{\rm{aff}}}^\prec (w, t_{\lambda(w)})} q^{rc(\Delta)} (1-q^r)^{\ell(\Delta)}
$$
So we see that the formula for coefficients of $\phi_{r,0}$ have a similar structure to those of $\phironeaug$. The latter can be written as
$$
\phironeaug(I_r^+ sw I_r^+) = [I_r : I_r^+]^{-1} \sum_{\Delta \in B_{\Phi_{\rm{aff}}}^\prec (w, t_{\lambda(w)})} \delta(s, w, J_\Delta) \vert S_{w, J_\Delta}^{\rm{dz}} \vert q^{rc(\Delta)} (1-q^r)^{\ell(\Delta)}.
$$


\subsection{Remarks on applying the formula}
\label{section::implementation-remarks}

Now that we have established the combinatorial formula for coefficients of $\phironeaug$, let us say a few words about calculating values of the formula.

\subsubsection{Some special cases}

Understanding the $\prec$-increasing paths through the Bruhat graph is an important first step in calculating results with Theorem~\ref{main-theorem}. In some cases, $B_{\Phi_{\rm{aff}}}^\prec(w, \wtrans)$ has a very simple description. The exercises following Chapter 5 of \cite{bjorner-brenti2005} describe $\widetilde{R}_{u,v}(Q)$ for elements $u, v \in W$ such that $\ell(u,v) \leq 4$; the following two corollaries consider the cases $\ell(w, t_{\lambda(w)}) = 0$ and $\ell(w, t_{\lambda(w)}) = 1$.

\begin{cor}
\label{cor::main-theorem-codim-0}
Suppose $w = t_\lambda$ for $\lambda \in W \mu$. Then
$$
\phironeaug(I_r^+ s t_\lambda I_r^+) = \begin{cases}
(-1)^d (q-1)^{d-1} (1-q^r)^{-d}, & {\rm{if}}\ N_r(s) \in A_{w, \emptyset, k_F} \\
0, & {\rm{otherwise}}.
\end{cases}
$$
\end{cor}

\begin{proof}
There are no non-trivial $\prec$-increasing paths from $t_\lambda$ to itself. Therefore, the formula reduces to
$$
[I_r : I_r^+]^{-1} \sum_{\kapchi \in S_w^{\rm{dz}}} 1 = (-1)^d (1 - q^r)^{-d} \vert S_{w,J}^{\rm{dz}}\vert.
$$

The relation $\lambda (\kappa_\chi) = 1$ is the only restriction on depth-zero endoscopic elements. If we view $\kappa_\chi = {\rm{diag}}(\kappa_1, \ldots, \kappa_d) \in \dualtorus(\valfield)$, then this restriction allows for a free choice of all $\kappa_i$ except for one coordinate, which is determined by the relation. It follows that $S_w$ is a torus, hence $\vert S_w^{\rm{dz}} \vert = (q-1)^{d-1}$.
\end{proof}

\begin{cor}
\label{cor::main-theorem-codim-1}
If $w = t_\lambda x$, for $x$ a reflection in $W$, such that $\ell(w, t_{\lambda}) = 1$, then
$$
\phironeaug(I_r^+ s w I_r^+) = \begin{cases}
(-1)^d \vert S_{w, J_\Delta}^{\rm{tors}} \cap K_{q-1} \vert (q-1)^{d-2} (1-q^r)^{1-d}, & {\rm{if}}\ N_r(s) \in A_{w,J_\Delta, k_F} \\
0, & {\rm{otherwise}}.
\end{cases}
$$
\end{cor}

\begin{proof}
The interval $B_{\Phi_{\rm{aff}}}^\prec(w, t_\lambda)$ contains a single path $\Delta = \{w, wx\}$. It is clear that $J_\Delta$ is a rank 1 system determined by the root $\alpha \in \Phi$ such that $x = s_\alpha$. Applying the formula shows that $\phi_{r,1}(I_r^+ sw I_r^+)$ equals
$$
(-1)^d \vert S_{w, J_\Delta}^{\rm{tors}} \cap K_{q-1} \vert (q-1)^{d-2} (1-q^r)^{1-d},
$$
because $[\ell(w, t_{\lambda}) - \ell(\Delta)]/2 = [1 - 1]/2 = 0$.
\end{proof}

\subsubsection{Implementation details}

Let us describe how to implement the formula in software. The data presented in the next sections was computed using SageMath~\cite{sagemath}; however, this process should be feasible in any mathematics package with a robust implementation of Coxeter groups and finitely generated abelian groups.

The first order of business is to enumerate the $\mu$-admissible set, because if $w \notin \rm{Adm}(\mu)$ then $\phironeaug(I_r^+ sw I_r^+) = 0$. By definition,
$$
{\rm{Adm}}(\mu) = \{w \in \widetilde{W} \mid w \le t_\lambda,\; \rm{some}\; \lambda \in W\mu\}.
$$
Therefore, we can employ the following naive algorithm:
\begin{enumerate}
\item Determine the orbit $W\mu$.
\item For every $\lambda \in W\mu$ and every $\bar{w} \in W$, compute $t_\lambda \bar{w}$.
\item If $t_\lambda \bar{w} \le t_\lambda$, then $w = t_\lambda \bar{w} \in {\rm{Adm}}(\mu)$.
\end{enumerate}
This approach is fast enough to handle small rank cases. For type $A_n$ systems, $\vert W \vert = (n+1)!$, which means this exhaustive strategy would quickly become infeasible.

Now the algorithm proceeds in parallel for each $w = t_\lambda \bar{w} \in {\rm{Adm}}(\mu)$. For each $w \in {\rm{Adm}}(\mu)$, compute $\ell(w)$ and $\ell(t_{\lambda})$; this determines the codimension $\ell(w, t_{\lambda})$. 

Next, we must find the minimal length coset representative of $t_\lambda$ with respect to the finite Weyl group, i.e., find $w_\lambda \in W$ such that $t_\lambda w_\lambda \in \widetilde{W}$ has minimal length. Once $w_\lambda^{-1}$ inverse is in hand, we can focus our attention on the finite group $W$ and the elements $w_\lambda^{-1} \bar{w}$ and $w_\lambda^{-1}$.

In order to compute $B_\Phi^\prec (w_\lambda^{-1}\bar{w}^{-1}, w_\lambda^{-1})$, we need to choose a reflection ordering $\prec$ for the finite Weyl group. Fortunately, one of Dyer's results provides a straightforward algorithm for making a consistent choice across all root system types and for all ranks.

\begin{prop}
\label{prop::refl-order-from-high-root}
Let $(\mathcal{W},\mathcal{S})$ be a finite Coxeter system with longest element $w_0$, and let $\mathcal{T} = \{t_1, \ldots, t_n\}$ be the set of reflections in $\mathcal{W}$. Then the total ordering $\prec$ on $\mathcal{T}$ such that $t_1 \prec \ldots \prec t_n$ is a reflection ordering if and only if there is a reduced expression $w_0 = s_1 \ldots s_n$, where $s_i \in \mathcal{S}$, such that $t_i = s_1 \ldots s_{i-1} s_i s_{i-1} \ldots s_1$.
\end{prop}

\begin{proof}
This is \cite{dyer1993}, Proposition 2.13.
\end{proof}

We come now to the main combinatorial part of the algorithm: enumerating the $\prec$-increasing paths through the Bruhat graph of the finite Weyl group. This is hard insofar as the Bruhat graph of $(W,S)$ grows rapidly in complexity as the rank of the group increases, but the basic problem has been studied due to the connection with Kazhdan-Lusztig theory.

The naive approach of creating the full Bruhat graph and then considering all paths is very expensive even in small examples. Instead, we take advantage of the relative scarcity of reflections in $W$ compared to $\vert W \vert$. For example, there are $\frac{n(n+1)}{2}$ reflections in a Weyl group of type $A_n$ while the group has order $(n+1)!$.

For elements $u, v \in W$, enumerate $B_\Phi^\prec(u, v)$ as follows:
\begin{enumerate}
\item Let $C$ denote the set of ``candidate paths'' in the Bruhat graph $\Omega_{(W,S)}$. This set is initially empty and will be built up as the algorithm proceeds.
\item For each reflection $t \in T$, if $u < ut$ in Bruhat order, add $\{u, ut\}$ to $C$.
\item For each reflection $t \in T$, and each candidate path $\Delta^\prime \in C$, with $x$ equal to the last vertex in $\Delta^\prime$ and $e$ the last edge in $\Delta^\prime$, check whether $x < xt$ and $e \prec t$. If both conditions are satisfied, add $\{\Delta^\prime, xt\}$ to $C$.
\item Iterate the above procedure until the total number of trials equals $\ell(u,v)$.
\item Finally, for each $\Delta^\prime \in C$, if the last vertex of $\Delta^\prime$ equals $v$, then $\Delta^\prime$ is a $\prec$-increasing path in $\Omega_{(W,S)}$ from $u$ to $v$.
\end{enumerate}
While this algorithm is not necessarily efficient, there is a manageable upper bound on the number of candidate paths to be considered: $\ell(u, v) \cdot \vert T \vert$.

Once we have enumerated $B_{\Phi}(w_\lambda^{-1} \bar{w}, w_\lambda^{-1})$, we can read off the statistics $c(\Delta)$ and $\ell(\Delta)$ for each path. The edges of the path determines the root system $J_\Delta$ for by looking at the intersection of all root subsystems in $\Phi$ which contain $E(\Delta)$.

For each path $\Delta \in B_\Phi^\prec (w_\lambda^{-1} \bar{w}, w_\lambda^{-1})$, the group $\chars{S_{w, J_\Delta}}$ is a finitely generated abelian group. It corresponds to $S_{w, J_\Delta}$ under the categorical anti-equivalence between diagonalizable algebraic groups and their character groups. Therefore, if we find the invariant factors for $\chars{S_{w, J}}$, as a finitely generated abelian group, then we get a description of $S_{w, J}$.

We know that $\chars{S_{w, J_\Delta}} = \cochars{\torus} / L_{w, J_\Delta}$, and $L_{w, J_\Delta}$ is generated by $\lambda$ and $\alpha^\vee$ for $\alpha \in J_\Delta^+$ (whenever $\bar{w} \in W_{J_\Delta}$, which is true here). Suppose we choose a generating set for $\cochars{\torus}$ and find the coordinates for $\lambda$ and the coroots $\alpha^\vee$ in terms of these generators. This is sufficient data to compute the invariant factors of $\chars{S_{w, J_{\Delta}}}$. If this group has any torsion, then take the additional step of computing $\vert S_{w, J_{\Delta}}^{\rm{tors}} \cap K_{q-1} \vert$ by comparing the order of the torsion elements with the $(q-1)$-th roots of unity.

This concludes the mathematical considerations for implementing the calculations in software. The calculations reported on in the next two sections and in the Appendix were done through a combination of by-hand calculation and SageMath. Results were obtained on a single 1.1 GHz core using approximately 800 MB of RAM in the largest cases; however, once the up-front work of computing the $\mu$-admissible set is done, the computation could be trivially parallelized across many cores.


\subsection{Results for general linear groups}

Let $F$ be a $p$-adic field. The $F$-points of the general linear group $GL_d$ are
$$
GL_d(F) = \{ g \in M_{d,d} (F) \mid {\rm{det}}(g) \neq 0 \},
$$
where $M_{d, d}(F)$ is the group of $M_{d, d}$ matrices under multiplication with coefficients in $F$. We fix a split maximal torus $T = \{{\rm{diag}}(t_1, \ldots, t_d) \mid t_i \in \mathbb{G}_m\}$ in $GL_d$. Since $GL_d$ is self-dual, we can identify $\torus$ with its dual $\dualtorus$ so that
$$
\dualtorus(\valfield) = \{ \kappa = {\rm{diag}}(\kappa_1, \ldots, \kappa_d) \mid \kappa_i \in \mathbb{C}^\times\}.
$$
The character group $\chars{\dualtorus}$ consists of coordinate projections $\varepsilon_i (\kappa) = \kappa_i$, which can also be viewed as cocharacters of $\torus$. The root system $\Phi = \Phi(G,T)$ is of type $A_{n-1}$, and its positive roots are $\alpha_{ij} = \varepsilon_i - \varepsilon_j$ for $1 \leq i < j \leq n$.

\subsubsection{Example: $GL_4(F)$, $\mu = (1,1,0,0)$}
\label{section::worked-example-gl4-1100}

Let us give a detailed overview of how to compute the coefficients of $\phironeaug$ when $G=GL_4$ and $\mu = (1, 1, 0, 0)$. Although this case is small enough for us to work through the details, we will explain how to read the results out of the tables found in the Appendix to help the reader understand the data in larger examples. Despite being the smallest example different from the Drinfeld case, we will see several interesting distinctions between the coefficients in Corollary~\ref{cor::drinfeld-case} and the coefficients described below.

First, let us consider how data specific to this case fill in some of the values in the formula of Theorem~\ref{main-theorem}. The split maximal torus $T$ has rank equal to 4. Translation elements $t_\lambda$ in the $W$-orbit of $t_\mu$ have length $\ell(t_\lambda) = 4$. Therefore, the formula becomes
$$
\sum_{\Delta \in B_{\Phi_{\rm{aff}}}^{\prec} (w, t_{\lambda(w)})} \delta(s, w,J_\Delta) \vert S_{w,J_\Delta}^{\rm{tors}} \cap K_{q-1}\vert (q-1)^{A(\Delta)} q^{r B(w, \Delta)} (1-q^r)^{C(\Delta)}
$$
where
$$
\begin{cases}
A(\Delta) = 4 - {\rm{rank}}(J_\Delta) - 1 \\
B(w, \Delta) = [(4 - \ell(w)) - \ell(\Delta)]/2 \\
C(\Delta) = \ell(\Delta) - 4
\end{cases}
$$

The root system is type $A_3$, so it has six positive roots, and the longest element of the finite Weyl group is $w_0 = s_{123121}$, where $s_1$, $s_2$ and $s_3$ correspond to the simple positive roots. The notation $s_{i_1 \cdots i_n}$ is shorthand for the product $s_{i_1} \cdots s_{i_n}$. Using Proposition~\ref{prop::refl-order-from-high-root}, we compute the reflection ordering,
$$
s_1 \prec s_{121} \prec s_{12321} \prec s_2 \prec s_{232} \prec s_3,
$$
which for the positive roots is
$$
\alpha_{12} \prec \alpha_{13} \prec \alpha_{14} \prec \alpha_{23} \prec \alpha_{24} \prec \alpha_{34}.
$$
It is also useful to note that the rank of $\Phi$ equals 3, so that we know any $\chi$-root system $J_\Delta$ of rank 3 must equal $\Phi$.

The $\mu$-admissible set ${\rm{Adm}}_{G_r}(\mu)$ contains 33 elements, according to the formula in \cite{haines2000}, Proposition 8.2; the highest length for $w \in {\rm{Adm}}_{G_r}(\mu)$ is $\ell(w) = 4$. We will group the calculations according to these lengths.

\noindent\textbf{Length 0:}

There is a unique $\mu$-admissible element such that $\ell(w) = 0$: $w = t_\mu s_{2312}$. The minimal coset representative $w_\lambda$ equals $\bar{w}^{-1} = s_{2312}$. Calculations show that
$$
B_\Phi^\prec(e, s_{2312}) = \{\Delta_1, \Delta_2\},
$$
where $E(\Delta_1) = \{s_{121},\: s_{232}\}$ and $E(\Delta_2) = \{s_1,\: s_{121},\: s_2,\: s_{232}\}$. The reflection subgroups of $W$ corresponding to these paths are
\begin{itemize}
\item $W_{J_{\Delta_1}} = \{e, s_{121}, s_{232}, s_{2312}\}$
\item $W_{J_{\Delta_2}} = W$
\end{itemize}
Hence $J_{\Delta_1} = \{\pm\alpha_{13}, \pm\alpha_{24}\}$ and $J_{\Delta_2} = \Phi$. So as an intermediate result we have
$$
\phironeaug (I_r^+ sw I_r^+) = \delta(s, w, J_{\Delta_1}) \vert S_{w,J_{\Delta_1}}^{\rm{dz}}\vert q^r (1-q^r)^{-2} + \delta(s, w, J_{\Delta_2}) \vert S_{w,J_{\Delta_2}}^{\rm{dz}} \vert (1-q^r)^{0}.
$$

Elements $\kappa \in S_{w, J_{\Delta_1}}$ satisfy $\alpha_{13}^\vee (\kappa) = \alpha_{24}^\vee (\kappa) = \mu (\kappa) = 1$,
hence they are subject to constraints $\kappa_1 \kappa_2 = 1$, $\kappa_1 = \kappa_3$ and $\kappa_2 = \kappa_4$; meanwhile $\kappa \in S_{w, J_{\Delta_2}}$ satisfy $\kappa_1 \kappa_2 = 1$ and $\kappa_1 = \kappa_2 = \kappa_3 = \kappa_4$ because $\alpha^\vee (\kappa) = 1$ for all $\alpha \in \Phi$. Therefore,
\begin{itemize}
\item $S_{w, J_{\Delta_1}} = \{ {\rm{diag}}(\kappa_1, \kappa_1^{-1}, \kappa_1, \kappa_1^{-1}) \in \dualtorus(\valfield) \mid \kappa_1 \in \mathbb{C}^\times\}$
\item $S_{w, J_{\Delta_2}} = \{ {\rm{diag}}(a, a, a, a) \in \dualtorus(\valfield) \mid a^2 = 1,\ a \in \mathbb{C}^\times \}$
\end{itemize}
The group $S_{w, J_{\Delta_1}}$ is a torus with rank equal to 1, while $S_{w, J_{\Delta_2}}$ is a torsion group of order 2. Assuming ${\rm{char}}(\resfield) \neq 2$, we have $\vert S_{w, J_{\Delta_2}}^{\rm{tors}} \cap K_{q-1} \vert = 2.$ Here's a second intermediate result, assuming ${\rm{char}}(\resfield) \neq 2$:
$$
\phironeaug (I_r^+ sw I_r^+) = \delta(s, w, J_{\Delta_1}) (q-1) q^r (1-q^r)^{-2} + 2 \delta(s, w, J_{\Delta_2}).
$$

Now we describe the finite critical groups $A_{w, J_{\Delta_1}, \resfield}$ and $A_{w, J_{\Delta_2}, \resfield}$, which tell us when $s \in \torus(k_r)$ give $\delta(s, w, J) = 1$. By definition, $\critgrpJ = \langle \nu(\resfield) \mid \nu \in L_{w,J}\rangle$ and $L_{w,J} = \langle w(\nu) - \nu, \alpha^\vee \mid \nu \in \chars{\dualtorus}, \alpha \in J^+\rangle$. Because $J_{\Delta_1} \subset J_{\Delta_2}$, we have the containment $A_{w, J_{\Delta_1}, \resfield} \subset A_{w, J_{\Delta_2}, \resfield}$.

In conclusion, if ${\rm{char}}(\resfield) \neq 2$ and $w = t_\mu s_{2312}$,
$$
\phironeaug (I_r^+ sw I_r^+) =
\begin{cases}
0, & {\rm{if}}\ N_r(s) \notin A_{w, J_{\Delta_2}, \resfield} \\
2, & {\rm{if}}\ N_r(s) \in A_{w, J_{\Delta_2}, \resfield}\setminus A_{w, J_{\Delta_1}, \resfield} \\
(q-1) q^r (1-q^r)^{-2} + 2, & {\rm{if}}\ N_r(s) \in A_{w, J_{\Delta_1}, \resfield}.
\end{cases}
$$

\noindent\textbf{Length 1:}

There are four $\mu$-admissible elements with $\ell(w) = 1$:
\begin{center}
\begin{tabular}{cccc}
$t_{(1,1,0,0)} s_{312}$, & $t_{(1,1,0,0)} s_{231}$, & $t_{(1,1,0,0)} s_{12312}$, & $t_{(1,0,1,0)} s_{23121}$.
\end{tabular}
\end{center}

Let us explain the calculation when $w = t_\mu s_{12312}$. There is a unique $\prec$-increasing path $\Delta$ from $w$ to $t_\mu$; $\ell(\Delta) =3$ and $E(\Delta) = \{s_{121}, s_2, s_{232}\}$. Thus $J_\Delta$ must contain the positive roots $\alpha_{13}$, $\alpha_{23}$, and $\alpha_{24}$, which forces $J_\Delta = \Phi$. Then
$$
S_{w, J_\Delta} = \{{\rm{diag}}(a,a,a,a) \in \dualtorus(\valfield) \mid a^2 = 1\}.
$$
Since there is only a single path $\Delta \in B_{\Phi_{\rm{aff}}}^\prec (w, t_{\lambda(w)})$, there is only a single finite critical group $A_{w, J_\Delta, \resfield}$.

The calculations for the other three $\mu$-admissible elements of length 1 are completely analogous, in the sense that each $B_{\Phi_{\rm{aff}}}^\prec (w, t_{\lambda(w)})$ contains a single path $\Delta$ such that the numbers $A(\Delta)$, $B(w,\Delta)$ and $C(\Delta)$ are as before. Therefore, if $\ell(w) = 1$ and ${\rm{char}}(\resfield) \neq 2$,
$$
\phironeaug (I_r^+ sw I_r^+) =
\begin{cases}
0, & {\rm{if}}\ N_r(s) \notin A_{w, J_{\Delta}, \resfield} \\
2(1-q^r)^{-1}, & {\rm{if}}\ N_r(s) \in A_{w, J_\Delta, \resfield}.
\end{cases}
$$

\noindent\textbf{Length 2:}

There are ten $\mu$-admissible elements with $\ell(w) = 2$; however, as we shall see in a moment, it will help to group them into two subsets.

\begin{center}
\begin{tabular}{|l|l|}
\hline
$X_1$ & $X_2$ \\
\hline
$t_{(0, 1, 1, 0)} s_{32}$ & $t_{(1, 0, 0, 1)} s_{12}$ \\
$t_{(1, 0, 1, 0)} s_{3121}$ & $t_{(1, 0, 1, 0)} s_{1232}$ \\
$t_{(1, 1, 0, 0)} s_{21}$ & $t_{(1, 0, 1, 0)} s_{31}$ \\
$t_{(1, 1, 0, 0)} s_{2321}$ & $t_{(1, 1, 0, 0)} s_{123121}$ \\
  & $t_{(1, 1, 0, 0)} s_{1231}$ \\
  & $t_{(1, 1, 0, 0)} s_{23}$  \\
\hline
\end{tabular}
\end{center}

The elements within each subset determine the same coefficients as the other elements of their subset. We will discuss a representative from each subset.

Suppose we choose $t_{(0, 1, 1, 0)} s_{32} \in X_1$. The set $B_{\Phi_{\rm{aff}}^\prec(w, t_{\lambda(w)})}$ contains a unique path $\Delta$ whose edges are $E(\Delta) = \{s_2, s_3\}$. Then $J_\Delta$ is the rank 2 subsystem whose positive roots are $\alpha_{23}, \alpha_{34}, \alpha_{24}$. Next, we have that
$$
S_{w,J_\Delta} = \{ {\rm{diag}}(\kappa_1, a, a, a) \in \dualtorus(\valfield) \mid \kappa_1 \in \mathbb{C}^\times, a^2 = 1\}.
$$
This is the direct product of a rank 1 torus and the torsion group
$$
S_{w,J_\Delta}^{\rm{tors}} = \{{\rm{diag}}(1, a, a, a) \in \dualtorus(\valfield) \mid a^2 = 1\}.
$$
So we have an example of a diagonalizable $S_{w,J_\Delta}$ that is neither a torus nor a torsion group. Then if $w \in X_1$ and ${\rm{char}}(\resfield) \neq 2$,
$$
\phironeaug(I_r^+ s w I_r^+) =
\begin{cases}
0, & {\rm{if}}\ N_r(s) \notin A_{w, J_\Delta, \resfield}, \\
2(q-1)(1-q^r)^{-2}, & {\rm{if}}\ N_r(s) \in A_{w, J_\Delta, \resfield}.
\end{cases}
$$

Now consider $t_{(1, 0, 0, 1)} s_{12} \in X_2$. There is again a unique $\prec$-increasing path from $w$ to $t_{(1,0,0,1)}$ whose edges are $E(\Delta) = \{s_1, s_{121}\}$. But now
$$
S_{w, J_\Delta} = \{\kappa = {\rm{diag}}(\kappa_1, \kappa_1, \kappa_1, \kappa_1^{-1}) \mid \kappa_1 \in \mathbb{C}^\times\}
$$
is a torus of rank 1, determined by the root system $J_\Delta$ whose positive roots are $J_\Delta^+ = \{\alpha_{12}, \alpha_{23}, \alpha_{13}\}$. We conclude that if $w \in X_2$, then
$$
\phironeaug(I_r^+ s w I_r^+) =
\begin{cases}
0, & {\rm{if}}\ N_r(s) \notin A_{w, J_\Delta, \resfield}, \\
(q-1)(1-q^r)^{-2}, & {\rm{if}}\ N_r(s) \in A_{w, J_\Delta, \resfield}.
\end{cases}
$$

\noindent\textbf{Length 3:}

There are twelve $\mu$-admissible elements with $\ell(w) = 3$:

\begin{center}
\begin{tabular}{llll}
$t_{(1, 1, 0, 0)} s_{12321}$, & $t_{(1, 1, 0, 0)} s_{121}$, & $t_{(1, 1, 0, 0)} s_{232}$, & $t_{(1, 1, 0, 0)} s_2$, \\
$t_{(1, 0, 1, 0)} s_{12321}$, & $t_{(1, 0, 1, 0)} s_3$, & $t_{(1, 0, 1, 0)} s_1$, & $t_{(0, 1, 0, 1)} s_2$ \\
$t_{(1, 0, 0, 1)} s_{121}$, & $t_{(1, 0, 0, 1)} s_1$, & $t_{(0, 1, 1, 0)} s_3$ & $t_{(0, 1, 1, 0)} s_{232}$.
\end{tabular}
\end{center}

Corollary~\ref{cor::main-theorem-codim-1} covers this case. Each $B_{\Phi_{\rm{aff}}}^\prec (w, t_{\lambda(w)})$ contains a unique path of length 1, whose sole edge can be read directly off of each element, e.g. the path $\Delta$ in the case $w = t_{(1, 1, 0, 0)} s_{12321}$ has edge $E(\Delta) = \{s_{12321}\}$.

In order to finish the calculation after invoking Corollary~\ref{cor::main-theorem-codim-1}, we need to check that $S_{w, J_\Delta}$ is torsion-free. We will do so in the case $w = t_{(1, 1, 0, 0)} s_{12321}$, since again, all of the other cases work the same way. Here,
$$
S_{w, J_\Delta} = \{ \kappa = {\rm{diag}}(\kappa_1, \kappa_1^{-1}, \kappa_3, \kappa_1) \mid \kappa_i \in \mathbb{C}^\times\},
$$
which is a rank-2 torus. Then $\vert S_{w, J_\Delta}^{\rm{tors}} \cap K_{q-1} \vert = 1$, and by the Corollary
$$
\phironeaug(I_r^+ s w I_r^+) =
\begin{cases}
0, & {\rm{if}}\ N_r(s) \notin A_{w, J_\Delta, \resfield}, \\
(q-1)^{2} (1-q^r)^{-3}, & {\rm{if}}\ N_r(s) \in A_{w, J_\Delta, \resfield}.
\end{cases}
$$

\noindent\textbf{Length 4:}

There are six $\mu$-admissible elements with $\ell(w) = 4$, namely the six elements $\lambda \in W\mu$. This case is settled by Corollary~\ref{cor::main-theorem-codim-0}:
$$
\phironeaug(I_r^+ s w I_r^+) =
\begin{cases}
0, & {\rm{if}}\ N_r(s) \notin A_{w, \emptyset, \resfield}, \\
(q-1)^{3} (1-q^r)^{-4}, & {\rm{if}}\ N_r(s) \in A_{w, \emptyset, \resfield}.
\end{cases}
$$

Let us summarize the nonzero values from the different cases. We use the notation $A_{w,J_\Delta,\resfield}$ to refer to the finite critical group determined by the unique path in each case, except when $\ell(w) = 0$, where we need to consider two such groups.
$$
\begin{cases}
2, & \ell(w) = 0,\ N_r(s) \in A_{w, J_{\Delta_2}, \resfield}\setminus A_{w, J_{\Delta_1}, \resfield} \\
(q-1) q^r (1-q^r)^{-2} + 2, & \ell(w)=0,\ N_r(s) \in A_{w, J_{\Delta_1}, \resfield}, \\ 
2(1-q^r)^{-1}, & \ell(w) = 1,\ N_r(s) \in A_{w, J_\Delta, \resfield}, \\
2(q-1)(1-q^r)^{-2}, & \ell(w)=2,\ w \in X_1\ N_r(s) \in A_{w, J_\Delta, \resfield}, \\
(q-1)(1-q^r)^{-2}, & \ell(w)=2,\ w \in X_2\ N_r(s) \in A_{w, J_\Delta, \resfield}, \\
(q-1)^{2} (1-q^r)^{-3}, & \ell(w) = 3,\ N_r(s) \in A_{w, J_\Delta, \resfield},\\
(q-1)^{3} (1-q^r)^{-4}, & \ell(w)=4,\ N_r(s) \in A_{w, \emptyset, k_F}.
\end{cases}
$$

\noindent\textbf{Reading data from tables in the appendix:}

The preceding example shows how the values $\phironeaug (I_r^+ sw I_r^+)$ all follow the same basic template. Table~\ref{table::gl4-1100-coeff} contains all of the data needed to compute a coefficient of $\phironeaug$. Here is the row for the length-zero element:

\begin{center}
\begin{tabular}{|c|c|c|c||c|c|c|}
\hline
\multicolumn{7}{|c|}{\cellcolor[gray]{0.9} $\ell(w) = 0$} \\
\hline
Path & $\ell(\Delta)$ & ${\rm{rank}}(J_\Delta)$ & Isom. class of $S_{w,J_\Delta}$ & $A(\Delta)$ &$B(w, \Delta)$ & $C(\Delta)$ \\
\hline
$\Delta_1$ & 2 & 2 & $\mathbb{Z}$ & 1 & 1 & -2 \\
$\Delta_2$ & 4 & 3 & $\mathbb{Z}/2\mathbb{Z}$ & 0 & 0 & 0 \\
\hline
\multicolumn{7}{|c|}{$A_{w, \Delta_1, k_F} \subseteq A_{w, \Delta_2, k_F}$} \\
\hline
\end{tabular}
\end{center}

Of course, we already know the rank of $S_{w,J_\Delta}$ is $A(\Delta)$. The ``Isom. class of $S_{w, J_\Delta}$'' is important because it specifies the torsion subgroup $S_{w,J_\Delta}^{\rm{tors}}$, if present. Thus we have the necessary path data to plug into
$$
\sum_{\Delta \in B_{\Phi_{\rm{aff}}}^{\prec} (w, t_{\lambda(w)})} \delta(s, w,J_\Delta) \vert S_{w,J_\Delta}^{\rm{tors}} \cap K_{q-1}\vert (q-1)^{A(\Delta)} q^{r B(w, \Delta)} (1-q^r)^{C(\Delta)},
$$
the template formula for this case, while the containment data $A_{w, \Delta_1, k_F} \subseteq A_{w, \Delta_2, k_F}$ tells us how to arrange the path values into the expression
$$
\phironeaug (I_r^+ sw I_r^+) =
\begin{cases}
0, & {\rm{if}}\ N_r(s) \notin A_{w, J_{\Delta_2}, \resfield} \\
2, & {\rm{if}}\ N_r(s) \in A_{w, J_{\Delta_2}, \resfield}\setminus A_{w, J_{\Delta_1}, \resfield} \\
(q-1) q^r (1-q^r)^{-2} + 2, & {\rm{if}}\ N_r(s) \in A_{w, J_{\Delta_1}, \resfield}.
\end{cases}
$$
If $B_{\Phi_{\rm{aff}}}^{\prec} (w, t_{\lambda(w)})$ contains a single path, the table will omit the information about $A_{w, J_\Delta, \resfield}$.

We saw that when $\ell(w) = 2$, there were two possibilities for $\phironeaug (I_r^+ sw I_r^+)$ depending on whether $w \in X_1$ or $w \in X_2$. Table~\ref{table::gl4-1100-adm} addresses this issue: Suppose we were given $w \in {\rm{Adm}}_{G_r}(\mu)$ with $\ell(w) = 2$; we would look up which subset $X_i$ contains $w$ in Table~\ref{table::gl4-1100-adm}, which would tell us which row in Table~\ref{table::gl4-1100-coeff} to use.

\subsubsection{$GL_5(F)$, $\mu= (1, 1, 0, 0, 0)$}

The group $GL_5$ is type $A_4$, hence its root system $\Phi$ has rank equal to 4. The maximal split torus $T \subset GL_5$ has rank $d = 5$. For all conjugates $\lambda \in W\cdot\mu$, $\ell(t_\lambda) = 6$. Therefore, the formula becomes
$$
(-1) \sum_{\Delta \in B_\Phi^{\prec} (w,\wtrans)} \delta(s, w,J_\Delta) \vert S_{w,J_\Delta}^{\rm{tors}} \cap K_{q-1} \vert (q-1)^{A(w, \Delta)} q^{r B(w, \Delta)} (1-q^r)^{C(\Delta)}
$$
where
$$
\begin{cases}
A(\Delta) = 5 - {\rm{rank}}(J_\Delta) - 1 \\
B(w, \Delta) = [(6 - \ell(w)) - \ell(\Delta)]/2 \\
C(\Delta) = \ell(\Delta) - 5
\end{cases}
$$
Fix the following reflection ordering on $W$:
$$
\alpha_{12} \prec \alpha_{13} \prec \alpha_{14} \prec \alpha_{15} \prec \alpha_{23} \prec \alpha_{24} \prec \alpha_{25} \prec \alpha_{34} \prec \alpha_{35} \prec \alpha_{45}.
$$
There are 131 elements in ${\rm{Adm}}_{G_r}(\mu)$. 

The raw data for this case can be found in the Appendix, Tables~\ref{table::gl5-11000-coeff} and \ref{table::gl5-11000-adm}. The story for this case is similar to what we saw for $(GL_4, (1,1,0,0))$; however, we will point out a few features before moving on to the $GL_6$ cases.

First, let us discuss the length-zero alcove: $w = t_{(1,1,0,0,0)} s_{234123}$. There are three paths in $B_{\Phi_{\rm{aff}}}^\prec(w, t_{\lambda(w)})$ whose edge sets are:
\begin{itemize}
\item $E(\Delta_1) = \{ s_1, s_{12321}, s_2, s_{23432}\}$
\item $E(\Delta_2) = \{ s_{121}, s_{12321}, s_{232}, s_{23432}\}$
\item $E(\Delta_3) = \{ s_1, s_{121}, s_{12321}, s_2, s_{232}, s_{23432}\}$
\end{itemize}
All three paths have $J_\Delta = \Phi$, hence they determine the same finite critical group $A_{w, J_\Delta, k_F}$. Compare this to the length-zero element in the case $(GL_4, \mu=(1,1,0,0))$, where there were two paths in $B_{\Phi_{\rm{aff}}}^\prec(w, t_{\lambda(w)})$ but one finite critical group was a subgroup of the other.

The relevant subgroup is $S_{w,J_\Delta} = \{{\rm{diag}}(a, a, a, a, a) \in \dualtorus(\valfield) \mid a^2 = 1\}$, i.e., a torsion group of order 2. Therefore, if $\ell(w) = 0$ is $\mu$-admissible and ${\rm{char}}(k_F) \neq 2$,
$$
\phironeaug(I_r^+ sw I_r^+) = 
\begin{cases}
0, & {\rm{if}}\ N_r(s) \notin A_{w,J_\Delta, k_F}, \\
(-1)\Big(4q^r(1-q^r)^{-1} + 2(1-q^r)\Big), & {\rm{if}}\ N_r(s) \in A_{w,J_\Delta, k_F}.
\end{cases}
$$

Second, notice that the coefficient for certain $\mu$-admissible length-two $w$ (specifically, those $w \in X_2$ listed in Table~\ref{table::gl5-11000-adm}) is very similar to that of the length-zero element in $(GL_4, \mu=(1,1,0,0))$. Here is a representative element:
$$
w = t_{(1,1,0,0,0)} s_{2312}.
$$
The set $B_{\Phi_{\rm{aff}}}^\prec(w, t_{\lambda(w)})$ for this element contains two paths, whose edge sets are:
\begin{itemize}
\item $E(\Delta_1) = \{s_{121}, s_{232}\}$
\item $E(\Delta_2) = \{s_1, s_{121}, s_2, s_{232}\}$.
\end{itemize}
These are identical to the paths given in Section~\ref{section::worked-example-gl4-1100} for the length-zero alcove. Of course, here we have a larger torus; so for example
$$
S_{w, J_{\Delta_2}} = \{{\rm{diag}}(a, a, a, a, \kappa_5) \in \dualtorus(\valfield) \mid a^2 = 1, \kappa_5 \in \valfield^\times\}
$$
has a free variable $\kappa_5$ where the corresponding relevant group in the $GL_4$ case was only a torsion group of order 2. If ${\rm{char}}(\resfield) \neq 2$,
$$
\phironeaug(I_r^+ sw I_r^+) = 
\begin{cases}
0,\ {\rm{if}}\ N_r(s) \notin A_{w, J_{\Delta_2}, \resfield}, \\
(-1)\Big(2(q-1)(1-q^r)^{-1}\Big),\ {\rm{if}}\ N_r(s) \in A_{w, J_{\Delta_2}, \resfield}\setminus A_{w, J_{\Delta_1}, \resfield}, \\
(-1)\Big((q-1)^2q^r(1-q^r)^{-3} + 2(q-1)(1-q^r)^{-1}\Big),\ {\rm{else}}.
\end{cases}
$$

On the other hand, if our $\mu$-admissible element $w$ has $\ell(w) = 1$, then the data in Table~\ref{table::gl5-11000-coeff} shows that the coefficient $\phironeaug(I_r^+ sw I_r^+)$ is essentially identical to that of the length-zero element for $GL_4$, even though the former's $\prec$-increasing paths are longer:
$$
\phironeaug (I_r^+ sw I_r^+) =
\begin{cases}
0, & {\rm{if}}\ N_r(s) \notin A_{w, J_{\Delta_2}, \resfield} \\
-2, & {\rm{if}}\ N_r(s) \in A_{w, J_{\Delta_2}, \resfield}\setminus A_{w, J_{\Delta_1}, \resfield} \\
(-1)\Big((q-1) q^r (1-q^r)^{-2} + 2\Big), & {\rm{if}}\ N_r(s) \in A_{w, J_{\Delta_1}, \resfield}.
\end{cases}
$$

Of course, we have called attention to these examples to make the point that patterns appear throughout the data. The polynomial given by the formula ultimately depends on the structure of Bruhat intervals in some finite Weyl group, and there are many isomorphic intervals between pairs of elements in different groups.

The examples given above cover all cases in $(GL_5, \mu = (1,1,0,0,0))$ where $B_{\Phi_{\rm{aff}}}^\prec(w, t_{\lambda(w)})$ contains multiple paths.

\subsubsection{$GL_6(F)$, $\mu= (1, 1, 0, 0, 0, 0)$}
\label{section::gl6-mu-110000}

The group $GL_6$ is type $A_5$, so $\Phi$ has rank 5. The split maximal torus $T$ has rank $d=6$. All $W$-conjugates $t_\lambda$ of $t_\mu$ have $\ell(t_\lambda) = 8$. The $\mu$-admissible set contains 473 elements. Following the previous examples, this data determines our template
$$
\sum_{\Delta \in B_\Phi^{\prec} (w,\wtrans)} \delta(s, w,J_\Delta) \vert S_{w,J_\Delta}^{\rm{tors}} \cap K_{q-1} \vert (q-1)^{A(w, \Delta)} q^{r B(w, \Delta)} (1-q^r)^{C(\Delta)}
$$
where
$$
\begin{cases}
A(\Delta) = 6 - {\rm{rank}}(J_\Delta) - 1 \\
B(w, \Delta) = [(8 - \ell(w)) - \ell(\Delta)]/2 \\
C(\Delta) = \ell(\Delta) - 6
\end{cases}
$$

Fix the following reflection ordering on $W$:
\begin{multline*}
\alpha_{12} \prec \alpha_{13} \prec \alpha_{14} \prec \alpha_{15} \prec \alpha_{16} \prec \alpha_{23} \prec \alpha_{24} \prec \alpha_{25} \prec \alpha_{26} \prec \\ 
\prec \alpha_{34} \prec \alpha_{35} \prec \alpha_{36} \prec \prec \alpha_{45} \prec \alpha_{46} \prec \alpha_{56}.
\end{multline*}

The raw data for calculating the coefficients of $\phironeaug$ in this case are spread across Tables~\ref{table::gl6-110000-coeff} and \ref{table::gl6-110000-coeff-cont}. There are sub-cases in the calculations for $\mu$-admissible $w$ of lengths 2, 3, 4, 5 and 6; Tables~\ref{table::gl6-110000-adm} and \ref{table::gl6-110000-adm-cont} list which elements correspond to a given sub-case.

The unique length-zero $\mu$-admissible element is $w = t_{(1, 1, 0, 0, 0, 0)} s_{23451234}$. There are five $\prec$-increasing paths from $w$ to $t_{\lambda(w)}$; their edge sets are:
\begin{itemize}
\item $E(\Delta_1) = \{s_{\alpha_{13}}, s_{\alpha_{15}}, s_{\alpha_{24}}, s_{\alpha_{26}}\}$
\item $E(\Delta_2) = \{s_{\alpha_{12}}, s_{\alpha_{13}}, s_{\alpha_{15}}, s_{\alpha_{23}}, s_{\alpha_{24}}, s_{\alpha_{26}}\}$
\item $E(\Delta_3) = \{s_{\alpha_{12}}, s_{\alpha_{14}}, s_{\alpha_{15}}, s_{\alpha_{23}}, s_{\alpha_{25}}, s_{\alpha_{26}}\}$
\item $E(\Delta_4) = \{s_{\alpha_{13}}, s_{\alpha_{14}}, s_{\alpha_{15}}, s_{\alpha_{24}}, s_{\alpha_{25}}, s_{\alpha_{26}}\}$
\item $E(\Delta_5) = \{s_{\alpha_{12}}, s_{\alpha_{13}}, s_{\alpha_{14}}, s_{\alpha_{15}} , s_{\alpha_{23}}, s_{\alpha_{24}}, s_{\alpha_{25}}, s_{\alpha_{26}}\}$
\end{itemize}
For paths $\Delta_2, \ldots \Delta_5$, $J_{\Delta_i} = \Phi$. The root system $J_{\Delta_4}$ is a subsystem of rank 4. Using the data from Table~\ref{table::gl6-110000-coeff}, we have for ${\rm{char}}(\resfield) \neq 2$,
$$
\phironeaug(I_r^+ sw I_r^+) = 
\begin{cases}
0,\ {\rm{if}}\ N_r(s) \notin A_{w, J_{\Delta_2}, \resfield}, \\
6q^r + (1 - q^r)^{2},\ {\rm{if}}\ N_r(s) \in A_{w, J_{\Delta_2}, \resfield}\setminus A_{w, J_{\Delta_1}, \resfield}, \\
(q-1)q^{2r}(1-q^r)^{-2} + 6q^r + (1 - q^r)^{2},\ {\rm{if}}\ N_r(s) \in A_{w, J_{\Delta_1}, \resfield}.
\end{cases}
$$

The data shows that when $w \in {\rm{Adm}}_{GL_6}(\mu=(1,1,0,0,0,0))$ has $\ell(w) = 1$, the coefficient $\phironeaug (I_r^+ sw I_r^+)$ is the same as the coefficient for the unique length-zero $w \in {\rm{Adm}}_{GL_5}(\mu=(1,1,0,0,0))$---other than their difference in sign.

\subsubsection{$GL_6(F)$, $\mu= (1, 1, 1, 0, 0, 0)$}

Most of the objects here are the same as in Section~\ref{section::gl6-mu-110000}, such as $\Phi$, the choice of maximal torus $\torus$, and the reflection ordering $\prec$ on $W$. In this case, the translation $t_\mu$ and its conjugates have length $\ell(t_\mu) = 9$. The template is
$$
\sum_{\Delta \in B_\Phi^{\prec} (w,\wtrans)} \delta(s, w,J_\Delta) \vert S_{w,J_\Delta}^{\rm{tors}} \cap K_{q-1} \vert (q-1)^{A(w, \Delta)} q^{r B(w, \Delta)} (1-q^r)^{C(\Delta)}
$$
where
$$
\begin{cases}
A(\Delta) = 6 - {\rm{rank}}(J_\Delta) - 1 \\
B(w, \Delta) = [(9 - \ell(w)) - \ell(\Delta)]/2 \\
C(\Delta) = \ell(\Delta) - 6.
\end{cases}
$$
There are 883 $\mu$-admissible elements, and as one would expect, there are many different cases for the coefficients of $\phironeaug$. Raw data for the coefficients is spread across Tables~\ref{table::gl6-111000-coeff}, \ref{table::gl6-111000-coeff-cont} and \ref{table::gl6-111000-coeff-cont2}. The subsets of $\mu$-admissible elements are further explained in Tables~\ref{table::gl6-111000-adm} and \ref{table::gl6-111000-adm-cont} as appropriate.

Let us discuss the coefficient for $w = t_{(1, 1, 1, 0, 0, 0)} s_{345234123}$, the unique length-zero $\mu$-admissible element. This is the most complex example we shall consider in this chapter: $B_{\Phi_{\rm{aff}}}^\prec (w, t_{\lambda(w)})$ contains nine paths, which yield five distinct root systems $J_\Delta$. The edges sets are:
\begin{itemize}
\item $E_{\Delta_1} = \{s_{\alpha_{14}}, s_{\alpha_{25}}, s_{\alpha_{36}}  \}$
\item $E_{\Delta_2} = \{s_{\alpha_{13}}, s_{\alpha_{14}}, s_{\alpha_{25}}, s_{\alpha_{34}}, s_{\alpha_{36}}  \}$
\item $E_{\Delta_3} = \{s_{\alpha_{14}}, s_{\alpha_{23}}, s_{\alpha_{25}}, s_{\alpha_{35}}, s_{\alpha_{36}}  \}$
\item $E_{\Delta_4} = \{s_{\alpha_{12}}, s_{\alpha_{14}}, s_{\alpha_{24}}, s_{\alpha_{25}}, s_{\alpha_{36}}  \}$
\item $E_{\Delta_5} = \{s_{\alpha_{12}}, s_{\alpha_{13}}, s_{\alpha_{14}}, s_{\alpha_{23}}, s_{\alpha_{25}}, s_{\alpha_{34}}, s_{\alpha_{36}}  \}$
\item $E_{\Delta_6} = \{s_{\alpha_{12}}, s_{\alpha_{14}}, s_{\alpha_{23}}, s_{\alpha_{25}}, s_{\alpha_{34}}, s_{\alpha_{35}}, s_{\alpha_{36}}  \}$
\item $E_{\Delta_7} = \{s_{\alpha_{13}}, s_{\alpha_{14}}, s_{\alpha_{24}}, s_{\alpha_{25}}, s_{\alpha_{34}}, s_{\alpha_{35}}, s_{\alpha_{36}}  \}$
\item $E_{\Delta_8} = \{s_{\alpha_{12}}, s_{\alpha_{13}}, s_{\alpha_{14}}, s_{\alpha_{24}}, s_{\alpha_{25}}, s_{\alpha_{35}}, s_{\alpha_{36}}  \}$
\item $E_{\Delta_9} = \{s_{\alpha_{12}}, s_{\alpha_{13}}, s_{\alpha_{14}}, s_{\alpha_{23}}, s_{\alpha_{24}}, s_{\alpha_{25}}, s_{\alpha_{34}}, s_{\alpha_{35}}, s_{\alpha_{36}}  \}$
\end{itemize}

The paths $\Delta_5, \ldots, \Delta_9$ all have $J_\Delta = \Phi$; however, the other four paths all give a distinct subsystem $J_\Delta \subset \Phi$.
\begin{itemize}
\item $J_{\Delta_1} = \{\pm \alpha_{14}\} \times \{\pm \alpha_{25}\} \times \{ \pm \alpha_{36}\}$
\item $J_{\Delta_2} = \{\pm\alpha_{13}, \pm\alpha_{14} \pm\alpha_{16}, \pm\alpha_{34}, \pm\alpha_{36}, \pm\alpha_{46} \} \times \{\pm \alpha_{25}\}$
\item $J_{\Delta_3} = \{\pm\alpha_{14}\} \times \{\pm\alpha_{23}, \pm\alpha_{25}, \pm\alpha_{26}, \pm\alpha_{35}, \pm\alpha_{36}, \pm\alpha_{56}\}$
\item $J_{\Delta_4} = \{\pm\alpha_{12}, \pm\alpha_{14}, \pm\alpha_{15}, \pm\alpha_{24}, \pm\alpha_{25}, \pm\alpha_{45}\} \times \{ \pm \alpha_{36}\}$
\end{itemize}

Here is the graph of inclusions for the finite critical groups:
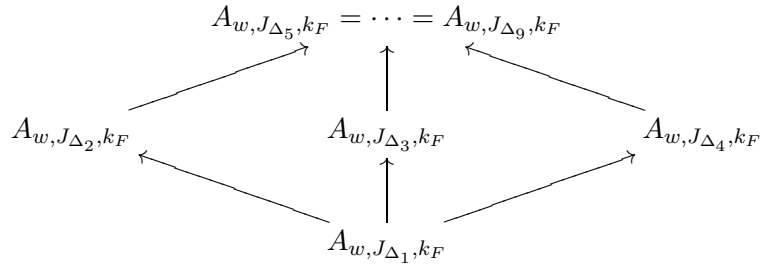
\[
\xymatrix{
 & A_{w,J_{\Delta_5}, k_F} = \cdots = A_{w, J_{\Delta_9}, k_F} & \\
A_{w, J_{\Delta_2}, k_F} \ar[ur] & A_{w, J_{\Delta_3}, k_F} \ar[u] & A_{w, J_{\Delta_4}, k_F} \ar[ul] \\
 & A_{w, J_{\Delta_1}, k_F} \ar[ul] \ar[u] \ar[ur]
}
\]

As usual, the relationships between finite critical groups characterize the different cases for the coefficient $\phironeaug (I_r^+ sw I_r^+)$ when $w = t_{(1, 1, 1, 0, 0, 0)} s_{345234123}$. If $N_r(s) \notin A_{w, J_{\Delta_5}, \resfield}$, then $\phironeaug (I_r^+ sw I_r^+) = 0$. Otherwise, there are three nonzero possibilities:

\noindent\textbf{Case 1:} If $N_r(s) \in A_{w, J_{\Delta_5}, \resfield} \setminus \Big(\bigcup_{i=2}^4 A_{w, J_{\Delta_i}, \resfield}\Big)$, the coefficient is
$$
12q^r(1-q^r) + 3(1-q^r)^3.
$$
\textbf{Case 2:} If $N_r (s) \in A_{w, J_{\Delta_i},\resfield} \setminus A_{w, J_{\Delta_1},\resfield}$ for $i = 2,3,4$, the coefficient is
$$
(q-1)q^{2r}(1-q^r)^{-1} + 12q^r(1-q^r) + 3(1-q^r)^3.
$$
\textbf{Case 3:} If $N_r(s) \in A_{w, J_{\Delta_1},\resfield}$, the coefficient is
$$
(q-1)^2q^{3r}(1-q^r)^{-3} + 3(q-1)q^{2r}(1-q^r)^{-1} + 12q^r(1-q^r) + 3(1-q^r)^3.
$$


\subsection{Results for general symplectic groups}

We begin by recalling the definition of $GSp_{2n}$ and the associated data required for the calculations of $\phironeaug$. Let $\tilde{I}_n$ be the $n\times n$ matrix with ones on the anti-diagonal and zeroes everywhere else. Let
$$
J = \begin{pmatrix} 0 & \tilde{I}_n \\ -\tilde{I}_n & 0 \end{pmatrix}.
$$
Then for a $p$-adic field $F$, the $F$-rational points of $GSp_{2n}$ are
$$
GSp_{2n}(F) = \{ g \in GL_{2n}(F) \mid {^t}g J g = c(g) J, c(g) \in F^\times\}
$$
We choose a split maximal torus $T$ with
$$
T(F) = \{{\rm{diag}}(t_1, \ldots, t_n, t_n^{-1}t_0,\ldots t_1^{-1}t_0) \mid t_i \in F^\times\}.
$$

For the $GL_n$ case, we showed how to compute $\vert S_{w, J}^{\rm{dz}}\vert$ from the definitions, looking at endoscopic elements in $\dualtorus(\valfield)$. For examples when $G=GSp_{2n}$, we show how to use Lemma~\ref{lemma::rel-grp-gln-gsp2n} to get the structure of $\chars{S_{w,J}}$, and hence that of $S_{w,J}$.

\subsubsection{Example: $GSp_4(F)$, $\mu = (1,1,0,0)$}

As with the $GL_n$ examples, we use data from this case to first write down a template formula for the coefficients of $\phironeaug$. The rank of $T$ is 3, while translation elements $t_\lambda$ for $\lambda \in W\mu$ have $\ell(t_\lambda) = 3$. Thus for $w \in \widetilde{W}$ and $s \in \torus(k_r)$, the formula for $\phironeaug(I_r^+ sw I_r^+)$ becomes
$$
(-1) \sum_{\Delta \in B_\Phi^{\prec} (w,\wtrans)} \delta(s, w,J_\Delta) \vert S_{w,J_\Delta}^{\rm{tors}} \cap K_{q-1} \vert (q-1)^{A(w, \Delta)} q^{r B(w, \Delta)} (1-q^r)^{C(\Delta)}
$$
where
$$
\begin{cases}
A(\Delta) = 3 - {\rm{rank}}(J_\Delta) - 1 \\
B(w, \Delta) = [(3 - \ell(w)) - \ell(\Delta)]/2 \\
C(\Delta) = \ell(\Delta) - 3.
\end{cases}
$$

Let us describe the characters and cocharacters of $T$. Let $\varepsilon_i$ be the $i$-th coordinate projection of $T$, and let $c$ be the similitude character. Then
$$
\chars{\torus} = \langle c, \varepsilon_1, \varepsilon_2, \varepsilon_3, \varepsilon_4\rangle / \langle \varepsilon_1 + \varepsilon_4 = \varepsilon_2 + \varepsilon_3 = c\rangle.
$$
Using coordinates in terms of the $\varepsilon_i$, the positive roots in $\Phi$, which is type $C_2$, are
\begin{align*}
\alpha_1 &= (1/2, -1/2, 1/2, -1/2) \\
\alpha_2 &= (0, 1, -1, 0) \\
\alpha_3 &= (1/2, 1/2, -1/2, -1/2) \\
\alpha_4 &= (1, 0, 0 -1).
\end{align*}
We can also describe the character lattice as $\chars{\torus} = \langle c_0, c_1, c_2 \rangle$, where the generators act on $t = {\rm{diag}}(t_1, t_2, t_2^{-1}t_0, t_1^{-1}t_0)$ by $c_i (t) = t_i$ for $i = 0,1,2$. In this coordinate system, the roots are written
\begin{align*}
\alpha_1 &= c_1 - c_2  \\
\alpha_2 &= 2 c_2 - c_0 \\
\alpha_3 &= c_1 + c_2 - c_0\\
\alpha_4 &= 2 c_1 - c_0.
\end{align*}
The cocharacter lattice $\cochars{T}$ is the free group on generators $e_0, e_1, e_2$ where for $x \in F2^\times$,
\begin{align*}
e_0 (x) &= {\rm{diag}}(1, 1, x, x) \\
e_1 (x) &= {\rm{diag}}(x, 1, 1, x^{-1}) \\
e_2 (x) &= {\rm{diag}}(1, x, x^{-1}, 1).
\end{align*}
We could again use coordinates in terms of maps $\breve{\varepsilon}_i$ sending $x \in F^\times$ to the $i$-th coordinate in $T$ with 1's elsewhere. Then an element $\nu \in \cochars{\torus}$ is a tuple $(a_1, a_2, a_3, a_4)$ such that $a_1 + a_4 = a_2 + a_3 = c$. In these coordinates, $\mu = (1,1,0,0)$ is the map $\mu(x) = {\rm{diag}}(x, x, 1, 1) \in T(F)$ for $x \in F^\times$. So $\mu = e_0 + e_1 + e_2$. The coroots are
\begin{align*}
\alpha_1^\vee &= (1, -1, 1, -1) = e_1 - e_2 \\
\alpha_2^\vee &= (0, 1, -1, 0) = e_2 \\
\alpha_3^\vee &= (1,1,-1,-1) = e_1 + e_2 \\
\alpha_4^\vee &= (1, 0, 0, -1) = e_1 \\
\end{align*}
Let $s_1$ and $s_2$ denote the simple reflections corresponding to the roots $\alpha_1$ and $\alpha_2$, respectively. The longest element of the finite Weyl group is $w_0 = s_{2121}$, where again the notation $s_{i_1 \cdots i_n}$ is shorthand for the product $s_{i_1} \cdots s_{i_n}$. By Proposition~\ref{prop::refl-order-from-high-root}, we have a reflection ordering,
$$
s_2 \prec s_{212} \prec s_{121} \prec s_1.
$$

When $\mu = (1,1,0,0)$, Proposition 8.2 of~\cite{haines2000} shows that the set ${\rm{Adm}}_{G_r}(\mu)$ contains 13 elements. The $\mu$-admissible elements range in length from 0 to 3.

See Table~\ref{table::gsp4-1100-coeff} for the raw coefficient data.

\noindent\textbf{Length 0:}

The unique length-zero $\mu$-admissible element is $w = t_{(1,1,0,0)} s_{212}$. We want to enumerate the $\prec$-increasing paths
$$
B_{\Phi}^{\prec} (w_\lambda^{-1} \bar{w}, w_\lambda^{-1}) = B_{\Phi}^{\prec} (1, s_{212}).
$$
It turns out there are only two paths, whose edge sets are
\begin{align*}
E(\Delta_1) &= \{s_{212}\}, \\
E(\Delta_2) &= \{s_2, s_{212}, s_{121}\}.
\end{align*}
It is clear that $J_{\Delta_1} = \{ \pm \alpha_3 \}$ and $J_{\Delta_2} = \Phi$. Plugging this data into the template gives us
$$
(-1) \Big( \delta(s, w,J_{\Delta_1}) \vert S_{w,J_{\Delta_1}}^{\rm{tors}} \cap K_{q-1} \vert (q-1) q^{r} (1-q^r)^{-2} + \delta(s, w,J_{\Delta_2}) \vert S_{w,J_{\Delta_2}}^{\rm{tors}} \cap K_{q-1} \vert \Big)
$$

It remains to describe $S_{w, J_{\Delta_i}}$ for $i = 1,2$. Consider the quotient
$$
\cochars{\torus} / \mathbb{Z}J_{\Delta_1}^\vee = \langle e_0, e_1, e_2 \rangle / \langle e_1 + e_2 = 0\rangle \cong \langle \bar{e}_0, \bar{e}_1 \rangle.
$$
Now use the method of Lemma~\ref{lemma::rel-grp-gln-gsp2n}, i.e., consider how $\lambda(w) = \mu = e_0 + e_1 + e_2$ appears in the above quotient. It is just $\bar{\mu} = \bar{e}_0 + \bar{e}_1 - \bar{e_1} = \bar{e_0}$. But then
$$
\chars{S_{w, J_{\Delta_1}}} \cong \langle \bar{e}_0, \bar{e}_1 \rangle / \langle \bar{e}_0 \rangle \cong \langle \bar{e}_1 \rangle \cong \mathbb{Z}.
$$
Similarly, for $J_{\Delta_2} = \Phi$,
$$
\cochars{\torus}/\mathbb{Z}\Phi^\vee = \langle \bar{e}_0 \rangle,
$$
because the relations imposed by the coroots in $\Phi^\vee$ force $\bar{e}_1 = \bar{e}_2 = 0$. Then in the quotient,
$$
\overline{\lambda(w)} = \bar{\mu} = \bar{e}_0.
$$
So $\chars{S_{w, J_{\Delta_2}}} = \{1\}$.

The finite critical groups satisfy $A_{w, J_{\Delta_1}, k_F} \subset A_{w, J_{\Delta_2}, k_F}$. This is enough information to give the coefficient data for this alcove:
$$
\phironeaug (I_r^+ sw I_r^+) = \begin{cases}
0, & {\rm{if}}\ N_r (s) \notin A_{w, J_{\Delta_2}, k_F} \\
-1, & {\rm{if}}\ N_r (s) \in A_{w, J_{\Delta_2}, k_F} \setminus A_{w, J_{\Delta_1}, k_F} \\
(-1) \Big( 1 + (q-1) q^{r} (1-q^r)^{-2} \Big), & {\rm{if}}\ N_r (s) \in A_{w, J_{\Delta_1}, k_F}.
\end{cases}
$$

\noindent\textbf{Length 1:}

There are three $\mu$-admissible elements $w$ such that $\ell(w) = 3$:
$$
\begin{tabular}{ccc}
$t_{(1,1,0,0)} s_{2121}$ & $t_{(1,1,0,0)} s_{21}$ & $t_{(1,0,1,0)} s_{12}$.
\end{tabular}
$$
In each case, $B_{\Phi}^\prec (w_\lambda^{-1} \bar{w}, w_\lambda^{-1})$ contains a single path $\Delta_1$ with $\ell(\Delta_1) = 2$.

When $w = t_{(1,1,0,0)} s_{2121}$, we have $E(\Delta_1) = \{s_{2}, s_{121}\}$ and $J_{\Delta_1} = \{\pm 2\alpha_2, \pm 2\alpha_4\}$. In this case,
$$
\cochars{\torus}/\mathbb{Z}J_{\Delta_1}^\vee = \langle e_0, e_1, e_2 \rangle / \langle e_1, e_2 \rangle.
$$
Then $\overline{\lambda(w)} = \bar{\mu} = \bar{e}_0$ means
$$
\chars{S_{w,J_{\Delta_1}}} = \langle \bar{e}_0 \rangle / \langle \bar{e}_0 \rangle = \{1\}.
$$
For $w = t_{(1,1,0,0)} s_{21}$, the edge set is $E(\Delta_1) = \{s_{212}, s_{121}\}$ and $J_{\Delta_1} = \Phi$. Finally, if $w = t_{(1,0,1,0)} s_{12}$, then $E(\Delta_1) = \{s_2, s_1\}$ and $J_{\Delta_1} = \Phi$. In both of these cases, $\chars{S_{w, J_{\Delta_1}}}$ is again trivial.

We have shown that
$$
\phironeaug (I_r^+ sw I_r^+) = \begin{cases}
0, & {\rm{if}}\ N_r (s) \notin A_{w, J_{\Delta_1}, k_F} \\
(-1)(1-q^r)^{-1}, & {\rm{if}}\ N_r (s) \in A_{w, J_{\Delta_1}, k_F}.
\end{cases}
$$

\noindent\textbf{Length 2:}

There are five $\mu$-admissible elements whose length is two. In each case, there is a single $\prec$-increasing path $w \stackrel{\Delta_1}{\longrightarrow} t_{\lambda(w)}$. Here's the data:

\begin{center}
\begin{tabular}{|c|c|c|}
\hline
$w$ & $E(\Delta_1)$ & $J_{\Delta_1}$ \\
\hline
$t_{(1,1,0,0)} s_2$ & $\{s_2\}$ & $\{\pm \alpha_2\}$ \\
$t_{(0,1,0,1)} s_2$ & $\{s_2\}$ & $\{\pm \alpha_2\}$ \\
$t_{(1,0,1,0)} s_1$ & $\{s_1\}$ & $\{\pm \alpha_1\}$ \\
$t_{(1,1,0,0)} s_{121}$ & $\{s_{121}\}$ & $\{\pm \alpha_4\}$ \\
$t_{(1,0,1,0)} s_{121}$ & $\{s_{121}\}$ & $\{\pm \alpha_4\}$ \\
\hline
\end{tabular}
\end{center}

Let us work out $\chars{S_{w, J_{\Delta_1}}}$ for $w = t_{(1,1,0,0)} s_2$. Starting with
$$
\cochars{\torus} / \mathbb{Z}J_{\Delta_1}^\vee = \langle e_0, e_1, e_2 \rangle / \langle e_2 \rangle = \langle \bar{e}_0, \bar{e}_1 \rangle,
$$
we have that $\overline{\lambda(w)} = \bar{\mu} = \bar{e}_0 + \bar{e}_1$. Then in the notation of Lemma~\ref{lemma::rel-grp-gln-gsp2n}, the map $c: \mathbb{Z} \rightarrow \mathbb{Z}^2$ is given by the tuple $c = (1, 1, 0)$. Since $c_1$ is a unit, $\chars{S_{w, J_{\Delta_1}}} \cong \mathbb{Z}$. The calculation of $\chars{S_{w, J_{\Delta_1}}}$ is similar for the other length-two $\mu$-admissible elements.

The coefficient for this case is
$$
\phironeaug (I_r^+ sw I_r^+) = \begin{cases}
0, & {\rm{if}}\ N_r (s) \notin A_{w, J_{\Delta_1}, k_F} \\
(-1)(q-1)(1-q^r)^{-2}, & {\rm{if}} N_r (s) \in A_{w, J_{\Delta_1}, k_F}.
\end{cases}
$$

\noindent\textbf{Length 3:}

The translation elements are $t_{(1,1,0,0)}$, $t_{(1,0,1,0)}$, $t_{(0,1,0,1)}$ and $t_{(0,0,1,1)}$. For any of these, Corollary~\ref{cor::main-theorem-codim-0} shows
$$
\phironeaug (I_r^+ s t_\lambda I_r^+) = \begin{cases}
0, & {\rm{if}}\ N_r(s) \notin A_{t_\lambda, \emptyset, k_F}, \\
(-1)(q-1)^{2}(1-q^r)^{-3}, & {\rm{if}} N_r(s) \in A_{t_\lambda, \emptyset, k_F}.
\end{cases}
$$

\subsubsection{$GSp_6(F)$, $\mu = (1,1,1,0,0,0)$}

In this case, the root system has type $C_3$, so ${\rm{rank}}(T) = 4$. For $\lambda \in W\mu$, with $\mu = (1,1,1,0,0,0)$, translation elements have $\ell(t_\lambda) = 6$. Use this data to fill in the template:
$$
\sum_{\Delta \in B_\Phi^{\prec} (w,\wtrans)} \delta(s, w,J_\Delta) \vert S_{w,J_\Delta}^{\rm{tors}} \cap K_{q-1} \vert (q-1)^{A(w, \Delta)} q^{r B(w, \Delta)} (1-q^r)^{C(\Delta)}
$$
where
$$
\begin{cases}
A(\Delta) = 4 - {\rm{rank}}(J_\Delta) - 1 \\
B(w, \Delta) = [(6 - \ell(w)) - \ell(\Delta)]/2 \\
C(\Delta) = \ell(\Delta) - 4.
\end{cases}
$$

Roots and coroots for $GSp_6$ can be described in the same coordinate systems used for $GSp_4$. In particular, the cocharacter lattice of $T$ is $\cochars{T} = \langle e_0, e_1, e_2, e_3 \rangle$, and we can express the coroots in terms of these generators:
\begin{align*}
\alpha_1^\vee &= (1, -1, 0, 0, 1, -1) = e_1 - e_2 & \alpha_6^\vee &= (1,1,0,0,-1,-1) = e_1 + e_2 \\
\alpha_2^\vee &= (0, 1, -1, -1, 1, 0) = e_2 - e_3 & \alpha_7^\vee &= (1,0,1,-1, -,-1) = e_1 + e_3 \\
\alpha_3^\vee &= (0,0,1,-1,0,0) = e_3 & \alpha_8^\vee &= (0, 1, 0, 0, -1, 0) = e_2 \\
\alpha_4^\vee &= (1, 0,-1,1, 0, -1) = e_1 - e_3 & \alpha_9^\vee &= (0, 1,1, -1, -1, 0) = e_2 + e_3 \\
\alpha_5^\vee &= (1,0,0,0,0,-1) = e_1 
\end{align*}
The simple reflections in $W$ are $s_1$, $s_2$ and $s_3$, corresponding to the roots $\alpha_1$, $\alpha_2$ and $\alpha_3$, respectively. We sometimes specify a reflection in terms of its corresponding positive coroot,
\begin{align*}
s_{\alpha_1} &= s_1 & s_{\alpha_4} &= s_{121}  & s_{\alpha_7} &= s_{31213} \\
s_{\alpha_2} &= s_2 & s_{\alpha_5} &= s_{12321} & s_{\alpha_8} &= s_{232} \\
s_{\alpha_3} &= s_3 & s_{\alpha_6} &= s_{2132312} & s_{\alpha_9} &= s_{323}
\end{align*}
The reflection ordering is:
$$
s_3 \prec s_{323} \prec s_{232} \prec s_{31213} \prec s_{2132312} \prec s_{12321} \prec s_2 \prec s_{121} \prec s_1.
$$
There are 79 $\mu$-admissible elements in this case, ranging in length from 0 to 6.

The data for this case is in Tables~\ref{table::gsp6-1100-coeff} and~\ref{table::gsp6-1100-adm}.

Consider $w=t_{(1,1,1,0,0,0)} s_{323123}$, the unique length-zero $\mu$-admissible element. The set $B_{\Phi}^{\prec} (w, t_{\lambda(w)})$ contains five paths, whose edge sets are
\begin{align*}
E(\Delta_1) &= \{s_{232}, s_{31213}\} \\
E(\Delta_2) &= \{s_3, s_{232}, s_{31213}, s_{12321}\} \\
E(\Delta_3) &= \{s_{323}, s_{31213}, s_{2132312}, s_{12321}\} \\
E(\Delta_4) &= \{s_3, s_{323}, s_{31213}, s_{2132312}\} \\
E(\Delta_5) &= \{s_3, s_{323}, s_{232}, s_{31213}, s_{2132312}, s_{12321}\}.
\end{align*}
The associated root systems for the first two paths are $J_{\Delta_1} = \{\pm \alpha_7\} \coprod \{\pm \alpha_8\}$, which is type $A_1 \times A_1$, and $J_{\Delta_2} = \{\pm \alpha_3, \pm \alpha_4, \pm \alpha_5, \pm \alpha_7\} \coprod \{ \pm \alpha_8 \}$, which has type $C_2 \times A_1$. The latter three associated root systems are $J_{\Delta_3} = J_{\Delta_4} = J_{\Delta_5} = \Phi$.

Next, we want to find $\chars{S_{w, J_{\Delta_i}}}$ to see if any torsion elements exist. Let us do the calculation for $\Delta_1$. We have
$$
\cochars{\torus}/\mathbb{Z}J_{\Delta_1}^\vee = \langle e_0, e_1, e_2, e_3 \rangle / \langle e_1 + e_3, e_2 \rangle = \langle \bar{e}_0, \bar{e}_1 \rangle.
$$
Since $\lambda(w) = \mu = e_0 + e_1 + e_2 + e_3 \in \cochars{\torus}$, it's image in the above quotient is $\bar{\mu} = \bar{e}_0$. This means $c_1 = 1$ in the notation of Lemma~\ref{lemma::rel-grp-gln-gsp2n}, so $\chars{S_{w, J_{\Delta_1}}} \cong \mathbb{Z}$ is torsion-free. A similar calculation for the other paths shows that $\chars{S_{w, J_{\Delta_i}}} = \{1\}$ for $i = 2,3,4,5$.

Here are the inclusions between the finite critical groups:
$$
A_{w, J_{\Delta_1}, k_F} \subset A_{w, J_{\Delta_2}, k_F} \subset A_{w, J_{\Delta_3}, k_F} = A_{w, J_{\Delta_4}, k_F} = A_{w, J_{\Delta_5}, k_F}.
$$
This is enough information to write down the coefficient $\phironeaug(I_r^+ s t_{(1,1,1,0,0,0)} s_{313123} I_r^+)$:
$$
\begin{cases}
0, & {\rm{if}}\ N_r (s) \notin A_{w,J_{\Delta_3}, k_F}, \\
2q^r + (1-q^r)^2, & {\rm{if}}\ N_r(s) \in A_{w, J_{\Delta_3}, k_F} \setminus A_{w, J_{\Delta_2}, k_F}, \\
3q^r + (1-q^r)^2, & {\rm{if}} N_r (s) \in A_{w, J_{\Delta_2}, k_F}, \\
(q-1)q^{2r} (1-q^r)^{-2} + 3q^r + (1-q^r)^2, & {\rm{if}}\ N_r (s) \in A_{w, J_{\Delta_1}, k_F}.
\end{cases}
$$

The $\mu$-admissible elements such that $\ell(w) = 3$ split into two cases (the exact lists are in Table~\ref{table::gsp6-1100-adm}). Consider $w = t_\mu s_{323}$ as an exemplar for the first case. There are two $\prec$-increasing paths $w \stackrel{\Delta_i}{\longrightarrow} t_{\lambda(w)}$; their edge sets are:
\begin{align*}
E(\Delta_1) &= \{ s_{323} \} \\
E(\Delta_2) &= \{ s_3, s_{323}, s_{232} \}.
\end{align*}
The associated root systems are $J_{\Delta_1} = \{ \pm \alpha_9 \}$ and $J_{\Delta_2} = \{\pm \alpha_2, \pm \alpha_3, \pm \alpha_8, \pm \alpha_9\}$, which has type $C_2$.

The usual method of calculation shows that
$$
\chars{S_{w, J_{\Delta_1}}} \cong \mathbb{Z} \times \mathbb{Z}
$$
and
$$
\chars{S_{w, J_{\Delta_2}}} \cong \mathbb{Z}.
$$
We also have $A_{w, J_{\Delta_1}, k_F} \subset A_{w, J_{\Delta_2}, k_F}$. So the coefficient in this case is,
$$
\begin{cases}
0, & {\rm{if}}\ N_r (s) \notin A_{w, J_{\Delta_2}, k_F}, \\
(q-1)(1-q^r)^{-1}, & {\rm{if}}\ N_r (s) \in A_{w, J_{\Delta_2}, k_F} \setminus A_{w, J_{\Delta_1}, k_F}, \\
(q-1)^2 q^r (1-q^r)^{-3} + (q-1)(1-q^r)^{-1}, & {\rm{if}}\ N_r (s) \in A_{w, J_{\Delta_1}, k_F}.
\end{cases}
$$

Now choose $w = t_{(1,0,0,1,1,0)} s_{123}$ as the representative for the second class of length-three $\mu$-admissible elements. There is only one $\prec$-increasing path in this case, whose edge set is
$E(\Delta_1) = \{s_3, s_2, s_1\}.$ These edges correspond to the three simple roots, so $J_{\Delta_1} = \Phi$. It follows that
$$
\phironeaug(I_r^+ sw I_r^+) = \begin{cases}
0, & {\rm{if}}\ N_r (s) \notin A_{w, J_{\Delta_1}, k_F}, \\
(1-q^r)^{-1}, & {\rm{if}}\ N_r (s) \in A_{w, J_{\Delta_1}, k_F}.
\end{cases}
$$


\appendix

\section{Tables of coefficient data}

This appendix presents the data needed to fully explain several cases of $\phironeaug$ for general linear groups and general symplectic groups:
\begin{enumerate}
\item $GL_4$, $\mu = (1, 1, 0, 0)$
\item $GL_5$, $\mu = (1, 1, 0, 0, 0)$
\item $GL_6$, $\mu = (1, 1, 0, 0, 0, 0)$
\item $GL_6$, $\mu = (1, 1, 1, 0, 0, 0)$
\item $GSp_4$, $\mu = (1, 1, 0, 0)$
\item $GSp_6$, $\mu = (1, 1, 1, 0, 0, 0)$
\end{enumerate}

In some tables, we will describe $\mu$-admissible elements; these have the form $w = t_\lambda \bar{w}$, where $\lambda$ is a conjugate of $\mu$ and $\bar{w}$ is an element of the finite Weyl group. The coordinates of the coweight $\lambda$ in $\cochars{T}$ are used. To save space, we write $\bar{w}$ as $s_{i_1 i_2 \ldots i_r}$ rather than $s_{i_1} s_{i_2} \cdots s_{i_r}$, where the $s_{i_k}$ are simple reflections in $W$.

See Section~\ref{section::worked-example-gl4-1100} for directions on using these tables.


\begin{table}[hp]
\caption{Coefficient data for $GL_4$, $\mu = (1,1,0,0)$}
\label{table::gl4-1100-coeff}
\renewcommand{\baselinestretch}{1}
\small\normalsize
\begin{center}
\begin{tabular}{|c|c|c|c||c|c|c|}
\hline
\multicolumn{7}{|c|}{\cellcolor[gray]{0.9} $\ell(w) = 0$} \\
\hline
Path & $\ell(\Delta)$ & ${\rm{rank}}(J_\Delta)$ & Isom. class of $S_{w,J_\Delta}$ & $A(\Delta)$ &$B(w, \Delta)$ & $C(\Delta)$ \\
\hline
$\Delta_1$ & 2 & 2 & $\mathbb{Z}$ & 1 & 1 & -2 \\
$\Delta_2$ & 4 & 3 & $\mathbb{Z}/2\mathbb{Z}$ & 0 & 0 & 0 \\
\hline
\multicolumn{7}{|c|}{$A_{w, \Delta_1, k_F} \subseteq A_{w, \Delta_2, k_F}$} \\
\hline
\hline
\multicolumn{7}{|c|}{\cellcolor[gray]{0.9} $\ell(w) = 1$} \\
\hline
Path & $\ell(\Delta)$ & ${\rm{rank}}(J_\Delta)$ & Isom. class of $S_{w,J_\Delta}$ & $A(\Delta)$ &$B(w, \Delta)$ & $C(\Delta)$ \\
\hline
$\Delta_1$ & 3 & 3 & $\mathbb{Z}/2\mathbb{Z}$ & 0 & 0 & -1 \\
\hline
\end{tabular}
\end{center}
\renewcommand{\baselinestretch}{2}
\small\normalsize
\end{table}

\begin{table}[hp]
\caption{Subsets of ${\rm{Adm}}(\mu)$ for $GL_4$, $\mu = (1,1,0,0)$}
\label{table::gl4-1100-adm}
\renewcommand{\baselinestretch}{1}
\small\normalsize
\begin{center}
\begin{tabular}{|p{5in}|}
\hline
\multicolumn{1}{|c|}{\cellcolor[gray]{0.9}$\ell(w) = 2$, $w \in X_1$} \\
\hline
$t_{(0, 1, 1, 0)} s_{32}$,
$t_{(1, 0, 1, 0)} s_{3121}$,
$t_{(1, 1, 0, 0)} s_{21}$,
$t_{(1, 1, 0, 0)} s_{2321}$ \\
\hline
\hline
\multicolumn{1}{|c|}{\cellcolor[gray]{0.9}$\ell(w) = 2$, $w \in X_2$} \\
\hline
$t_{(1, 0, 0, 1)} s_{12}$,
$t_{(1, 0, 1, 0)} s_{1232}$,
$t_{(1, 0, 1, 0)} s_{31}$,
$t_{(1, 1, 0, 0)} s_{123121}$,
$t_{(1, 1, 0, 0)} s_{1231}$,
$t_{(1, 1, 0, 0)} s_{23}$ \\
\hline
\end{tabular}
\end{center}
\renewcommand{\baselinestretch}{2}
\small\normalsize
\end{table}

\begin{table}[hp]
\caption{Coefficient data for $GL_5$, $\mu = (1,1,0,0,0)$}
\label{table::gl5-11000-coeff}
\renewcommand{\baselinestretch}{1}
\small\normalsize
\begin{center}
\begin{tabular}{|c|c|c|c||c|c|c|}
\hline
\multicolumn{7}{|c|}{\cellcolor[gray]{0.9} $\ell(w) = 0$} \\
\hline
Path & $\ell(\Delta)$ & ${\rm{rank}}(J_\Delta)$ & Isom. class of $S_{w,J_\Delta}$ & $A(\Delta)$ &$B(w, \Delta)$ & $C(\Delta)$ \\
\hline
$\Delta_1$ & 4 & 4 & $\mathbb{Z}/2\mathbb{Z}$ & 0 & 1 & -1 \\
$\Delta_2$ & 4 & 4 & $\mathbb{Z}/2\mathbb{Z}$ & 0 & 1 & -1 \\
$\Delta_3$ & 6 & 4 & $\mathbb{Z}/2\mathbb{Z}$ & 0 & 0 & 1 \\
\hline
\multicolumn{7}{|c|}{$A_{w, \Delta_1, k_F} = A_{w, \Delta_2, k_F} = A_{w, \Delta_3, k_F}$} \\
\hline
\hline
\multicolumn{7}{|c|}{\cellcolor[gray]{0.9}$\ell(w) = 1$} \\
\hline
Path & $\ell(\Delta)$ & ${\rm{rank}}(J_\Delta)$ & Isom. class of $S_{w,J_\Delta}$ & $A(\Delta)$ &$B(w, \Delta)$ & $C(\Delta)$ \\
\hline
\end{tabular}
\end{center}
\renewcommand{\baselinestretch}{2}
\small\normalsize
\end{table}

\begin{table}[hp]
\caption{Subsets of ${\rm{Adm}}(\mu)$ for $GL_5$, $\mu = (1,1,0,0,0)$}
\label{table::gl5-11000-adm}
\renewcommand{\baselinestretch}{1}
\small\normalsize
\begin{center}
\begin{tabular}{|p{5in}|}
\hline
\multicolumn{1}{|c|}{\cellcolor[gray]{0.9}$\ell(w) = 2$, $w \in X_1$} \\
\hline
$t_{(1, 0, 0, 1, 0)} s_{4123}$,
$t_{(1, 0, 1, 0, 0)} s_{123423}$,
$t_{(1, 0, 1, 0, 0)} s_{341232}$,
$t_{(1, 0, 1, 0, 0)} s_{3412}$,
$t_{(1, 1, 0, 0, 0)} s_{12341231}$,
$t_{(1, 1, 0, 0, 0)} s_{12341232}$,
$t_{(1, 1, 0, 0, 0)} s_{123412}$,
$t_{(1, 1, 0, 0, 0)} s_{234121}$,
$t_{(1, 1, 0, 0, 0)} s_{23412321}$,
$t_{(1, 1, 0, 0, 0)} s_{2341}$\\
\hline
\hline
\multicolumn{1}{|c|}{\cellcolor[gray]{0.9}$\ell(w)=2$, $w \in X_2$} \\
\hline
$t_{(0, 1, 1, 0, 0)} s_{3423}$,
$t_{(1, 0, 1, 0, 0)} s_{341231}$,
$t_{(1, 1, 0, 0, 0)} s_{2312}$,
$t_{(1, 1, 0, 0, 0)} s_{23412312}$,
$t_{(1, 1, 0, 0, 0)} s_{234312}$ \\
\hline
\hline
\multicolumn{1}{|c|}{\cellcolor[gray]{0.9}$\ell(w) = 3$, $w \in X_3$} \\
\hline
$t_{(0, 1, 0, 1, 0)} s_{423}$,
$t_{(0, 1, 1, 0, 0)} s_{23423}$,
$t_{(0, 1, 1, 0, 0)} s_{34232}$,
$t_{(0, 1, 1, 0, 0)} s_{342}$,
$t_{(1, 0, 0, 1, 0)} s_{41231}$,
$t_{(1, 0, 0, 1, 0)} s_{41232}$,
$t_{(1, 0, 1, 0, 0)} s_{1234231}$,
$t_{(1, 0, 1, 0, 0)} s_{312}$,
$t_{(1, 0, 1, 0, 0)} s_{34121}$,
$t_{(1, 0, 1, 0, 0)} s_{3412321}$,
$t_{(1, 0, 1, 0, 0)} s_{34312}$,
$t_{(1, 1, 0, 0, 0)} s_{12312}$,
$t_{(1, 1, 0, 0, 0)} s_{123412312}$,
$t_{(1, 1, 0, 0, 0)} s_{1234312}$,
$t_{(1, 1, 0, 0, 0)} s_{23121}$,
$t_{(1, 1, 0, 0, 0)} s_{231}$,
$t_{(1, 1, 0, 0, 0)} s_{234123121}$,
$t_{(1, 1, 0, 0, 0)} s_{23421}$,
$t_{(1, 1, 0, 0, 0)} s_{2343121}$,
$t_{(1, 1, 0, 0, 0)} s_{23431}$ \\
\hline
\hline
\multicolumn{1}{|c|}{\cellcolor[gray]{0.9}$\ell(w) = 3$, $w \in X_4$} \\
\hline
$t_{(1, 0, 0, 0, 1)} s_{123}$,
$t_{(1, 0, 0, 1, 0)} s_{12343}$,
$t_{(1, 0, 0, 1, 0)} s_{412}$,
$t_{(1, 0, 1, 0, 0)} s_{1234232}$,
$t_{(1, 0, 1, 0, 0)} s_{12342}$,
$t_{(1, 0, 1, 0, 0)} s_{341}$,
$t_{(1, 1, 0, 0, 0)} s_{1234121}$,
$t_{(1, 1, 0, 0, 0)} s_{123412321}$,
$t_{(1, 1, 0, 0, 0)} s_{12341}$,
$t_{(1, 1, 0, 0, 0)} s_{234}$ \\
\hline
\end{tabular}
\end{center}
\renewcommand{\baselinestretch}{2}
\small\normalsize
\end{table}

\begin{table}[hp]
\caption{Coefficient data for $GL_6$, $\mu = (1,1,0,0,0,0)$}
\label{table::gl6-110000-coeff}
\renewcommand{\baselinestretch}{1}
\small\normalsize
\begin{center}
\begin{tabular}{|c|c|c|c||c|c|c|}
\hline
\multicolumn{7}{|c|}{\cellcolor[gray]{0.9} $\ell(w) = 0$} \\
\hline
Path & $\ell(\Delta)$ & ${\rm{rank}}(J_\Delta)$ & Isom. class of $S_{w,J_\Delta}$ & $A(\Delta)$ &$B(w, \Delta)$ & $C(\Delta)$ \\
\hline
$\Delta_1$ & 4 & 4 & $\mathbb{Z}$ & 1 & 2 & -2 \\
$\Delta_2$ & 6 & 5 & $\mathbb{Z}/2\mathbb{Z}$ & 0 & 1 & 0 \\
$\Delta_3$ & 6 & 5 & $\mathbb{Z}/2\mathbb{Z}$ & 0 & 1 & 0 \\
$\Delta_4$ & 6 & 5 & $\mathbb{Z}/2\mathbb{Z}$ & 0 & 1 & 0 \\
$\Delta_5$ & 8 & 5 & $\mathbb{Z}/2\mathbb{Z}$ & 0 & 0 & 2 \\
\hline
\multicolumn{7}{|c|}{$A_{w, \Delta_1, k_F} \subset A_{w, \Delta_2, k_F} = A_{w, \Delta_3, k_F} = A_{w, \Delta_4, k_F} = A_{w, \Delta_5, k_F}$} \\
\hline
\end{tabular}
\end{center}
\renewcommand{\baselinestretch}{2}
\small\normalsize
Table continues on next page.
\end{table}

\begin{table}[hp]
\caption{Coefficient data for $GL_6$, $\mu = (1,1,0,0,0,0)$ (continued)}
\label{table::gl6-110000-coeff-cont}
\renewcommand{\baselinestretch}{1}
\small\normalsize
\begin{center}
\begin{tabular}{|c|c|c|c||c|c|c|}
\hline
\multicolumn{7}{|c|}{\cellcolor[gray]{0.9} $\ell(w) = 4, w \in X_5$} \\
\hline
Path & $\ell(\Delta)$ & ${\rm{rank}}(J_\Delta)$ & Isom. class of $S_{w,J_\Delta}$ & $A(\Delta)$ &$B(w, \Delta)$ & $C(\Delta)$ \\
\hline
$\Delta_1$ & 4 & 4 & $\mathbb{Z}$ & 1 & 0 & -2 \\
\hline
\hline
\multicolumn{7}{|c|}{\cellcolor[gray]{0.9} $\ell(w) = 4, w \in X_6$} \\
\hline
Path & $\ell(\Delta)$ & ${\rm{rank}}(J_\Delta)$ & Isom. class of $S_{w,J_\Delta}$ & $A(\Delta)$ &$B(w, \Delta)$ & $C(\Delta)$ \\
\hline
$\Delta_1$ & 2 & 2 & $\mathbb{Z} \times \mathbb{Z} \times \mathbb{Z}$ & 3 & 1 & -4 \\
$\Delta_2$ & 4 & 3 & $\mathbb{Z} \times \mathbb{Z} \times \mathbb{Z}/2\mathbb{Z}$ & 2 & 0 & -2 \\
\hline
\multicolumn{7}{|c|}{$A_{w, \Delta_1, k_F} \subset A_{w, \Delta_2, k_F}$} \\
\hline
\end{tabular}
\end{center}
\renewcommand{\baselinestretch}{2}
\small\normalsize
\end{table}

\begin{table}[hp]
\caption{Subsets of ${\rm{Adm}}(\mu)$ for $GL_6$, $\mu = (1,1,0,0,0,0)$}
\label{table::gl6-110000-adm}
\renewcommand{\baselinestretch}{1}
\small\normalsize
\begin{center}
\begin{tabular}{|p{5in}|}
\hline
\multicolumn{1}{|c|}{\cellcolor[gray]{0.9}$\ell(w) = 2$, $w \in X_1$} \\
\hline
$t_{(0, 1, 1, 0, 0, 0)} s_{345234}$,
$t_{(1, 0, 1, 0, 0, 0)} s_{34512341}$,
$t_{(1, 1, 0, 0, 0, 0)} s_{234123}$,
$t_{(1, 1, 0, 0, 0, 0)} s_{2345123412}$,
$t_{(1, 1, 0, 0, 0, 0)} s_{2345123423}$,
$t_{(1, 1, 0, 0, 0, 0)} s_{23454123}$ \\
\hline
\hline
\multicolumn{1}{|c|}{\cellcolor[gray]{0.9}$\ell(w) = 2$, $w \in X_2$} \\
\hline
$t_{(1, 0, 0, 1, 0, 0)} s_{451234}$,
$t_{(1, 0, 1, 0, 0, 0)} s_{12345234}$,
$t_{(1, 0, 1, 0, 0, 0)} s_{34512342}$,
$t_{(1, 0, 1, 0, 0, 0)} s_{34512343}$,
$t_{(1, 0, 1, 0, 0, 0)} s_{345123}$,
$t_{(1, 1, 0, 0, 0, 0)} s_{1234512341}$,
$t_{(1, 1, 0, 0, 0, 0)} s_{1234512342}$,
$t_{(1, 1, 0, 0, 0, 0)} s_{1234512343}$,
$t_{(1, 1, 0, 0, 0, 0)} s_{12345123}$,
$t_{(1, 1, 0, 0, 0, 0)} s_{23451231}$,
$t_{(1, 1, 0, 0, 0, 0)} s_{23451232}$,
$t_{(1, 1, 0, 0, 0, 0)} s_{2345123421}$,
$t_{(1, 1, 0, 0, 0, 0)} s_{2345123431}$,
$t_{(1, 1, 0, 0, 0, 0)} s_{2345123432}$,
$t_{(1, 1, 0, 0, 0, 0)} s_{234512}$ \\
\hline
\hline
\multicolumn{1}{|c|}{\cellcolor[gray]{0.9}$\ell(w) = 3$, $w \in X_3$} \\
\hline
$X_3$ contains the 30 length-three elements not contained in $X_4$. \\
\hline
\hline
\multicolumn{1}{|c|}{\cellcolor[gray]{0.9}$\ell(w) = 3$, $w \in X_4$} \\
\hline
\end{tabular}
\end{center}
\renewcommand{\baselinestretch}{2}
\small\normalsize
Table continues on next page.
\end{table}

\begin{table}[hp]
\caption{Subsets of ${\rm{Adm}}(\mu)$ for $GL_6$, $\mu = (1,1,0,0,0,0)$ (continued)}
\label{table::gl6-110000-adm-cont}
\renewcommand{\baselinestretch}{1}
\small\normalsize
\begin{center}
\begin{tabular}{|p{5in}|}
\hline
\multicolumn{1}{|c|}{\cellcolor[gray]{0.9}$\ell(w) = 5$, $w \in X_8$} \\
\hline
\end{tabular}
\end{center}
\renewcommand{\baselinestretch}{2}
\small\normalsize
\end{table}

\begin{table}[hp]
\caption{Coefficient data for $GL_6$, $\mu = (1,1,1,0,0,0)$}
\label{table::gl6-111000-coeff}
\renewcommand{\baselinestretch}{1}
\small\normalsize
\begin{center}
\begin{tabular}{|c|c|c|c||c|c|c|}
\hline
\multicolumn{7}{|c|}{\cellcolor[gray]{0.9} $\ell(w) = 0$} \\
\hline
Path & $\ell(\Delta)$ & ${\rm{rank}}(J_\Delta)$ & Isom. class of $S_{w,J_\Delta}$ & $A(\Delta)$ &$B(w, \Delta)$ & $C(\Delta)$ \\
\hline
$\Delta_1$ & 3 & 3 & $\mathbb{Z} \times \mathbb{Z}$ & 2 & 3 & -3 \\
$\Delta_2$ & 5 & 4 & $\mathbb{Z}$ & 1 & 2 & -1 \\
$\Delta_3$ & 5 & 4 & $\mathbb{Z}$ & 1 & 2 & -1 \\
$\Delta_4$ & 5 & 4 & $\mathbb{Z}$ & 1 & 2 & -1 \\
$\Delta_5$ & 7 & 5 & $\mathbb{Z}/3\mathbb{Z}$ & 0 & 1 & 1 \\
$\Delta_6$ & 7 & 5 & $\mathbb{Z}/3\mathbb{Z}$ & 0 & 1 & 1 \\
\hline
\end{tabular}
\end{center}
\renewcommand{\baselinestretch}{2}
\small\normalsize
Table continues on next page.
\end{table}

\begin{table}[hp]
\caption{Coefficient data for $GL_6$, $\mu = (1,1,1,0,0,0)$ (continued)}
\label{table::gl6-111000-coeff-cont}
\renewcommand{\baselinestretch}{1}
\small\normalsize
\begin{center}
\begin{tabular}{|c|c|c|c||c|c|c|}
\hline
\multicolumn{7}{|c|}{\cellcolor[gray]{0.9} $\ell(w) = 2, w \in X_2$} \\
\hline
Path & $\ell(\Delta)$ & ${\rm{rank}}(J_\Delta)$ & Isom. class of $S_{w,J_\Delta}$ & $A(\Delta)$ &$B(w, \Delta)$ & $C(\Delta)$ \\
\hline
$\Delta_1$ & 5 & 5 & $\mathbb{Z}/3\mathbb{Z}$ & 0 & 1 & -1 \\
$\Delta_2$ & 5 & 5 & $\mathbb{Z}/3\mathbb{Z}$ & 0 & 1 & -1 \\
$\Delta_3$ & 7 & 5 & $\mathbb{Z}/3\mathbb{Z}$ & 0 & 0 & 1 \\
\hline
\multicolumn{7}{|c|}{$A_{w,J_{\Delta_1}, k_F} = A_{w,J_{\Delta_2},k_F} = A_{w,J_{\Delta_3},k_F}$} \\
\hline
\hline
\multicolumn{7}{|c|}{\cellcolor[gray]{0.9} $\ell(w) = 3, w \in X_3$} \\
\hline
Path & $\ell(\Delta)$ & ${\rm{rank}}(J_\Delta)$ & Isom. class of $S_{w,J_\Delta}$ & $A(\Delta)$ &$B(w, \Delta)$ & $C(\Delta)$ \\
\hline
\end{tabular}
\end{center}
\renewcommand{\baselinestretch}{2}
\small\normalsize
Table continues on next page.
\end{table}

\begin{table}[hp]
\caption{Coefficient data for $GL_6$, $\mu = (1,1,1,0,0,0)$ (continued)}
\label{table::gl6-111000-coeff-cont2}
\renewcommand{\baselinestretch}{1}
\small\normalsize
\begin{center}
\begin{tabular}{|c|c|c|c||c|c|c|}
\hline
\multicolumn{7}{|c|}{\cellcolor[gray]{0.9} $\ell(w) = 5, w \in X_9$} \\
\hline
Path & $\ell(\Delta)$ & ${\rm{rank}}(J_\Delta)$ & Isom. class of $S_{w,J_\Delta}$ & $A(\Delta)$ &$B(w, \Delta)$ & $C(\Delta)$ \\
\hline
$\Delta_1$ & 4 & 4 & $\mathbb{Z}$ & 1 & 0 & -2 \\
\hline
\hline
\multicolumn{7}{|c|}{\cellcolor[gray]{0.9} $\ell(w) = 5, w \in X_{10}$} \\
\hline
Path & $\ell(\Delta)$ & ${\rm{rank}}(J_\Delta)$ & Isom. class of $S_{w,J_\Delta}$ & $A(\Delta)$ &$B(w, \Delta)$ & $C(\Delta)$ \\
\hline
$\Delta_1$ & 4 & 4 & $\mathbb{Z} \times \mathbb{Z}/3\mathbb{Z}$ & 1 & 0 & -2 \\
\hline
\hline
\multicolumn{7}{|c|}{\cellcolor[gray]{0.9} $\ell(w) = 5, w \in X_{11}$} \\
\hline
Path & $\ell(\Delta)$ & ${\rm{rank}}(J_\Delta)$ & Isom. class of $S_{w,J_\Delta}$ & $A(\Delta)$ &$B(w, \Delta)$ & $C(\Delta)$ \\
\hline
\end{tabular}
\end{center}
\renewcommand{\baselinestretch}{2}
\small\normalsize
\end{table}

\begin{table}[hp]
\caption{Subsets of ${\rm{Adm}}(\mu)$ for $GL_6$, $\mu = (1,1,1,0,0,0)$}
\label{table::gl6-111000-adm}
\renewcommand{\baselinestretch}{1}
\small\normalsize
\begin{center}
\begin{tabular}{|p{5in}|}
\hline
\multicolumn{1}{|c|}{\cellcolor[gray]{0.9}$\ell(w) = 2$, $w \in X_1$} \\
\hline
$t_{(1, 1, 0, 1, 0, 0)} s_{4523412}$,
$t_{(1, 1, 1, 0, 0, 0)} s_{23452341232}$,
$t_{(1, 1, 1, 0, 0, 0)} s_{34512341231}$ \\
\hline
\hline
\multicolumn{1}{|c|}{\cellcolor[gray]{0.9}$\ell(w) = 2$, $w \in X_2$} \\
\hline
$X_2$ contains the 18 length-two elements not contained in $X_1$. \\
\hline
\hline
\multicolumn{1}{|c|}{\cellcolor[gray]{0.9}$\ell(w) = 3$, $w \in X_3$} \\
\hline
$X_3$ contains the 44 length-three elements not contained in $X_4 \cup X_5$. \\
\hline
\hline
\multicolumn{1}{|c|}{\cellcolor[gray]{0.9}$\ell(w) = 3$, $w \in X_4$} \\
\hline
$t_{(1, 1, 0, 0, 0, 1)} s_{234123}$,
$t_{(1, 1, 0, 0, 1, 0)} s_{23454123}$,
$t_{(1, 1, 0, 1, 0, 0)} s_{2345123423}$,
$t_{(1, 1, 1, 0, 0, 0)} s_{2345123412}$,
$t_{(1, 1, 1, 0, 0, 0)} s_{34512341}$,
$t_{(1, 1, 1, 0, 0, 0)} s_{345234}$ \\
\hline
\end{tabular}
\end{center}
\renewcommand{\baselinestretch}{2}
\small\normalsize
Table continues on next page.
\end{table}

\begin{table}[hp]
\caption{Subsets of ${\rm{Adm}}(\mu)$ for $GL_6$, $\mu = (1,1,1,0,0,0)$ (continued)}
\label{table::gl6-111000-adm-cont}
\renewcommand{\baselinestretch}{1}
\small\normalsize
\begin{center}
\begin{tabular}{|p{5in}|}
\hline
\multicolumn{1}{|c|}{\cellcolor[gray]{0.9}$\ell(w) = 5$, $w \in X_9$} \\
\hline
$X_9$ contains the 90 length-five elements not contained in $X_{10} \cup X_{11}$. \\
\hline
\hline
\multicolumn{1}{|c|}{\cellcolor[gray]{0.9}$\ell(w) = 5$, $w \in X_{10}$} \\
\hline
\end{tabular}
\end{center}
\renewcommand{\baselinestretch}{2}
\small\normalsize
\end{table}


\begin{table}[hp]
\caption{Coefficient data for $GSp_4$, $\mu = (1,1,0,0)$}
\label{table::gsp4-1100-coeff}
\renewcommand{\baselinestretch}{1}
\small\normalsize
\begin{center}
\begin{tabular}{|c|c|c|c||c|c|c|}
\hline
\multicolumn{7}{|c|}{\cellcolor[gray]{0.9} $\ell(w) = 0$} \\
\hline
Path & $\ell(\Delta)$ & ${\rm{rank}}(J_\Delta)$ & Isom. class of $S_{w,J_\Delta}$ & $A(\Delta)$ &$B(w, \Delta)$ & $C(\Delta)$ \\
\hline
$\Delta_1$ & 1 & 1 & $\mathbb{Z}$ & 1 & 1 & -2 \\
$\Delta_2$ & 3 & 2 & $\{1\}$ & 0 & 0 & 0 \\
\hline
\multicolumn{7}{|c|}{$A_{w, J_{\Delta_1}, k_F} \subseteq A_{w, J_{\Delta_2}, k_F}$} \\
\hline
\hline
\multicolumn{7}{|c|}{\cellcolor[gray]{0.9} $\ell(w) = 1$} \\
\hline
Path & $\ell(\Delta)$ & ${\rm{rank}}(J_\Delta)$ & Isom. class of $S_{w,J_\Delta}$ & $A(\Delta)$ &$B(w, \Delta)$ & $C(\Delta)$ \\
\hline
$\Delta_1$ & 2 & 2 & $\{1\}$ & 0 & 0 & -1 \\
\hline
\hline
\multicolumn{7}{|c|}{\cellcolor[gray]{0.9} $\ell(w) = 2$} \\
\hline
Path & $\ell(\Delta)$ & ${\rm{rank}}(J_\Delta)$ & Isom. class of $S_{w,J_\Delta}$ & $A(\Delta)$ &$B(w, \Delta)$ & $C(\Delta)$ \\
\hline
$\Delta_1$ & 1 & 1 & $\mathbb{Z}$ & 1 & 0 & -2 \\
\hline
\hline
\multicolumn{7}{|c|}{\cellcolor[gray]{0.9} $\ell(w) = 3$} \\
\hline
Path & $\ell(\Delta)$ & ${\rm{rank}}(J_\Delta)$ & Isom. class of $S_{w,J_\Delta}$ & $A(\Delta)$ &$B(w, \Delta)$ & $C(\Delta)$ \\
\hline
$\emptyset$ & 0 & 0 & $\mathbb{Z} \times \mathbb{Z}$ & 2 & 0 & -3 \\
\hline
\end{tabular}
\end{center}
\renewcommand{\baselinestretch}{2}
\small\normalsize
\end{table}

%

\begin{table}[hp]
\caption{Coefficient data for $GSp_6$, $\mu = (1,1,1,0,0,0)$}
\label{table::gsp6-1100-coeff}
\renewcommand{\baselinestretch}{1}
\small\normalsize
\begin{center}
\begin{tabular}{|c|c|c|c||c|c|c|}
\hline
\multicolumn{7}{|c|}{\cellcolor[gray]{0.9} $\ell(w) = 0$} \\
\hline
Path & $\ell(\Delta)$ & ${\rm{rank}}(J_\Delta)$ & Isom. class of $S_{w,J_\Delta}$ & $A(\Delta)$ &$B(w, \Delta)$ & $C(\Delta)$ \\
\hline
$\Delta_1$ & 2 & 2 & $\mathbb{Z}$ & 1 & 2 & -2 \\
$\Delta_2$ & 4 & 3 & $\{1\}$ & 0 & 1 & 0 \\
$\Delta_3$ & 4 & 3 & $\{1\}$ & 0 & 1 & 0 \\
$\Delta_4$ & 4 & 3 & $\{1\}$ & 0 & 1 & 0 \\
$\Delta_5$ & 6 & 3 & $\{1\}$ & 0 & 0 & 2 \\
\hline
\multicolumn{7}{|c|}{$A_{w, J_{\Delta_1}, k_F} \subset A_{w, J_{\Delta_2}, k_F} \subset A_{w, J_{\Delta_3}, k_F} = A_{w, J_{\Delta_4}, k_F} = A_{w, J_{\Delta_5}, k_F}$} \\
\hline
\end{tabular}
\end{center}
\renewcommand{\baselinestretch}{2}
\small\normalsize
\end{table}

\begin{table}[hp]
\caption{Subsets of ${\rm{Adm}}(\mu)$ for $GSp_6$, $\mu = (1,1,1,0,0,0)$}
\label{table::gsp6-1100-adm}
\renewcommand{\baselinestretch}{1}
\small\normalsize
\begin{center}
\begin{tabular}{|p{5in}|}
\hline
\multicolumn{1}{|c|}{\cellcolor[gray]{0.9}$\ell(w) = 2$, $w \in X_1$} \\
\hline
$t_{(0,1,1,0,0,1)} s_{323}$,
$t_{(1,0,1,0,1,0)} s_{31231}$,
$t_{(1,1,0,1,0,0)} s_{2312312}$,
$t_{(1,1,1,0,0,0)} s_{2312312}$,
$t_{(1,1,1,0,0,0)} s_{31231}$,
$t_{(1,1,1,0,0,0)} s_{323}$ \\
\hline
\hline
\multicolumn{1}{|c|}{\cellcolor[gray]{0.9}$\ell(w) = 3$, $w \in X_2$} \\
\hline
$t_{(1,0,0,1,1,0)} s_{123}$,
$t_{(1,0,1,0,1,0)} s_{31232}$,
$t_{(1,0,1,0,1,0)} s_{312}$,
$t_{(1,1,0,1,0,0)} s_{12312}$,
$t_{(1,1,0,1,0,0)} s_{23121}$,
$t_{(1,1,0,1,0,0)} s_{2312321}$,
$t_{(1,1,0,1,0,0)} s_{231}$,
$t_{(1,1,1,0,0,0)} s_{3123121}$,
$t_{(1,1,1,0,0,0)} s_{321}$,
$t_{(1,1,1,0,0,0)} s_{323123121}$,
$t_{(1,1,1,0,0,0)} s_{32321}$ \\
\hline
\end{tabular}
\end{center}
\renewcommand{\baselinestretch}{2}
\small\normalsize
\end{table}

\bibliographystyle{amsplain}
\bibliography{thesis_bibliography}

\end{document}